\documentclass[reqno,10pt]{amsart}
\usepackage{xcolor}

\usepackage{amssymb}
\usepackage{comment}
\usepackage{bbm}
\usepackage{tikz}

\usepackage[all,cmtip]{xy}
\usepackage{hyperref}
\newcounter{mt}

\newtheorem{MainTheorem}[mt]{Theorem}

\newtheorem{Proposition}{Proposition}[section]
\newtheorem{Definition}[Proposition]{Definition}
\newtheorem{Lemma}[Proposition]{Lemma}
\newtheorem{Theorem}[Proposition]{Theorem}
\newtheorem{Corollary}[Proposition]{Corollary}
\newtheorem{Remark}[Proposition]{Remark}

\DeclareMathOperator{\Val}{Val}

\DeclareMathOperator{\nc}{nc}

\DeclareMathOperator{\Gr}{Gr}
\DeclareMathOperator{\Res}{Res}

\DeclareMathOperator{\glob}{glob}
\DeclareMathOperator{\sgn}{sgn}
\DeclareMathOperator{\id}{Id}
\renewcommand{\Re}{\mathrm {Re}}
\newcommand{\p}{\mathbb{P}}
\DeclareMathOperator{\vol}{vol}

\DeclareMathOperator{\Dens}{Dens}
\DeclareMathOperator{\Span}{Span}

\DeclareMathOperator{\Image}{Im}
\DeclareMathOperator{\Supp}{Supp}
\DeclareMathOperator{\Hom}{Hom}

\DeclareMathOperator{\Sym}{Sym}
\DeclareMathOperator{\Ker}{Ker}

\DeclareMathOperator{\WF}{WF}
\DeclareMathOperator{\LC}{LC}

\DeclareMathOperator{\OO}{O}
\DeclareMathOperator{\sign}{sign}

\DeclareMathOperator{\diag}{Diag}
\newcommand{\R}{\mathbb{R}}
\newcommand{\C}{\mathbb{C}}
\renewcommand{\i}{\mathbf{i}}

\makeatletter
\def\moverlay{\mathpalette\mov@rlay}
\def\mov@rlay#1#2{\leavevmode\vtop{%
		\baselineskip\z@skip \lineskiplimit-\maxdimen
		\ialign{\hfil$\m@th#1##$\hfil\cr#2\crcr}}}
\newcommand{\charfusion}[3][\mathord]{
	#1{\ifx#1\mathop\vphantom{#2}\fi
		\mathpalette\mov@rlay{#2\cr#3}
	}
	\ifx#1\mathop\expandafter\displaylimits\fi}
\makeatother

\newcommand{\cupdot}{\charfusion[\mathbin]{\cup}{\cdot}}

\newcommand{\largewedge}{\mbox{\Large $\wedge$}}

\newtheorem{note}{Remark} 
\def\note#1{\ifvmode\leavevmode\fi\vadjust{\vbox to0pt{\vss
 \hbox to 0pt{\hskip\hsize\hskip1em
\vbox{\hsize3.5cm\small\raggedright\pretolerance10000
 \noindent #1\hfill}\hss}\vbox to8pt{\vfil}\vss}}}

\begin{document}

\title{Curvature measures of pseudo-Riemannian manifolds}

\author{Andreas Bernig}
\author{Dmitry Faifman}
\author{Gil Solanes}

\email{bernig@math.uni-frankfurt.de}
\email{faifmand@tauex.tau.ac.il} 
\email{solanes@mat.uab.cat}
\address{Institut f\"ur Mathematik, Goethe-Universit\"at Frankfurt,
Robert-Mayer-Str. 10, 60629 Frankfurt, Germany}
\address{School of Mathematical Sciences, Tel Aviv University, Tel Aviv 6997801, Israel}
\address{Departament de Matem\`atiques, Universitat Aut\`onoma de Barcelona, 08193 Bellaterra, Spain, and 
Centre de Recerca Matem\`atica, Campus de Bellaterra, 08193 Bellaterra, Spain}


\begin{abstract}
The Weyl principle is extended from the Riemannian to the pseudo-Riemannian setting, and subsequently to manifolds equipped with generic symmetric $(0,2)$-tensors. More precisely, we construct a family of generalized curvature measures attached to such manifolds, extending the Riemannian Lipschitz-Killing curvature measures introduced by Federer. We then show that they behave naturally under isometric immersions, in particular they do not depend on the ambient signature. Consequently, we extend Theorema Egregium to surfaces equipped with a generic metric of changing signature, and more generally, establish the existence as distributions of intrinsically defined Lipschitz-Killing curvatures for such manifolds of arbitrary dimension. This includes in particular the scalar curvature and the Chern-Gauss-Bonnet integrand. Finally, we deduce a Chern-Gauss-Bonnet theorem for pseudo-Riemannian manifolds with generic boundary.
\end{abstract}

\thanks{{\it MSC classification}:  53C65, 
53C50 
\\ A.B. was supported
 by DFG grant BE 2484/5-2\\
 D.F. was partially supported by an NSERC Discovery Grant and ISF Grant 1750/20.
 \\
 G.S. was supported by the Serra H\'unter Programme and FEDER/MICINN grants MTM2015-66165-P, IEDI-2015-00634 and PGC2018-095998-B-I00}
\maketitle


\section{Introduction}

\subsection{The Weyl principle in the Riemannian case}

The Steiner formula from 1840 states that the volume of an $r$-tube around a compact convex body in Euclidean space is a polynomial in $r$. The (suitably normalized) coefficients are called intrinsic volumes. Their importance in convex and integral geometry stems from Hadwiger's theorem which characterizes the intrinsic volumes as the only (up to linear combinations) rigid motion invariant and continuous valuations on the space of compact convex bodies. 

Hermann Weyl proved in 1939 a version of Steiner's tube formula for compact submanifolds in Euclidean space \cite{weyl_tubes}. In this case, the volume of the $r$-tube is still a polynomial \emph{for small enough $r$}. The so-called \emph{Weyl principle} is the striking insight that the intrinsic volumes are expressible in terms of the inner (Riemannian) metric of the submanifold. More precisely, they can be written as integrals of certain polynomials in the curvature tensor called \emph{Lipschitz-Killing curvatures}. Intrinsic volumes are among the most fundamental Riemannian invariants, and include the Riemannian volume, the total scalar curvature, and the Euler characteristic. Allendoerfer-Weil \cite{allendoefer_weil} used Weyl's principle to give an extrinsic proof of the Chern-Gauss-Bonnet theorem; an intrinsic proof was later given by Chern \cite{chern44}.

Later, Federer unified Steiner's and Weyl's results by introducing intrinsic volumes for compact sets of positive reach, which include convex bodies and submanifolds of Euclidean space. Moreover, he showed that intrinsic volumes admit local versions, called curvature measures, which apply to regions of compact sets \cite{federer59}.

Using Nash's embedding theorem, it follows from Weyl's principle that one can associate to each Riemannian manifold $(M,g)$ a canonical family of curvature measures $\Lambda_k^M$ known as the Lipschitz-Killing curvature measures, which behave naturally under isometric immersions. More precisely, if $(M,g) \looparrowright (N,h)$ is an isometric immersion, then $\Lambda_k^N|_M=\Lambda_k^M$. Moreover, this property characterizes the Lipschitz-Killing curvature measures (up to linear combinations).  

We mention some more recent developments related to Weyl's principle. Alesker has developed a far-reaching theory of \emph{smooth valuations on manifolds} \cite{alesker_val_man1, alesker_val_man2, alesker_survey07, alesker_val_man4, alesker_bernig, alesker_val_man3}, which includes intrinsic volumes as fundamental examples. A prominent result in this theory is the existence of a natural product of such valuations. The space spanned by the intrinsic volumes is closed under Alesker's product of smooth valuations, and is called the \emph{Lipschitz-Killing algebra of $M$}. Any isometric immersion of Riemannian manifolds induces an algebra morphism of the corresponding Lipschitz-Killing algebras. Stated otherwise, there is a functor from the category of Riemannian manifolds and isometric immersions to the category of algebras. 

Based on the observation from \cite{bernig_fu_solanes} that curvature measures form a module over the algebra of smooth valuations, and using Cartan's calculus, Fu-Wannerer \cite{fu_wannerer} construct a module of Riemannian curvature measures over the Lipschitz-Killing algebra. They show that the structure constants in the module product are independent of the Riemannian metric. This allows a surprising transfer of Crofton-style integral-geometric formulas from easy Riemannian manifolds (e.g. round spheres) to more complicated ones (e.g. complex space forms). 

Let us complete this introduction on Weyl's principle by some applications and further developments. Donnelly \cite{donnelly} has shown that the intrinsic volumes of compact Riemannian manifolds are invariants of the spectrum of the Laplacian acting on differential forms. This shows that they are not only intrinsic invariants, but spectral invariants.  

In \cite{bernig_fu_solanes} the local kinematic formulas for isometry invariant curvature measures on complex space forms were determined, compare also \cite{bernig_fu_hig, bernig_fu_solanes_proceedings}. Tube formulas are special cases, with applications to Chern classes of complex analytic submanifolds in complex space forms \cite[Section 3]{bernig_fu_solanes}, \cite{gray_book}. The Fu-Wannerer transfer principle above may be used to simplify the determination of the kinematic formulas \cite{fu_wannerer}. 

Some conjectures on the behaviour of Lipschitz-Killing curvatures under Gromov-Hausdorff convergence with potential applications to the theory of Alexandrov spaces with curvature bounded from below are contained in the recent paper \cite{alesker_alexandrov_spaces}. Let us also mention \cite{faifman_contact}, where analogues of the Lipschitz-Killing curvature measures that are natural under embeddings have been constructed for contact manifolds.

In this work we address the natural problem of extending the Weyl principle to the pseudo-Riemannian setting. 

\subsection{Overview of the main results} 

By a pseudo-Riemannian manifold we understand a pair $(M,g)$, where $M$ is a smooth manifold and $g\in\Gamma^\infty(M, \Sym^2T^*M)$ a field of quadratic forms which is everywhere non-degenerate. We will often write $M^{p,q}$ to denote a pseudo-Riemannian manifold of signature $(p,q)$. The Lipschitz-Killing curvatures of a pseudo-Riemannian manifold are defined as in the Riemannian case, but the extension to submanifolds presents fundamental difficulties.

To date, only flat pseudo-Riemannian space has been considered from this perspective, where a Hadwiger-type classification and some Crofton-type formulas have been obtained \cite{alesker_faifman, bernig_faifman_opq, faifman_crofton}.

Although a version of Nash's embedding theorem in the pseudo-Riemannian case is available (see Theorem \ref{thm_nash}), a direct approach via tube formulas as in the Riemannian/Euclidean case is too restrictive (but compare \cite{willison09}). The reason is that without some rather strong assumptions, the tubes will not be compact. On the other hand, a natural substitute for the intrinsic volumes is provided by the isometry invariant generalized valuations on pseudo-Euclidean spaces from the recent work \cite{bernig_faifman_opq}. A local version of them - isometry invariant generalized curvature measures on pseudo-Euclidean spaces - will be constructed in this paper. 

Our main result is that these generalized curvature measures can be extended to pseudo-Riemannian manifolds and that they behave naturally with respect to isometric immersions (of arbitrary signatures). We show furthermore that those curvature measures naturally extend to generic symmetric $(0,2)$-tensor fields $g$ whose signature need not be constant.

 Such metrics have been studied by a few authors \cite{aguirre_fernandez_lafuente, honda_saji_teramoto,kriele_kossowski_a,kossowski97, kriele_kossowski_b,Pelletier,steller}, and appear quite naturally in certain settings. For instance, every compact hypersurface in Minkowski space $\R^{n,1}$ must be of changing signature.  A very general construction, with ties to affine differential geometry, optimization, K\"ahler manifolds and mirror symmetry, is that of \emph{Hessian-type metrics} \cite{hitchin, totaro}. 

More broadly, symmetric $(0,2)$-tensors without signature constraints appear in various settings, e.g. as the Ricci curvature tensor of a torsion-free affine connection preserving a volume form \cite{kobayashi}, such as a symplectic connection.
  In physics, signature changing metrics are in the heart of the \emph{no boundary proposal} of Hartle and Hawking \cite{hartle_hawking}, and more generally appear in cosmology \cite{sakharov} and quantum gravity \cite{white_weinfurtner_visser}, as well as in optical metamaterials \cite{smolyaninov_narimanov}.

At the same time, mathematical results in the changing signature setting are of quite limited scope, often exclusive to two dimensions, and imposing a long list of restrictions on the metric; results such as Gauss-Bonnet theorem are only applicable under further, non-generic restrictions. 

Our approach through valuation theory allows us to work in general dimension, with one simple generic restriction on the metric, as follows. 

\begin{Definition}
A metric of changing signature $g$ on a smooth manifold $X$ is \emph{LC (light-cone)-regular} if $0$ is a regular value of $g\in C^\infty(TX\setminus\underline{0})$.
\end{Definition} 
For some examples of LC-regular metrics, see section \ref{sec:lc_regular}.

Let us first state some corollaries of our results to metrics of changing signature. The first is Gauss' Theorema Egregium for LC-regular manifolds. For simplicity, let us state it for a surface embedded in the four-dimensional standard flat space $\R^{2,2}=(\R^4, Q=x_1^2+x_2^2-x_3^2-x_4^2)$. 
	
There is a natural complex valued distribution (generalized measure) on the associated oriented projectivization $S_Q^3=\mathbb P_+(\R^{2,2})$, denoted $\lambda_0\in\mathcal M^{-\infty}(S_Q^3, \mathbb C)$, which is $\OO(2,2)$-invariant. Identifying $S_Q^3$ with the Euclidean sphere $S^3$, 
\begin{displaymath}
\lambda_0=-\frac{1}{2\pi^2}\left(Q(x)^{-2}+\pi\i\delta_0'(Q(x))\right)\vol_{S^3}.
\end{displaymath}
One can show that $\Re\lambda_0$, $\Image\lambda_0$ span the space of such invariant distributions.  

\begin{MainTheorem}[Theorema Egregium] \label{thm:egregium}
Let $e:(X^2,g)\hookrightarrow \R^{2,2}$ be an isometric embedding of an LC-regular surface. Let $NX\subset \R^{2,2}\times S_Q^3$ be the normal bundle, $\pi:NX\to X$ and $\nu:NX\to S_Q^3$ the natural projections. 

Define the \emph{Gaussian curvature distribution} by $\kappa^X_0=\pi_*(\nu^*\lambda_0)\in\mathcal M^{-\infty}(X,\C)$.
Then $\kappa_0^X$ depends only on $(X,g)$, and restricts to $\frac{\i^q}{2\pi}KdA$ at the open subset of signature $(2-q,q)$, where $K$ is the sectional curvature and $dA$ the area element. 
\end{MainTheorem}

A similar statement holds in arbitrary dimensions and ambient signature. The LC-regularity condition guarantees that $\kappa^X_0$ is well-defined. If a local isometric embedding as a hypersurface is available, $\kappa_0^X$ can be defined as (twice) the pull-back by the Gauss map of $\lambda_0\in \mathcal M^{-\infty}(\mathbb P_+(\R^{p,q}),\C)^{\OO(p,q)}$, as in the classical setting.
	
The next result generalizes Theorem \ref{thm:egregium}, asserting the extendability of the Lipschitz-Killing curvatures to LC-regular manifolds as follows.
	
Consider the category $\mathbf{LCMet}$, whose objects are LC-regular manifolds without boundary, and morphisms are open isometric inclusions; and the category $\mathbf{GMsr}$ of pairs $(X,\mu)$, where $X$ is a manifold without boundary and $\mu\in\mathcal  M^{-\infty}(X,\C)$ a complex valued distribution, and whose morphisms are open inclusions $j:(X,\mu)\hookrightarrow (Y,\nu)$ such that $j^*\nu=\mu$.
	In the language of Atiyah-Bott-Patodi \cite{atiyah_bott_patodi}, a distribution-valued invariant of LC-regular metrics is any covariant functor $\Omega: \mathbf{LCMet}\to \mathbf{GMsr}$, such that $\Omega(X,g)=(X,\Omega^{X,g})$.
	
\begin{MainTheorem}[Lipschitz-Killing curvatures of LC-regular manifolds]\label{thm:distribution_curvatures}
There exist for all $k\geq 0$ natural distribution-valued $k$-homogeneous invariants $\kappa_k: \mathbf{LCMet}\to \mathbf{GMsr}$ such that, whenever $g$ is non-degenerate of signature $(p,q)$,
\begin{displaymath}
\kappa_k^{X,g}=\i^q \cdot \mathrm{LK}_k\cdot \vol_g,
\end{displaymath}
where $\mathrm{LK}_k$ is the classical $k$-th Lipschitz-Killing curvature.
\end{MainTheorem}

In particular  $\kappa_{0}^{X,g}$ (resp. $\kappa_{\dim X-2}^{X,g}$) is a multiple of the Chern-Gauss-Bonnet integrand, (resp. the scalar curvature) when $(X,g)$ is pseudo-Riemannian.
Naturally, if $X$ is a surface then $\kappa_0^{X,g}$ coincides with the Gaussian curvature from Theorem \ref{thm:egregium}. 
Note that on compact LC-regular manifolds, the above distributions can be integrated to give global Lipschitz-Killing invariants.

The last corollary of our main result is a Chern-Gauss-Bonnet theorem, which extends the Chern-Avez theorem on closed pseudo-Riemannian manifolds. We present here two notable special cases: closed manifolds of changing signature, and pseudo-Riemannian manifolds with generic boundary.

Pseudo-Riemannian manifolds with boundary have been previously considered, see \cite{birman_nomizu,gilkey_park15, honda_saji_teramoto}. Typically, non-generic restrictions are imposed on the boundary, e.g. the metric induced on the boundary is assumed non-degenerate, or its Gaussian curvature bounded. For instance, a generic smooth domain in $\R^{p,q}$, such as the Euclidean ball, violates both restrictions.

\begin{MainTheorem}[Chern-Gauss-Bonnet]\label{thm:CGB}\mbox{ }
\\\emph{i)} If $(X, g)$ is a closed LC-regular manifold, then
		\begin{displaymath}
		\chi(X)=\int_X \kappa_0^{X,g}.
		\end{displaymath}
\\\emph{ii)} If $(X, g)$ is a compact pseudo-Riemannian manifold with boundary, and $(\partial X, g)$ is LC-regular, then 
		\begin{displaymath}
		\chi(X)=\int_X \kappa_0^{X,g} + \int_{\partial X} {\lambda}^{\partial X,g},
		\end{displaymath}
		where ${\lambda}^{\partial X,g}\in\mathcal M^{-\infty}(\partial X, \mathbb C)$ generalizes geodesic curvature (see Section \ref{sec:gauss_bonnet}).
\end{MainTheorem}

Theorems \ref{thm:egregium}, \ref{thm:distribution_curvatures} and \ref{thm:CGB} are corollaries of the Weyl principle for pseudo-Riemannian manifolds and its extension to LC-regular manifolds, which we establish in the paper. To state these results, let us introduce some further notation. Let $\mathbf{\Psi Met}$ 
denote the category of pseudo-Riemannian manifolds with isometric immersions. Define  $\mathbf {GVal}$ to be the category where the objects are pairs $(M, \mu)$, with $M$ a smooth manifold, and $\mu \in \mathcal V^{-\infty}(M, \mathbb C)$ (the space of generalized valuations, see Subsection \ref{subsec_valuations}). The morphisms $e:(M, \mu_M)\to (N,\mu_N)$ are immersions $e:M \looparrowright N$ such that $e^*\mu_N$ is well-defined, and $\mu_M=e^*\mu_N$. 

Similarly, let $\mathbf{GCrv}$ be the category where the objects are pairs $(M, \Phi)$, with $M$ a smooth manifold, and $\Phi \in \mathcal C^{-\infty}(M, \mathbb C)$ (the space of generalized curvature measures, see Subsection \ref{subsec_curv_meas}). The morphisms $e:(M, \Phi_M)\to (N,\Phi_N)$ are immersions $e:M \looparrowright N$ such that $e^*\Phi_N$ is well-defined, and $\Phi_M=e^*\Phi_N$. The globalization gives rise to a functor $\glob:\mathbf {GCrv}\to \mathbf {GVal}$.

A generalized valuation valued invariant of pseudo-Riemannian metrics is a covariant functor $\mathbf{\Psi Met}\to  \mathbf {GVal}$ intertwining the forgetful functor to the category of smooth manifolds. Generalized curvature measure valued invariants are defined similarly. 

\begin{MainTheorem}[Weyl principle in the pseudo-Riemannian category] \label{mainthm_weyl}
There are generalized valuations (resp. curvature measures) valued invariants of pseudo-Riemannian metrics $\mu_k: M\mapsto \mu_k^M\in\mathcal V^{-\infty}(M,\C)$, resp. $\Lambda_k :M \mapsto \Lambda_k^M\in\mathcal C^{-\infty}(M,\C)$, such that $\mu_k=\glob \circ \Lambda_k$, for all integer $k\geq 0$.
Any generalized valuation valued invariant $\mu$ is given by a unique infinite linear combination $\mu=\sum_{k=0}^\infty a_k \mu_k+\sum_{k=1}^\infty b_k\bar \mu_k$.
\end{MainTheorem}

Let us comment on some aspects of this theorem.

The uniqueness statement holds also for curvature measures, see \cite{part2}, except the last sum starts at $k=0$. This is because $\mu_0$ is the real-valued Euler characteristic, while $\Lambda_0$ has a non-zero imaginary part. 

To recover the full array of invariants, we are forced to allow some singularities, hence the generalized, rather than smooth, curvature measures and valuations. These singularities pose the largest technical challenge that we have to overcome.

The functor $\Lambda_k$ is called \emph{the $k$-th Lipschitz-Killing curvature measure}, and $\mu_k$ is called \emph{the $k$-th intrinsic volume}. Note that both are complex valued, which simplifies the formulations of many statements.   
As $\Lambda_k$ and $\mu_k$ are generalized, they can only be applied or restricted to sets which are in general position with respect to the null-directions of the metric. We call such sets \emph{LC(light-cone)- transversal sets}, see Definition \ref{def:LC_transversality} for a precise description. 

At first glance, the pseudo-Riemannian Weyl principle seems straightforward to anticipate, given the Riemannian picture, at least up to signs and constants. However, a closer inspection reveals the results to be the outcome of an array of cancellations and coincidences that is far richer than that encountered by H. Weyl in \cite{weyl_tubes}.  One remarkable coincidence is the existence of a distinguished basis of Lipschitz-Killing curvatures, that restrict independently of the signature of the ambient metric. For another, the integral \eqref{eq:integral_definition} computed in Section \ref{sec_computation_integrals} that appears out of the geometry, must have the various parameters perfectly tuned to yield the recursive relation \eqref{eq:integral}, which is at the heart of the proof of the Weyl principle.
	
One advantage of the language of valuation theory is that it allows easy transition between manifolds of different dimension through restriction. This suggests a simultaneous treatment of all signatures, switching between different signatures using isometric immersions. However, while the Riemannian Weyl principle applies to arbitrary submanifolds, a typical submanifold in the pseudo-Riemannian setting inherits a symmetric field of $(0,2)$-tensors of non-constant signature. Thus it is natural to expand the class of admissible tensors to metrics of changing signature. Crucially, the LC-transversality of a submanifold turns out to coincide with the intrinsic property of \emph{LC-regularity} of the induced metric.

\begin{MainTheorem}[Weyl principle for LC-regular manifolds]
\label{main:lc_regular}
$(X,g)$ is LC-regular if and only if  $e(X)$ is LC-transversal for some (equivalently, any) isometric embedding $e: X\hookrightarrow M^{p,q}$.
The restriction $\Lambda_k^X:=e^*\Lambda_k^M$ is independent of $e$. 
\end{MainTheorem}

\begin{Proposition}[Basic properties  of $\Lambda_k,\mu_k$] \label{prop_basic_properties} \mbox{}
\\\emph{i)} $\Lambda_k$ extends the Riemannian Federer curvature measures; $\mu_k$ extends the intrinsic volumes.
\\\emph{ii)} $\mu_k^{X,g}$ and $\Lambda_k^{X, g}$ depend continuously on $g$ in the $C^\infty$ topology. 
\\\emph{iii)} Homogeneity:
\begin{displaymath}
\Lambda_k^{X,\lambda g}  =\left\{\begin{array}{cc} \sqrt{\lambda}^k \Lambda_k^{X,g}, & \lambda>0\\ \sqrt{\lambda}^k \overline{\Lambda}_k^{X,g}, & \lambda<0\end{array}\right. \qquad \mu_k^{X,\lambda g}  =\left\{\begin{array}{cc}  \sqrt{\lambda}^k \mu_k^{X,g}, & \lambda>0\\ \sqrt{\lambda}^k \overline{\mu}_k^{X,g}, & \lambda<0 \end{array}\right.
\end{displaymath}
\\\emph{iv)} $\Lambda_{\dim X}^{X,g}=\i^{q_x}\vol_{X,g}$, where $\i=\sqrt{-1}$ and $q_x$ is the negative index at $T_xX$.
\\\emph{v)} $\mu_0(X, g)=\chi$.
\end{Proposition}
The distribution-valued curvatures of Theorem \ref{thm:distribution_curvatures} are simply the interior terms of $\Lambda_k$.
The Chern-Gauss-Bonnet Theorem \ref{thm:CGB} follows from the last property. For a sharper continuity statement, see Corollary \ref{cor:LC_continuity}.

\subsection{Plan of the paper}
In Section \ref{sec:background} we recall the theory of valuations and curvature measures on manifolds, and introduce the central notion of generalized curvature measures. In Section \ref{sec_computation_integrals} we introduce a convenient language to treat homogeneous distributions, and compute a distributional integral that plays a key role in the proof of Theorem \ref{mainthm_weyl}.
In Section \ref{sec:lc_regular} we study LC-regular manifolds and LC-transversal submanifolds of pseudo-Riemannian manifolds, and show that the two notions are equivalent. 
In Section \ref{sec_construction} we construct the pseudo-Riemannian Lipschitz-Killing curvature measures. In Section \ref{sec_restriction} we compute the restriction of those curvature measures to pseudo-Riemannian submanifolds, establishing the existence part of Theorem \ref{mainthm_weyl}, which then combines with results of Section \ref{sec:lc_regular} to prove Theorem \ref{main:lc_regular}. In Section \ref{sec:properties}, some basic properties of the Lipschitz-Killing curvature measures are established, and the uniqueness part of Theorem \ref{mainthm_weyl} is proved. Finally, in Section \ref{sec_gauss_bonnet} we generalize the Chern-Gauss-Bonnet theorem using the Weyl principle.

\subsection{Acknowledgements}
We would like to thank Bo'az Klartag for his insightful input on distributional integrals, and the referee for numerous valuable comments. This work was partially done  during D.F.'s term at the University of Toronto as Coxeter Assistant Professor, as well as a CRM-ISM postdoctoral fellow in Montreal, which we gratefully acknowledge.   
\section{Background. Generalized valuations and curvature measures}\label{sec:background}

\subsection{Some terminology}

Let $M$ be a smooth manifold.  We will always assume that $M$ is oriented and connected, although all statements can be adjusted to the general case. All manifolds are without boundary, unless indicated otherwise.

A \emph{pseudo-Riemannian metric} on $M$ is a smooth field $Q$ of non-degenerate quadratic forms. Since $M$ is connected, the signature of these quadratic forms is constant and will be denoted by $(p,q)$. The notation $M^{p,q}$ will mean that $M$ is equipped with some pseudo-Riemannian metric of that signature. The simplest pseudo-Riemannian manifold is $\R^{p,q}$, which is $\R^{p+q}$ endowed with the flat metric $Q=dx_1^2+\cdots+dx_p^2-dx_{p+1}^2-\cdots-dx_{p+q}^2$. We will refer to $Q$ and the usual Euclidean inner product $P$ of $\R^{p+q}$ as the standard bilinear forms of $\R^{p,q}$.

We will make frequent use of the following version of Nash's embedding theorem.  
\begin{Theorem}[Pseudo-Riemannian Nash embedding theorem, {\cite{clarke70}}] \label{thm_nash}
 Any pseudo-Riemannian manifold $M^{p,q}$ admits an isometric embedding $e:M\hookrightarrow \R^{p',q'}$.
\end{Theorem}

A \emph{generalized measure} on $M$ is an element of $\mathcal M^{-\infty}(M):=C^\infty_c(M)^*$, while a \emph{generalized function} is an element of $C^{-\infty}(M):= \mathcal{M}^\infty_c(M)^*$ (where $C_c^\infty$ and $\mathcal M_c^\infty$ are the spaces of complex valued compactly supported smooth functions, resp. measures). Put differently, a generalized measure is a generalized section of the complex line bundle $\Dens(TM)$, where $\Dens(V)$ is the space of complex valued Lebesgue measures on a real vector space $V$.


A \emph{generalized $k$-form}, or $(n-k)$-current, is an element of $\Omega^k_{-\infty}(M):=\Omega_c^{n-k}(M)^*$, which again we take complex valued. Note that $\Omega^k(M) \hookrightarrow \Omega^k_{-\infty}(M)$. The exterior differential extends as $d:\Omega^k_{-\infty}(M) \to \Omega^{k+1}_{-\infty}(M)$. When $\partial M\neq\emptyset$, such currents are called \emph{supported}.  Write $\Omega(M)=\bigoplus_{k=0}^n \Omega^k(M)$, $\Omega_{-\infty}(M)=\bigoplus_{k=0}^n \Omega_{-\infty}^k(M)$.


The \emph{wave front set} of a generalized $k$-form $\omega$ is a closed conical subset $\WF(\omega) \subset T^*M \setminus \underline{0}$. If $\WF(\omega)=\emptyset$, then $\omega$ is a smooth form. The space of all generalized $k$-forms $\omega$ with $\WF(\omega) \subset \Gamma$ for some fixed closed conical set $\Gamma \subset T^*M \setminus \underline{0}$ is denoted by $\Omega_{-\infty,\Gamma}^k(M)$. We refer to \cite{hoermander_pde1} for the definition and more details.

Our main results do not depend on a particular choice of topology. To get good continuity properties for the various curvature measures that we construct, we use the \emph{normal topology} \cite{brouder_dang_helein, dabrowski_brouder} on $\Omega_{-\infty,\Gamma}^k(M)$ instead of the more common H\"ormander topology, as it renders the operations of pull-back and push-forward continuous, rather than merely sequentially continuous. The compactly supported smooth $k$-forms are sequentially dense in both topologies. For $\Gamma=T^*M\setminus\underline 0$, the normal topology on $\Omega_{-\infty,\Gamma}^k(M)=\Omega_{-\infty}^k(M)$ is the strong dual topology.

The \emph{wedge product} admits a partial extension to generalized forms. More precisely, let $a:T^*M \to T^*M, (x,\xi) \mapsto (x,-\xi)$ be the antipodal map. If $\Gamma_1 \cap a\Gamma_2=\emptyset$, then the wedge product extends as a jointly sequentially continuous and hypocontinuous, hence separately continuous bilinear map
$
\Omega^{k_1}_{-\infty,\Gamma_1}(M) \times \Omega^{k_2}_{-\infty,\Gamma_2}(M) \to \Omega^{k_1+k_2}_{-\infty}(M).
$
Recall that \emph{hypocontinuity} means that for every neighborhood $W$ of $0$ in $\Omega^{k_1+k_2}_{-\infty}(M)$ and every bounded subset $B$ of $\Omega^{k_2}_{-\infty,\Gamma_2}(M)$, there exists a neighborhood $U$ of $0$ in $\Omega^{k_1}_{-\infty,\Gamma_1}(M)$ such that the image of $U \times B$ is contained in $W$; and similarly with the roles of the factors interchanged.

Given a vector space $V$, its \emph{oriented projectivization} is ${\mathbb P_+(V)=(V\setminus \{0\})/\R_+}$. More generally, a vector bundle $E\to M$ defines the fiber bundle $\mathbb P_+(E)\to M$.
The \emph{cosphere bundle} of $M$ is the sphere bundle $\p_M= \mathbb P_+(T^*M)$. Its total space is a $(2n-1)$-dimensional manifold which carries a canonic contact structure.

A form $\omega\in\Omega(\mathbb P_M)$ is \emph{vertical} if it vanishes on the contact hyperplanes. A generalized form $\eta\in \Omega_{-\infty}(\mathbb P_M)$ is \emph{vertical} if $\langle \eta, \omega\rangle=0$ for any vertical  $\omega\in \Omega_c(\mathbb P_M)$.
If $e: X \looparrowright M$ is an immersion, then 
\begin{displaymath}
N^*X:=\{(e(x),[\xi]) \in \p_M: x \in X,de|_x^*(\xi)=0\}
\end{displaymath}
is called the \emph{conormal bundle} of $X$ in $M$. 
We will use the same notation for its lift to $T^*M$ (with or without the zero section) when no confusion can arise. When $M$ has to be specified, we write $N^*X=T^*_XM$.

\subsection{Valuations}
\label{subsec_valuations}

A valuation on an $n$-dimensional vector space $V$ is a finitely additive functional $\mu$ on the space of convex bodies. Finite additivity means that 
\begin{displaymath}
\mu(K \cap L)+\mu(K \cup L)=\mu(K)+\mu(L),
\end{displaymath}
whenever $K,L,K \cup L$ are compact convex sets. Continuity of valuations is with respect to the Hausdorff distance on the space of convex bodies. Examples of valuations are the Euler characteristic ($\mu(K)=1$), the Lebesgue measure, the intrinsic volumes or more generally arbitrary mixed volumes.

The space $\Val$ of translation invariant continuous valuations is a Banach space on which the general group $\mathrm{GL}(V)$ acts. The subspace of $k$-homogeneous valuations of parity $\epsilon \in\{\pm\}$ is denoted by 
\begin{displaymath}
\Val_k^\epsilon=\{\mu \in \Val| \mu(tK)=t^k\mu(K), t>0, \mu(-K)=\epsilon \mu(K)\}.
\end{displaymath}
Alesker's irreducibility theorem states that the $\mathrm{GL}(V)$-modules $\Val_k^\epsilon$ are irreducible.

The dense subspace of smooth vectors is denoted by $\Val^{\infty}$ and its elements are called \emph{smooth translation invariant valuations}.   

Alesker proved that a valuation $\mu$ is smooth if and only if there are translation invariant differential forms $\phi \in \Omega^n(V), \omega \in \Omega^{n-1}(\p_V)$ such that 
\begin{equation} \label{eq_def_smoothness}
\mu(K)=\int_K \phi+\int_{\nc(K)} \omega.
\end{equation}
Here $\nc(K)$ is the conormal cycle of $K$ in the cosphere bundle $\p_V$ of $V$.

 This observation can be utilized to extend valuation theory to general smooth manifolds, resulting in a deep theory developed by Alesker, Fu and others. In this theory, convex bodies are replaced by some family of test bodies, such as the family  $\mathcal P(M)$ of compact differentiable polyhedra. A smooth valuation on $M$ is a functional $\mu: \mathcal P(M)\to \C$ given by some forms $\phi \in \Omega^n(M)$ and $\omega \in \Omega^{n-1}(\p_M)$, 
\begin{equation} \label{eq_def_smoothness_manifold}
\mu(A)=\int_A \phi+\int_{\nc(A)} \omega.
\end{equation}
Here $\nc(A)$ is the conormal cycle of $A \in \mathcal{P}(M)$, which is a Legendrian cycle in the cosphere bundle $\p_M$ of $M$.  We will also write $\mu=[[\phi,\omega]]$. The space of complex valued smooth valuations on $M$ is denoted by $\mathcal{V}^\infty(M)$. It admits a natural filtration
\begin{equation} \label{eq_filtration_vals}
\mathcal V^\infty(M) = \mathcal W_0^\infty(M) \supset \mathcal W_1^\infty(M) \supset \cdots \supset \mathcal W_{\dim M}^\infty(M) = \mathcal M^\infty(M),
\end{equation}
and the \emph{Euler-Verdier involution} given by $\sigma([[\phi,\omega]])=(-1)^n[[\phi,a^*\omega]]$ where $a\colon \p_M\to\p_M$ is the fiberwise antipodal map.

The differential forms defining $\mu$ are not unique. By \cite{bernig_broecker07}, one has $[[\phi,\omega]]=0$ if and only if $D\omega+\pi^*\phi=0,\pi_*\omega=0$. Here $\pi:\p_M \to M$ is the projection and $\pi^*$ (resp. $\pi_*$) denotes the pull-back (resp. push-forward) of differential forms; and $D$ is the Rumin differential operator \cite{bernig_broecker07,rumin94}. 

Alesker and Fu \cite{alesker_val_man3} (see also \cite{alesker_bernig} and \cite{fu_alesker_product}) have introduced a product structure on the space $\mathcal{V}^{\infty}(M)$ of smooth valuations on a manifold $M$, which is compatible with the filtration \eqref{eq_filtration_vals}. It has led to several deep applications in the integral geometry of isotropic spaces \cite{alesker_intgeo, alesker_bernig_convolution, bernig_fu_solanes,solanes, solanes_wannerer, wannerer_angular}. The product satisfies a version of Poincar\'e duality, which gives rise to the notion of \emph{generalized valuations} on a manifold. Generalized valuations appear quite naturally in Hadwiger-type theorems for non-compact groups such as the Lorentz group \cite{alesker_faifman, bernig_faifman_opq, faifman_contact}.

The space $\mathcal{V}_c^\infty(M)$ of compactly supported smooth valuations on $M$ has a natural topology, and  the space of generalized valuations on $M$ is defined as $\mathcal{V}^{-\infty}(M):=\mathcal{V}^\infty_c(M)^*$. A generalized valuation $\psi$ can be described by two generalized forms $\zeta\in C^{-\infty}(M), \tau\in\Omega_{-\infty}^n(\mathbb P_M)$ such that 
\begin{displaymath}
\langle \psi, [[\phi,\omega]]\rangle=\langle \zeta,\phi\rangle+\langle \tau,\omega\rangle,
\end{displaymath}
where $\phi \in \Omega^n_c(M), \omega \in \Omega^{n-1}_c(\p_M)$.
The form $\tau$ must be closed and vertical, and $\pi_*\tau=d \zeta$. Any such pair defines a generalized valuation, denoted $[(\zeta,\tau)]\in \mathcal V^{-\infty}(M)$, and this correspondence is one-to-one. 
There is a natural inclusion of the (not necessarily compact) differentiable polyhedra in $\mathcal V^{-\infty}(M)$, $A \mapsto \chi_A:= [(\mathbbm 1_A,  [[N^*A]]\hspace{1pt})]$. 

By \cite{alesker_bernig}, $[[\phi,\omega]]\in\mathcal V^\infty(M)$ can be identified with the generalized valuation 
\begin{equation}\label{eq:valuation_current}
[(\zeta,\tau)]=[(\pi_*\omega,a^*(D\omega+\pi^*\phi))].
\end{equation} 
Moreover, $[(\zeta,\tau)]\in\mathcal V^{-\infty}(M)$ is smooth if and only if both $\zeta$ and $\tau$ are smooth.

The \emph{wave front set} of a generalized valuation $\psi$ is the pair $(\WF(\zeta),\WF(\tau))\subset  (T^*M\setminus \underline{0}) \times ( T^*\p_M \setminus \underline{0})$. Given closed conical sets $\Lambda \subset T^*M \setminus \underline{0}, \Gamma \subset T^*\p_M \setminus \underline{0}$, the space of generalized valuations with $\WF(\zeta) \subset \Lambda, \WF(\tau) \subset \Gamma$ is denoted by $\mathcal{V}^{-\infty}_{\Lambda,\Gamma}(M)$. It is equipped with the normal topology, inherited from the corresponding spaces of distributions. Under the additional assumption 
$\pi^*\Lambda \subset \Gamma, \ \pi^*\pi_*\Gamma \subset \Gamma$, the space $\mathcal{V}^\infty(M)$ is sequentially dense in $\mathcal{V}^{-\infty}_{\Lambda,\Gamma}(M)$ \cite{alesker_bernig_convolution}. Here 
\begin{align}
\pi_*\Gamma & =\{(x,\tau) \in T^*M\setminus \underline{0}: \exists \xi \in \mathbb P_M|_x, \quad d\pi|_{(x,[\xi])}^*\tau \in \Gamma\}, \label{eq_push_forward_wf}\\
\pi^*\Lambda & =\{(x,[\xi],\eta) \in T^*\p_M \setminus \underline{0}: \exists \tau \in \Lambda_x, \eta=d\pi|_{(x,[\xi])}^*\tau \}.  \label{eq_pull_back_wf}
\end{align} 

The filtration \eqref{eq_filtration_vals} can be extended to a natural filtration of $\mathcal{V}^{-\infty}(M)$:
\begin{displaymath} 
\mathcal{V}^{-\infty}(M) = \mathcal W_0^{-\infty}(M) \supset\mathcal W_1^{-\infty}(M) \supset \cdots \supset \mathcal W_{\dim M}^{-\infty}(M)=\mathcal M^{-\infty}(M).
\end{displaymath}

The Euler-Verdier involution extends  to $\mathcal V^{-\infty}(M)$  
 by $\sigma[(\zeta,\tau)]=[(\zeta,(-1)^n a^*\tau)]$.

\subsection{Curvature measures}
\label{subsec_curv_meas}


Let $M$ be a smooth manifold of dimension $n$. A functional $\Phi:\mathcal P(M)\to \mathcal M(M)$ is called a \emph{smooth curvature measure} if for all Borel subsets $B\subset M$
\begin{displaymath}
\Phi(A,B)=\int_{A \cap B} \phi+\int_{\nc(A) \cap \pi^{-1}B} \omega,\qquad A\in\mathcal P(M).
\end{displaymath}
We will also write $\Phi=[\phi,\omega]$. The space of smooth curvature measures is denoted by $\mathcal{C}(M)$. Fundamental examples of smooth curvature measures on $\R^n$ are Federer's curvature measures \cite{federer59, zaehle86}, which are moreover rigid motion covariant.

As observed in \cite{solanes_wannerer}, the natural filtration on the space of $(n-1)$-forms on the sphere bundle $\p_M \to M$ induces a filtration 
\begin{displaymath}
\mathcal{C}(M)=\mathcal{C}_0(M) \supset \mathcal{C}_1(M) \supset \ldots \supset \mathcal{C}_n(M)=\mathcal M^\infty(M), 
\end{displaymath}
where the last identification maps $\mu \in \mathcal M^\infty(M)$ to $\Phi(A,B)=\mu(A \cap B)$.

Given a smooth function $f \in C^\infty(M)$, we define a smooth valuation $\Phi_f$ by 
\begin{displaymath}
\Phi_f(A):=\int_M f d\Phi(A,\cdot). 
\end{displaymath} 
In particular, if $f \equiv 1$, we get the \emph{globalization map} 
\begin{displaymath}
\glob:\mathcal{C}(M) \to \mathcal{V}^\infty(M),\quad \Phi \mapsto \Phi_1.
\end{displaymath}

Next we introduce the new notion of generalized curvature measure by replacing the test bodies $\mathcal P(M)$ with $\mathcal V^\infty_c(M)$, as in the definition of generalized valuations.

\begin{Definition}
The space $\mathcal{C}^{-\infty}(M)$ of \emph{generalized curvature measures} consists of functionals $\Phi:\mathcal V^\infty_c(M)\to \mathcal M^{-\infty}(M)$ such that there are $ \phi \in \Omega^n_{-\infty}(M), \omega \in \Omega^{n-1}_{-\infty}(\p_M)$ with 
\begin{displaymath}
\Phi([(\zeta,\tau)],f)=\langle \phi, f\cdot \zeta\rangle+ \langle \omega, \pi^*f\cdot \tau\rangle 
\end{displaymath}
for all $[(\zeta,\tau)] \in \mathcal V^\infty_c(M),f\in C^\infty_c(M)$.
\end{Definition}
We write $\Phi=[\phi,\omega]$. Note that $\Phi=0$ if  $\phi=0$ and $\omega=\eta+d\eta'$ with $\eta, \eta'$ vertical. 

The natural filtration on the space of generalized $(n-1)$-forms on the sphere bundle $\p_M \to M$ induces a filtration
\begin{displaymath}
\mathcal{C}^{-\infty}(M)=\mathcal{C}^{-\infty}_0(M) \supset \mathcal{C}_1^{-\infty}(M) \supset \cdots \supset \mathcal{C}^{-\infty}_n(M)=\mathcal M^{-\infty}(M).
\end{displaymath}

\begin{Definition}
For $\Phi \in \mathcal{C}^{-\infty}(M)$ and $f \in C^\infty(M)$, we define $\Phi_f \in \mathcal{V}^{-\infty}(M)$ by 
\begin{displaymath}
\Phi_f(\mu):=\Phi(\mu,f), \quad  \mu \in \mathcal{V}^\infty_c(M).
\end{displaymath}
In particular, we obtain a \emph{globalization map} 
\begin{displaymath}
\glob:\mathcal{C}^{-\infty}(M) \to \mathcal{V}^{-\infty}(M),\qquad  \Phi \mapsto \Phi_1,
\end{displaymath}
and we write $\glob[\phi,\omega]=[[\phi,\omega]]$.
\end{Definition}

In terms of generalized forms,  
\begin{displaymath}
[\phi,\omega]_f=[(f \cdot \pi_*\omega, a^*(D(\pi^*f \cdot \omega)+\pi^*(f\phi)))].
\end{displaymath}

Here the various operations (Rumin differential, push-forward and pull-back) are extended from smooth forms to generalized forms by duality in the usual way.

\begin{Proposition}
The globalization map $\glob:\mathcal{C}^{-\infty}(M) \to \mathcal{V}^{-\infty}(M)$ is surjective.
\end{Proposition}

\proof
It is well-known \cite[eq. 6.38]{melrose93} that for every smooth manifold $X$, the inclusion $\Omega(X)\hookrightarrow\Omega_{-\infty}(X)$ induces an isomorphism between the corresponding de Rham cohomologies of $X$.

Let a generalized valuation be given by a pair of generalized forms $\zeta \in C^{-\infty}(M)$, $\tau \in \Omega^n_{-\infty}(\p_M)$, with $\tau$ vertical and closed and $\pi_*\tau=d\zeta$. Let $[\tau]\in H^n(\p_M)$ be the cohomology class of $\tau$. A part of the exact Gysin sequence \cite[Proposition 4.13]{bott_tu} is 
\begin{displaymath}
H^n(M) \stackrel{\pi^*}{\to} H^n(\p_M) \stackrel{\pi_*}{\to} H^1(M).
\end{displaymath} 

By our assumptions we have $\pi_*[\tau]=[\pi_*\tau]=[d\zeta]=0$, hence there exists some class $[\phi] \in H^n(M)$ with $\pi^*[\phi]=[\tau]$, which means $\tau=d\omega+\pi^*\phi$ for some $\omega \in \Omega^{n-1}_{-\infty}(\p_M), \phi \in \Omega^n_{-\infty}(M)$. Since $\tau$ is vertical, we have $d\omega=\tau-\pi^*\phi=D\omega$. Moreover, $\pi_*\pi^*\phi=0$, hence $d\pi_*\omega=\pi_*\tau=d\zeta\Rightarrow \pi_*\omega=\zeta+\lambda$ for some $\lambda \in \R$. 

Let $(\phi_c,\omega_c)$ be a pair of smooth forms with $D\omega_c+\pi^*\phi_c=0, \pi_*\omega_c=1$. Such forms were constructed by Chern \cite{chern44}. Then $D(\omega-\lambda \omega_c)+\pi^*(\phi-\lambda \phi_c)=\tau$ and $\pi_*(\omega-\lambda \omega_c)=\zeta$, hence the generalized curvature measure $[(\phi-\lambda \phi_c,\omega-\lambda \omega_c)]$ maps under $\glob$ to the given generalized valuation.  
\endproof

Fix a pair of closed cones $\Lambda\subset T^*M\setminus \underline{0}$, $\Gamma\subset T^*\mathbb P_M \setminus \underline{0}$. Define $\mathcal C_{\Lambda,\Gamma}^{-\infty}(M)$ as the space of curvature measures that can be represented by a pair $(\phi,\omega)\in \Omega^n_{-\infty,\Lambda}(M)\times \Omega^{n-1}_{-\infty,\Gamma}(\mathbb P_M)$. As differential operators do not increase the wave front set, the globalization map acts from $\mathcal C_{\Lambda,\Gamma}^{-\infty}(M)$ to $\mathcal V_{\pi_*\Gamma,-\Gamma\cup\pi^*\Lambda}^{-\infty}(M)$. We topologize $\mathcal C_{\Lambda,\Gamma}^{-\infty}(M)$ by quotienting the normal topology on $ \Omega^n_{-\infty,\Lambda}(M)\times \Omega^{n-1}_{-\infty,\Gamma}(\mathbb P_M)$.

Next, we find a subset of $\mathcal P(M)$ on which generalized curvature measures can be evaluated. Any $A\in\mathcal P(M)$ admits a canonical decomposition as a finite union of smooth, locally closed submanifolds called the smooth strata of $A$.
\begin{Proposition}\label{prop:curvature_measures_at_polyhedra}
Let $ \Lambda\subset T^*M \setminus \underline 0$ and $\Gamma \subset T^*\mathbb P_M \setminus \underline 0$ be closed conical sets. Assume $A \in \mathcal P(M)$ is a differentiable polyhedron, such that each smooth stratum $ A_i$ of $A$ has $N^*A_i \cap  \Lambda=\emptyset$, and each smooth stratum $Z_j$ of $\nc(A)$ satisfies $N^*Z_j \cap  \Gamma=\emptyset$. Then the evaluation map $\mathrm{ev}_A:\mathcal C^\infty(M)\to \mathcal M(M)$, $\Phi\mapsto \Phi(A, \bullet)$ extends to a continuous map  $\mathrm{ev}_A:\mathcal C_{\Lambda,\Gamma}^{-\infty}(M)\to \mathcal M^{-\infty}(M)$.
\end{Proposition}

\proof
Take $f\in C_c^\infty(M)$. Then $\WF(f\cdot [[A]])\subset  \cup_i N^*A_i$, and $\WF(\pi^*f\cdot [[\nc(A)]]) \subset \cup_j N^*Z_j$.
Take $\Phi\in \mathcal C^{-\infty}(M)$ represented by the forms $(\phi,\omega)$ with wave fronts in $(\Lambda, \Gamma)$, respectively. Now
\begin{displaymath}
\langle  \mathrm{ev}_A(\Phi), f\rangle:=\langle \phi, f\cdot [[A]]\rangle+ \langle \omega, \pi^*f\cdot [[\nc(A)]]\rangle 
\end{displaymath} 
is the sought after extension.
\endproof

\subsection{Restriction of generalized curvature measures}\label{sec:restriction_of_curvature_measures}

Given an immersion $e\colon X\looparrowright Y$, it is possible to pull-back smooth valuations and curvature measures through $e$. This operation is called \emph{restriction}, and can be extended to generalized valuations and curvature measures if certain conditions are satisfied.

If $e:X\hookrightarrow Y$ is an embedding, the operation of restriction of smooth valuations $e^*:\mathcal V^\infty(Y)\to \mathcal V^\infty(X)$ is characterized by $e^*\psi(A)=\psi(e(A))$. The restriction of curvature measures is analogous: $e^*\Phi(A, B)=\Phi(e(A), \tilde B)$ where $\tilde B\subset Y$ is any Borel subset such that $\tilde B\cap e(A)=B$.
 	
Let us recall the construction in \cite{alesker_intgeo}, compare also \cite{fu_wannerer}. Assume first $e:X\hookrightarrow Y$. Consider the natural maps $q: \mathbb P_+(N^*X)\to X$, $\theta: \mathbb P_+(N^*X)\to \mathbb P_Y$, $\beta:\widetilde{ \mathbb P_Y\times_Y X}\to \mathbb P_X$, $\alpha:\widetilde{ \mathbb P_Y\times_Y X}\to \mathbb P_Y$. Here $q$ is the projection, $\theta$ the fiberwise inclusion, $\widetilde{ \mathbb P_Y\times_Y X}$ the oriented blow-up along $N^*X$, $\beta$ is the smooth map extending the projection $\mathbb P_Y\times_Y X \setminus N^*X\to \mathbb P_X$, and $\alpha$ the composition of the blow-up map with the natural inclusion $ \mathbb P_Y\times_Y X\to \mathbb P_Y$. Let 
\begin{equation}\label{eq:restriction_curvature_measure}
\phi'=q_*\theta^*\omega, \qquad \omega'=\beta_*\alpha^*\omega.
\end{equation}
Then, the restrictions of the smooth curvature measure $[\phi,\omega]$ and its globalization $[[\phi,\omega]]$ are given by 
\begin{align}
 e^*[\phi,\omega]&=[\phi',\omega'],\label{eq:res_curv}\\
 e^*[[\phi,\omega]]&=[[\phi',\omega']].\label{eq:res_val}
\end{align}
Since the construction is local, it immediately extends to immersions $e:X\looparrowright Y$. 

Formula \eqref{eq:res_val} can be used to extend the pull-back to a continuous map
\begin{displaymath}
e^*:\mathcal{V}_{\Lambda,\Gamma}^{-\infty}(Y) \to \mathcal{V}^{-\infty}(X).
\end{displaymath}
provided the transversality conditions
\begin{equation} \label{eq_wf_conditions_pull_back}
\Lambda \cap T^*_XY =\emptyset, \quad \Gamma \cap T^*_{X \times_Y\p_Y} \p_Y=\emptyset, \quad \Gamma \cap T^*_{\p_*(T^*_XY)} \p_Y=\emptyset
\end{equation}   
are satisfied, see \cite{alesker_intgeo} for details.

The case of generalized curvature measures is similar and reads as follows. Note the minor difference in the assumptions, which is due to the fact that $[\phi,\omega]\in\mathcal C^{-\infty}_{\WF(\phi),\WF(\omega)}$, while $[[\phi,\omega]]\in\mathcal V^{-\infty}_{\WF(\zeta),\WF(\tau)}$ with $\zeta,\tau$ given by \eqref{eq:valuation_current}.

\begin{Proposition}\label{prop:restricting_curvature_measures}
 Let $e\colon X \looparrowright Y$ be an immersion of manifolds. Assume it is transversal to $(\emptyset, \Gamma)$ in the sense of \eqref{eq_wf_conditions_pull_back}, and $\Lambda\subset T^*Y \setminus \underline{0}$ is arbitrary. The restriction map $e^*\colon \mathcal C(Y)\rightarrow \mathcal C(X)$ then extends to a continuous map 
\begin{displaymath}
 e^*\colon \mathcal C_{ \Lambda,\Gamma}^{-\infty}(Y)\rightarrow \mathcal C^{-\infty}(X),
\end{displaymath}given by \eqref{eq:res_curv}.
It satisfies $e^*(\Phi_f)=(e^*\Phi)_{e^*f}$ for all $f \in C^\infty(Y)$, in particular $\glob \circ e^*=e^* \circ \glob$.
\end{Proposition}
 
\proof
Given $\Phi=[\phi,\omega]\in\mathcal C_{\Lambda,\Gamma}^{-\infty}(Y)$, the generalized forms $\phi'=q_*\theta^*\omega$ and $\omega'=\beta_*\alpha^*\omega$ are well-defined by standard wave front set considerations as in \cite[Claim 3.5.4]{alesker_intgeo}. We need to check that $[\phi',\omega']=0$ if $\Phi=0$.  By continuity one has $q_*\theta^*(\pi^*f\cdot\omega)=e^*f\cdot\phi'$ and $\beta_*\alpha^*(\pi^*f\cdot \omega)=\pi^*e^*f\cdot\omega'$ for $f\in C^\infty(Y)$. Likewise $\glob[\phi',\omega']=e^*\glob[\phi,\omega]$ by \cite[eq.~(3.5.4)]{alesker_intgeo}.
Hence $[\phi',\omega']_{e^*f}=e^*([\phi,\omega]_f)=0$ for all smooth $f$, which implies $[\phi',\omega']=0$. We may thus put $e^*\Phi:=[\phi',\omega']$.
 \endproof

Given a pair $(\phi,\omega)$ of generalized forms, we will denote by $e^*(\phi,\omega)=(\phi',\omega')$ \emph{the restriction in the sense of curvature measures}, where $\phi',\omega'$ are given by \eqref{eq:restriction_curvature_measure}.

\section{Homogeneous distributions}
\label{sec_computation_integrals}

In this section we recall some homogeneous distributions on the real line that will be used throughout the paper, and introduce some useful terminology for working with generalized differential forms that are homogeneous in a certain sense. In the last subsection, we compute the integral of a distribution-valued functional that will be in the center of our proof of the Weyl principle. 
 By $B(x,y)$ we denote the Euler Beta function.
 We introduce the notation 
\begin{equation} \label{eq:sine_integral}
S(a,b):=\int_{0}^{\frac \pi 2}\sin^a t\cos^{b} tdt= \frac12 B\left(\frac {a+1}2, \frac{b+1}2\right).
\end{equation}
By $\omega_k$ we denote the volume of the $k$-dimensional Euclidean unit ball.

\subsection{Homogeneous generalized functions on the real line}

We refer to \cite{gelfand_shilov} and \cite[Section 3.2]{hoermander_pde1} for the material in this subsection. 

Recall that the Dirac function $\delta_0$ satisfies 
\begin{displaymath}
\langle \delta_0^{(j)},f(x)dx\rangle =(-1)^jf^{(j)}(0). 
\end{displaymath}

Given $x \in \R$, we write $x=x_+-x_-$ with $x_+=\max(0,x), x_-=-\min(0,x)$. Let $s \in \C$ with $\mathrm{Re} s>-1$. Then the function $x_+^s$ is locally integrable and defines a generalized function on $\R$. It is well-known that $x_+^s$ extends to a meromorphic family of generalized functions with simple poles at $s=-1,-2,\ldots$. Explicitly, if $-k-1<\mathrm{Re}s<-k, k \in \mathbb N$ and $\phi \in C^\infty_c(\R)$, then 

\begin{equation}\label{eq:homogeneous_distribution_def}
\langle x_{\pm}^s,\phi(x)dx \rangle=\int_0^{\infty}x^s \left(\phi( \pm x)-\sum_{i=0}^{k-1} (\pm1)^i\frac{\phi^{(i)}(0)}{i!}x^i\right) dx,
\end{equation}
and 
\begin{equation} \label{eq_residue_x+}
\Res_{s=-k} x_{\pm}^s=(\mp 1)^{k-1}\frac{\delta_0^{(k-1)}}{(k-1)!}.
\end{equation}

Similarly, the locally integrable functions given  by $|x|^s:=x_+^s+x_-^s, \sign(x)|x|^s:=x_+^s-x_-^s$ for $\mathrm{Re}s>-1$, extend to meromorphic families of generalized functions. The poles and residues are given by
 \begin{align} 
\Res_{s=-2k+1} |x|^s & =2\frac{\delta_0^{(2k-2)}}{(2k-2)!}, \qquad
\Res_{s=-2k} \sign(x) |x|^s & =-2\frac{\delta_0^{(2k-1)}}{(2k-1)!}\label{eq:res_both}.
\end{align}


We introduce some further notation that will be convenient in our computations.
\begin{align*}
 x^k & := \begin{cases}
 |x|^k & k \text{ is even},\\
 \sign(x)|x|^k & k \text{ is odd}.
 \end{cases}
 \end{align*}
 
Define the following $s$-homogeneous generalized functions for half-integer $s<0$.
\begin{align}
\chi_0^s(x) & :=\begin{cases} 
x_+^s & \text{ if } s\in (2\mathbb Z+1)/2, \\
x^s &  \text{ if } s\in \mathbb Z, \label{eq:chi_0}
\end{cases} \\
\chi_1^s(x) & :=\begin{cases} 
(-1)^{s+\frac12} x_-^s& \text{ if } s\in (2\mathbb Z+1)/2, \\
 (-1)^{s+1} \frac{\pi}{(-s-1)!}\delta_0^{(-s-1)}(x) &  \text{ if } s\in \mathbb Z. \label{eq:chi_1}
\end{cases}
\end{align}

The index $i$ of $\chi_i^s\in\{\chi_0^s,\chi_1^s\}$ is understood modulo $2$. We note the identities 
\begin{equation}
\chi_i^{s+1}(-1)=-\chi_i^s(-1),\quad \chi_i^s(1)=\delta_{0,i},\quad  \chi_i^{s}(-1)=(-1)^{\lfloor s+\frac12\rfloor}\delta_{2s,i},
\end{equation}
where $\delta_{i,j}$ is the Kronecker delta.

\subsection{Homogeneity w.r.t. a function}
\label{subsec_homogeneous_forms_manifolds}

Fix $\sigma\in C^\infty(X)$ on an $m$-dimensional manifold $X$, and assume the level set $X_0=\sigma^{-1}(0)$ contains no critical points of $\sigma$.

\begin{Definition}\label{def:sigma_homogeneous}
A smooth form $\omega\in\Omega(X\setminus X_0)$ is \emph{homogeneous w.r.t. $\sigma$} if $\omega=f(\sigma)\omega'$ with $\omega'\in \Omega(X)$ smooth, and $f(t)=(\sign t)^{\epsilon}|t|^s$ for some $\epsilon\in\{0,1\}$, $s\in\mathbb C$. 

For $\delta\in \R/2\mathbb Z$, let $\mathcal H_\delta(X,\sigma)$ be the set of homogeneous w.r.t. $\sigma$ forms $\omega \in \Omega(X\setminus X_0)$ as above such that $s+\epsilon\equiv \delta \mod 2$. We say $\omega$ has \emph{($\sigma$-)degree} $\deg_\sigma \omega=\delta$.
\end{Definition}

We proceed to establish some basic properties of this class of forms.

\begin{Lemma}\label{lem:ignore_outskirts}
Let $U\subset X$ be open, and $X_0\subset U$.  If $\omega\in\Omega(X\setminus X_0)$ and $\omega|_{U \setminus X_0}\in \mathcal H_\delta(U,\sigma)$, then $\omega\in \mathcal H_\delta(X,\sigma)$.
\end{Lemma}

\proof 
Put $\omega|_{U \setminus X_0}=f(\sigma)\omega'_U$. Then $\omega':=f(\sigma)^{-1}\omega$ extends $\omega'_U$ smoothly to $X$.
\endproof

\begin{Proposition}\label{prop:sigma} 
The space $\mathcal H_\delta(X,\sigma)$ is linear and closed under exterior derivative. The wedge product on forms defines a product 
\begin{displaymath}
\wedge: \mathcal H_{\delta_1}(X,\sigma)\otimes \mathcal H_{\delta_2}(X,\sigma)\to \mathcal H_{\delta_1+\delta_2}(X,\sigma).\end{displaymath}
\end{Proposition}

\proof 
Let us check linearity. Take $\omega_j=f_j(\sigma)\omega'_j$, $j=1,2$ as in Definition \ref{def:sigma_homogeneous}, such that $s_1+\epsilon_{1}\equiv s_2+\epsilon_{2}\equiv \delta \mod 2$. We may assume $s_1\leq s_2$. Then $f_2(\sigma)= \sigma^p f_1(\sigma)$ $p=s_2-s_1 \in \mathbb N \cup\{0\}$, and we conclude that $\omega_1+\omega_2=f_1(\sigma)(\omega_1'+\sigma^p\omega_2')\in\mathcal H_\delta(X,\sigma)$.

Let us check that $d(\mathcal H_\delta(X,\sigma))\subset \mathcal H_\delta(X,\sigma)$. Put $\omega = f(\sigma)\omega'$ as before. We get $d\omega=df\wedge\omega'+fd\omega'$. Now if $f=|\sigma|^s$ then $df=s|\sigma|^{s-1}\sign\sigma d\sigma$, and similarly for the odd case, implying the statement. The last assertion, too, is straightforward.
\endproof

\begin{Definition}\label{def:meromorphic_families}
	Given $\omega=f(\sigma)\omega'\in\mathcal H_\delta(X,\sigma)$ as in Definition \ref{def:sigma_homogeneous}, and a meromorphic in $s\in \C$ family of $s$-homogeneous functions $\chi^s\in C^{-\infty}(\R)$, we define a meromorphic family of generalized forms by setting $\chi^s\cdot\omega:=(\chi^s\cdot f)(\sigma)\omega'$, where $\chi^s \cdot f\in C^{-\infty}(\R)$ is the homogeneous function defined for large $\mathrm{Re}(s)$ by continuity, and then extended meromorphically to $s\in \C$.
\end{Definition}
In particular, the wave front set \begin{equation}\label{eq:WF_sigma}\WF(\chi^s\cdot\omega)\subset \sigma^*\WF(\chi^s\cdot f)\subset \sigma^*N^*\{0\}=N^*X_0.\end{equation}

\begin{Proposition}\label{prop:independence_meromorphic}
Let  $\omega \in \mathcal H_\delta(X,\sigma)$. The following families are analytic at $s=0$.
\begin{itemize}
\item $\sigma_\pm^s\cdot \omega$, if either $\delta\notin \mathbb Z/2\mathbb Z$, or if $\omega$ extends as an $L^1_{\mathrm{loc}}$ form to $X$.
\item $\sign(\sigma)^\delta|\sigma|^s\cdot\omega$, if $\delta \in\mathbb Z/2\mathbb Z$. 
\end{itemize}
\end{Proposition} 

\proof
This follows from the classification of homogeneous distributions. 
\endproof

\begin{Lemma}\label{lem:sigma_pull_back}
Let $\pi:Y\to X$ be a surjective submersion, and $\omega\in\Omega(X\setminus X_0)$. Assume $\pi^*\omega\in\Omega(Y\setminus \pi^{-1}(X_0))$ satisfies $\pi^*\omega\in\mathcal H_\delta(Y, \pi^*\sigma)$. Then $\omega\in\mathcal H_\delta(X,\sigma)$. 
\end{Lemma}

\proof
Write $\pi^*\omega=f(\pi^*\sigma) \omega'_Y$ as in Definition \ref{def:sigma_homogeneous}. It follows that 
$\omega_Y'\in\Omega(Y)$ equals $\pi^*(f(\sigma)^{-1}\omega)$ over $Y\setminus \pi^{-1}(X_0)$. Setting $\omega'_X:=f(\sigma)^{-1}\omega\in\Omega(X\setminus X_0)$, we can write $\pi^*\omega_X'=\omega_Y'$. It follows that $\omega'_X$ extends smoothly to $X$. Indeed, let $s:U\to Y$ be any smooth locally defined right inverse to $\pi$. Then $\omega'_X|_U=s^*\pi^*\omega'_X|_U=s^*\omega'_Y|_U$. As $\omega=f(\sigma)\omega'_X$, we are done.
\endproof

\subsection{A two-variable distributional integral }
The Fourier transform is given by
\begin{displaymath}
\mathcal F f(\xi)=\int_{\R^n} e^{-\i \langle \xi,x\rangle} f(x)\ dx, \quad\xi \in \R^n.
\end{displaymath}
%
%
%
The Fourier transform of homogeneous functions can be found in \cite{gelfand_shilov}. Note the opposite sign in the definition of $\mathcal F$. One computes that for half-integers $s <0$,
\begin{equation} \label{eq_fourier_chi}
\mathcal F(\chi_i^s)(\xi)=\alpha_i^s \xi_+^{-s-1}+\overline{ \alpha_i^s} \xi_-^{-s-1},\quad\mbox{  where}\quad \alpha_0^s = \frac{\Gamma(-s)}{\pi} e^{\i \pi \frac{s}{2}},\quad \alpha_1^s=\i \alpha_0^s.
\end{equation}

\begin{Definition}
For $\chi \in C^{-\infty}(\R)$, denote $$\chi(\sigma\cos^2t+\rho\sin^2t):=\pi_t^*\chi\in C^{-\infty}(\R^2),$$ where the projection $\pi_t:\R^2\to\R$ is given by $(\sigma,\rho)\mapsto \sigma\cos^2t+\rho\sin^2t$. 

For integers $m\geq a\geq 0$, define
$J_{m,a}(\sigma,\rho;\chi)\in C^{-\infty}(\R^2)$ by
\begin{equation}\label{eq:integral_definition}
J_{m,a}(\sigma,\rho;\chi):=\int_0^{\frac \pi 2}\chi(\sigma\cos^2t+\rho\sin^2t)\sin^a t\cos^{m-a}t dt. 
\end{equation}
\end{Definition} 

\begin{Proposition} \label{prop:J_integral_2d} Let $m\geq a\geq0$ be integers, and $i\in\mathbb Z_2$. It holds that
\begin{equation}
J_{m,a}(\sigma,\rho;\chi_i^{-\frac{m+2}{2}})= 
\label{eq:integral}
S(a,m-a) \sum_{j\in\mathbb Z_2}(-1)^{(i+1)j}\chi_{i+j}^{-\frac{m-a+1}{2}}(\sigma) \chi_j^{-\frac{a+1}{2}}(\rho). 
\end{equation}
In particular, $J_{m,a}(\sigma, 1;\chi_i^{-\frac{m+2}{2}})=S(a, m-a)\chi_i^{-\frac{m+1-a}{2}}(\sigma)$.
\end{Proposition}

\proof
Let $f\in \mathcal S'(\R)$. Using the change of variables $x= \sigma\cos^2 t+ \rho\sin^2 t,y=\rho-\sigma$  we find that
\begin{displaymath}
\mathcal F (f(\sigma \cos ^2 t+\rho \sin^2 t))(u,v)=2\pi \mathcal Ff (u+v)\delta_0(- u\sin^2 t + v\cos^2 t).
\end{displaymath}

It follows that
\begin{displaymath}
\mathcal F (J_{m,a}(\sigma,\rho;\chi_i^{-\frac{m+2}{2}}))(u,v)=2\pi \mathcal F\chi_i^{-1-\frac m 2}(u+v) \int_0^{\frac\pi 2} \delta_0(- u\sin^2 t + v\cos^2 t ) \sin^a t\cos^{m-a}tdt.
\end{displaymath} 

For a test function $\psi\in C^\infty_c(\R^2)$ we have 
\begin{displaymath}
\left\langle  \delta_0(-\sin^2 t u +\cos^2 t v ), \psi(u,v)\right\rangle = \int_\R \psi(-z\cos^2t , -z\sin^2t )dz.
\end{displaymath}

We then verify that
\begin{multline*}
\int_{[0,\pi/2]\times \R} \psi(-z\cos^2t , -z\sin^2t ) \sin^a t\cos^{m-a}t dtdz
\\=\frac12\int_{[0,\infty)^2\cup (-\infty,0]^2}\psi(u,v)\left(\frac{v}{u+v}\right)^{\frac a 2}\left(\frac{u}{u+v}\right)^{\frac {m-a} 2}\frac{1}{\sqrt{uv}}dudv.
\end{multline*}

We now compute that 
\begin{align*}
&\mathcal F (J_{m,a}(\sigma,\rho;\chi_i^{-1-\frac{m}{2}}))(u,v)\\
&= \pi \mathcal{F} \chi_i^{-1-\frac{m}{2}}(u+v) \left(\frac{v}{u+v}\right)^{\frac a 2}\left(\frac{u}{u+v}\right)^{\frac {m-a} 2}\frac{1}{\sqrt{uv}}\mathbbm 1_{[0,\infty)^2\cup (-\infty,0]^2}(u,v)
\\&= \pi\frac{\mathcal{F} \chi_i^{-1-\frac{m}{2}}(u+v)}{|u+v|^{\frac m 2}} \left(u_+^{\frac {m-a-1} 2} v_+^{\frac {a-1} 2}+u_-^{\frac {m-a-1} 2} v_-^{\frac {a-1} 2}\right)
\\& = \pi\alpha_i^{-1-\frac{m}{2}}u_+^{\frac{m-a-1}{2}} v_+^{\frac{a-1}{2}}+\pi\overline{ \alpha_i^{-1-\frac{m}{2}}}u_-^{\frac{m-a-1}{2}} v_-^{\frac{a-1}{2}},
\end{align*}
which by \eqref{eq_fourier_chi} is the Fourier transform of the right hand side of \eqref{eq:integral}.
\endproof


\section{Light-cone regularity and light-cone transversality}\label{sec:lc_regular}

We consider manifolds equipped with a metric of changing signature. We define the intrinsic notion of \emph{LC-regularity}, and the extrinsic notion of \emph{LC-transversality} for such manifolds. The main result of the section amounts to the equivalence of the two notions. This equivalence can be seen as a necessary differential-topological precursor to Gauss' Theorema Egregium for such manifolds.

\subsection{Generic metrics of changing signature.}
By \emph{a metric of changing signature} $g$ on a manifold $X$ we understand any smooth symmetric $(0,2)$-tensor $g$. It defines a natural non-negative absolutely continuous measure, denoted $\vol_{X,g}$. Define the light-cone $\LC_X\subset \mathbb P_+(TX)$ as the set of null directions. Let $X_{\mathrm{ND}}\subset X$ denote the non-degenerate points of $g$.

\begin{Definition}\label{def:LC_regularity}
	We call $(X,g)$ \emph{LC-regular} if $0$ is a regular value of $g\in C^\infty(TX\setminus\underline{0})$.
\end{Definition}
In particular, $\LC_X\subset \mathbb P_+(TX)$ is then a smooth hypersurface. Below are some examples of LC-regular metrics in $\R^2$, with the non-degenerate signatures indicated.

\begin{tikzpicture}

\draw[black, thin] (-0.9,1) -- (0.9,1);
\draw[gray, thin, dashed] (0, 0.1) -- (0,1.9);
\node[] at (0,1.5) {\footnotesize $(2,0)$};
\node[] at (0,0.5) {\footnotesize $(1,1)$};
\node[label] at (0,-0.2){\scriptsize $ dx^2+y\cdot dy^2$};

\draw[step=1cm,black, thin] (3-0.9,0.1) grid (3.9,1.9);
\node[] at (3.5,0.5) {\footnotesize $(1,1)$};
\node[] at (3.5,1.5) {\footnotesize $(2,0)$};
\node[] at (2.5,1.5) {\footnotesize $(1,1)$};
\node[] at (2.5,0.5) {\footnotesize $(0,2)$};
\node[label] at (3,-0.2){\scriptsize $x\cdot dx^2+y\cdot dy^2$};

\draw[black, thin] (5.1,1) -- (6.9,1);
\draw[gray, thin, dashed] (6, 0.1) -- (6,1.9);
\node[] at (6,1.5) {\footnotesize $(2,0)$};
\node[] at (6,0.5) {\footnotesize $(0,2)$};
\node[label] at (6,-0.2){\scriptsize $y\cdot dx^2+ y\cdot dy^2$};

\draw[gray, thin, dashed] (8.6,1) -- (10.4,1);
\draw[gray, thin, dashed] (9.5, 0.1) -- (9.5,1.9);
\node at (9.5,1) [circle,fill,inner sep=0.8pt]{};
\node[] at (9,0.7) {\footnotesize $(1,1)$};
\node[label] at (9.5,-0.2){\scriptsize $ x\cdot dx^2+2y\cdot dx dy- x\cdot dy^2$};

\end{tikzpicture}

Note in particular that  $\det (g_{ij})$ need not be Morse at its zeros.

\begin{Lemma}
 Take $p\in (X,g)$ and $v \in \Ker(g_p)$, and let $V$ be any vector field near $p$ extending $v$. Then $d_pg(v):=d_p(g(V))\in T_p^*X$ is independent of $V$.
\end{Lemma}

\proof
 Choose a coordinate chart $U=\R^n$ near $p$, and $w\in T_pX$. We have \[L_w(g(V))=L_w \langle gV, V\rangle=\langle (L_wg)V, V\rangle +2\langle gV, L_wV\rangle=\langle (L_wg)v, v\rangle. \qedhere\] 
\endproof

\begin{Lemma}\label{lem:reduce_to_kernel}
$(X,g)$ is LC-regular if and only if $d_pg(v)\neq 0$ for any $v\in \Ker(g_p)\setminus 0$.
\end{Lemma}

\proof
If $v\in T_pX\setminus \Ker(g_p)$ and $g(v)=0$, then there exists some $w \in T_pX \subset T_{(p,v)}TX$ with $g_p(v,w)\neq0$. It follows that $dg|_{(p,v)}(w)=2g_p(v,w) \neq 0$. 

Thus $(X,g)$ is LC-regular if and only if $d_{p,v}g\neq 0$ for all $v\in\Ker (g)\setminus \{0\}$. 
For $v\in \Ker(g)$, $W\in T_{p,v} TX$ and $w=d\pi(W)\in T_pX$, it clearly holds that $d_{p,v}g(W)=d_pg(v)(w)$.
\endproof

\begin{Proposition} \label{prop_lc_regularity_pseudo_riemann}
	Any pseudo-Riemannian manifold $(X,g)$ is LC-regular. Conversely, the non-degenerate subset $X_{\mathrm{ND}}$ of an LC-regular $(X,g)$ is open and dense.
\end{Proposition}
	
\proof
The first statement is immediate from Lemma \ref{lem:reduce_to_kernel}.	
	
For the second, assume that $(X,g)$ is LC-regular. The set where $g$ is pseudo-Riemannian is clearly open. If it is not dense, we may replace $X$ with a degenerate neighborhood. We thus have $X=\R^n$, and $g=\sum g_{ij}dx_idx_j$, $\det (g_{ij})=0$.
	
Find the greatest $k\geq 1$ such that $\dim \Ker g \geq k$ in $\R^n$. We may further assume that $\dim \Ker g|_0=k$, and $\Ker g|_0=\Span\{e_{n-k+1},\dots, e_n\}$. It then holds that $g_{ij}(0)=g_{ji}(0)=0$ for $1\leq i\leq n$, $n-k+1\leq j\leq n$. Let $\gamma(t)$ be any curve through $0$. Denote by $M_j(t)$ the principal $j$-minor of $g|_{\gamma(t)}$ with the last $(n-j)$ rows and columns deleted.  Writing $\det M_{n-k+1}(t)$ explicitly, we conclude that $(\det M_{n-k+1})'(0)=g_{n-k+1,n-k+1}'(0)\det M_{n-k}(0)$. 

Since $\dim \Ker g|_0=k$, $\det M_{n-k}(0)\neq 0$, while by LC-regularity, one can choose $\gamma(t)$ such that $g_{n-k+1,n-k+1}'(0)\neq 0$. It follows that $\det M_{n-k+1}(t)\neq 0$ for small $t>0$, contradicting $\dim \Ker g\geq k$. 
\endproof

\subsection{LC-transversality and LC-regularity}

Let $(M,Q)$ be a pseudo-Riemannian manifold. For $X\subset M$, $N^QX$ will denote the normal bundle with respect to $Q$ in either $\mathbb P_+(TM)$ or $TM$, depending on context. Let $\LC_M^*\subset \mathbb P_M$ correspond to $\LC_M\subset \mathbb P_+(TM)$ under the identification induced by $Q$.
	
\begin{Definition}\label{def:LC_transversality}
A differentiable polyhedron $ A\subset M$ is \emph{LC-transversal} if each smooth stratum of $\nc(A)$ intersects $\LC^*_M$ transversally. In particular, a submanifold $X \subset M$ is LC-transversal if $N^QX\pitchfork \LC_M$ in $\mathbb P_+(TM)$, or equivalently if $Q\in C^\infty (N^QX\setminus\underline 0)$ has no critical zeros.
\end{Definition}

LC-transversality is generic in the following sense.  

\begin{Proposition}\label{prop:transversality_is_generic}
	Any differentiable polyhedron $A \subset M$ becomes LC-transversal following an arbitrarily small perturbation of $A$.
\end{Proposition}
	
\proof
Immediate from the transversality theorem. 
\endproof

Let us consider the important class of LC-transversal hypersurfaces. Let $X^n\subset M^{p,q}$ be an oriented hypersurface. At $x\in X$, the Gauss map $\nu:X\to \mathbb P_+(TM)$ gives rise to the shape operator $S_x:T_xX\to T_{\nu(x)}\mathbb P_+(T_xM)$ by composing $d_x\nu: T_xX \to T_{x,\nu(x)}\mathbb P_+(TM)$ with the projection $\pi_V:T_{x,\nu(x)}\mathbb P_+(TM)\to T_{\nu(x)}\mathbb P_+(T_xM)$ on the vertical subspace induced by the Levi-Civita connection. We say that $X$ has nonzero principal curvatures at $x\in X$ if $S_x$ is bijective.
For $M=\R^{p,q}$, this is equivalent to $M$ having non-zero Gaussian curvature with respect to any Euclidean metric (as the connections of the Euclidean and pseudo-Euclidean metrics coincide).

\begin{Lemma}\label{lem:nondegenerate_hypersurface}
Assume that a hypersurface $X\subset (M^{p,q},Q)$ has nonzero principal curvatures at every point $x$ such that $Q|_{T_xX}$ is degenerate. Then $X$ is LC-transversal.
\end{Lemma}

\proof
We should check that 
\begin{displaymath}
\mathrm{Im}(d_x\nu)+ T_{x,\nu(x)}\LC_M=T_{x,\nu(x)} \mathbb P_+(TM).
\end{displaymath}
Decompose into the horizontal and vertical subspaces: $T_{x,\nu(x)}\mathbb P_+(TM)= H\oplus V$. Clearly $T_{x,\nu(x)}\LC_M=H\oplus V_0$, where $V_0\subset V$ is a hyperplane. By assumption, $\textrm{Im}(\pi_V\circ d_x\nu)=V$, showing $X$ is LC-transversal.
\endproof

Back to general codimension, we now show that LC-transversality is stable under isometric embeddings.
\begin{Proposition}\label{prop:LC_regularity_transitivity}
Let $M\subset W$ be a pair of pseudo-Riemannian manifolds. Then a differentiable polyhedron $A \subset M$ is LC-transversal if and only if $A \subset W$ is LC-transversal.
\end{Proposition}
	
\proof 
We write  $N^W_pM$ for the $Q$-normal space to $T_pM$ in $T_pW$, and similarly for the normal cycle. Let us first assume that $A=X$ is a submanifold, $\partial X=\emptyset$. 
	
Let $X\subset M$ be LC-transversal. Consider $(p,\nu)\in \LC_W\cap N^W_pX$ with $\nu\neq 0.$ We should find a path $x(t)\in X$, $x(0)=p$, and $\nu(t)\in N^W_{x(t)}X$ with $\left.\frac{d}{dt}\right|_{t=0} Q(\nu(t))\neq 0$. We have the $Q$-orthogonal decomposition $T_pW=T_pM\oplus N_p^W M$, and we decompose $\nu=\nu_1+\nu_2$ accordingly. In particular, $Q(\nu_1)=-Q(\nu_2)$, and both $\nu_1,\nu_2\in N_p^WX$.
	
If $Q(\nu_1)=Q(\nu_2)=0$ , use the LC-transversality of $X\subset M$ to choose $\nu_1(t)\in N_{x(t)}^MX\subset N_{x(t)}^WX$ such that $\left.\frac{d}{dt}\right|_{t=0}Q(\nu_1(t))\neq 0$. One can choose $\nu_2(t)\in N_{x(t)}^WM\subset N_{x(t)}^WX$ such that $Q(\nu_2(t))\equiv 0$, and therefore for $\nu(t):=\nu_1(t)+\nu_2(t)$ we have $\left.\frac{d}{dt}\right|_{t=0}Q(\nu(t))=\left.\frac{d}{dt}\right|_{t=0}Q(\nu_1(t))\neq 0$.
	
Otherwise, we have $Q(\nu_1)=-Q(\nu_2)\neq 0$, and can take $x(t)=p$ and $\nu(t)=(1+t)\nu_1+(1-t)\nu_2$.  Then $Q(\nu(t))=(1+t)^2Q(\nu_1)+(1-t)^2Q(\nu_2)=4tQ(\nu_1)$, concluding this case. 
	
Assume now $X\subset W$ is LC-transversal, and take $(p,\nu)\in N_p^MX\cap \LC_M$ with $\nu\neq 0$. Since $N_p^MX\subset N_p^WX$, it follows we can find a curve $x(t)\in X$ and $\nu(t)\in N_{x(t)}^WX$ with $x(0)=p$, $\nu(0)=\nu$ such that $\left.\frac{d}{dt}\right|_{t=0}Q(\nu(t))\neq 0$. Decompose $Q$-orthogonally $\nu(t)=\nu_1(t)+\nu_2(t)\in T_{x(t)}M\oplus N_{x(t)}^WM$, so that $Q(\nu(t))=Q(\nu_1(t))+Q(\nu_2(t))$. It follows that $\nu_1(t)=\nu(t)-\nu_2(t)\in N_{x(t)}^MX$. Note that $\nu_2(0)=0$, hence $\left.\frac{d}{dt}\right|_{t=0}Q(\nu_2(t))=0$. We conclude that $\left.\frac{d}{dt}\right|_{t=0}Q(\nu_1(t))\neq 0$, implying $\left.\frac{d}{dt}\right|_{t=0}(x(t), \nu_1(t))\notin T_{p,\nu_1}\LC_M$. It follows that $X\subset M$ is LC-transversal.
	
Finally, consider a differentiable polyhedron $ A$. A smooth stratum $Z_M$ of $\nc^M(A)$ gives rise to a smooth subset $Z_W$ of $\nc^W(A)$ in $W$ (which in general is only a subset of a smooth stratum): Assume $Z_M=\{(x,[\xi]): x\in  F, \xi \in L_x \}$,  where $F$ is a face of $A$, and $L_x\subset T_xM$ is a cone. Then $Z_W=\{ (x,\xi+\nu): x\in F, \xi\in L_x, \nu\in N^W_xM\}$. The union of all subsets $Z_W$ obtained that way is $\nc^W(A)$, unless $A$ is full-dimensional in $M$, whence one or two strata contained in $N^WM$ are not included, however the latter do not meet $\LC_W$. The proof above now can be repeated verbatim. 
\endproof
	
We now proceed to establish that for a manifold $X$, LC-regularity is equivalent to LC-transversality in the following sense.
	
\begin{Proposition}\label{prop:LC_regular_is_intrinsic}
	Let $(X,g)$ be a metric of changing signature. Then the following are equivalent.
	\\\emph{i)} There exists an isometric immersion $e:X\hookrightarrow M$ into a pseudo-Riemannian $M$ such that $e(X)$ is LC-transversal.
	\\\emph{ii)} For any isometric immersion $e:X\hookrightarrow M$ into a pseudo-Riemannian $M$, $e(X)$ is LC-transversal.
	\\\emph{iii)} $(X,g)$ is LC-regular.
\end{Proposition}

We will need the following simple embedding results.

\begin{Lemma}\label{lemma:baby_nash}
Let $U\subset \R^n$ be open, and $g$ a metric of changing signature on $U$. Then one can find $M^{n,n}$ and an isometric embedding $e:(U,g)\hookrightarrow M$. 
\end{Lemma}
	
\proof
We define $e:U\to \R^{2n}, x \mapsto (x, 0)$, and construct a pseudo-Riemannian metric $G$ on $U \times \R^n$. We use the standard Euclidean structures on $\R^n, \R^{2n}$ to identify $g, G$ with fields of symmetric matrices. Set  
\begin{displaymath}
G|_{(x,y)}=\left(\begin{matrix} g|_x & R(x) \cdot \mathrm{Id}_n\\R(x) \cdot \mathrm{Id}_n & 0 \end{matrix}\right).
\end{displaymath}
Clearly $e^*G=g$, and choosing $R(x)\gg \|g_x\|_\infty$, $G$ has signature $(n,n)$.
\endproof
	
\begin{Lemma}\label{lemma:two_manifolds_simultaneous_embedding}
Let $e_j:(X^n,g)\hookrightarrow (M_j, Q_j)$ be isometric embeddings, with $(M_j, Q_j)$ pseudo-Riemannian, $j=1,2$. Then for any $p\in X$ one can find a neighborhood $U\subset X$ of $p$ and open sets $U_j\subset M_j$  with $U_j\cap e_j(X)=e_j(U)$, and a pseudo-Riemannian $(M,Q)$, and isometric embeddings $f_j:U_j\hookrightarrow M$ such that $f_1\circ e_1|_{U}=f_2\circ e_2|_{U}$, the latter common image $\tilde X$ coincides with $f_1U_1\cap f_2U_2$, and $T(f_1U_1)\cap T(f_2U_2)=T\tilde X$. 
\end{Lemma}
	
\proof
We may work locally, thus assume $X=(\R^n,g)$, $p=0$, $M_j=(\R^{n+k_j}, Q_j)$, and $e_j(X)$ is the coordinate subspace in each $M_j$ given by $y_{n+1}=\dots=y_{n+k_j}=0$. Take  $\tilde M=\R^{n+k_1+k_2}$, and set 
\begin{align*} 
\tilde f_1(y):=(y_1,\dots,y_n,y_{n+1},\dots,y_{n+k_1}, 0,\dots,0),\\ 
\tilde f_2(y):=(y_1,\dots, y_n, 0,\dots,0, y_{n+1},\dots,y_{n+k_2}). 
\end{align*}
	
Now choose an arbitrary symmetric $2$-tensor $\tilde Q$ on $\tilde M$ restricting to $Q_j$ on $M_j$, and use Lemma \ref{lemma:baby_nash} to find an isometric embedding $e:(\tilde M, \tilde Q) \hookrightarrow (M,Q)$ in a pseudo-Riemannian manifold. It remains to set $f_j=e\circ \tilde f_j$.
	\endproof
	
	Finally, we will need the following statements, which will allow us to identify the tangent and normal bundles in a particular case. Denote $\perp_Q(E):=E^Q$.
	
\begin{Proposition}\label{prop:linear_smooth_normal_isometry}
Let $Q$ be the standard $(n,n)$ form on $\R^{2n}=\R^{n,n}$. Let $E_0\in \Gr_n(\R^{n,n})$ be any subspace. Then there is a $\perp_Q$-invariant open neighborhood $U$ of $E_0$, and a smooth section $B: U\to \Hom(E, E^Q)$ such that $B_{E^Q}=B_E^{-1}$, and for all $x\in E\in U$, $Q(B_Ex)=-Q(x)$.
\end{Proposition}
	
\proof
 Let $P$ be a Euclidean structure on $\R^{2n}$. Let $X_P$ be $\Gr_n(\R^{2n})$ equipped with the standard $ \OO(P)$-invariant metric, and let $d_P(E,F)$ denote its distance function. Note that $X_P$ is isometric to $X_{P'}$ for any other $P'$, in particular its convexity radius $\epsilon_0$ is independent of $P$.
A Euclidean structure given by $P(x,y)=Q(Sx,y)$, with $S\in\OO(Q)$ an involution, is called \emph{$Q$-compatible}, see \cite{bernig_faifman_opq} for some basic properties of compatible Euclidean structures. For such $P$ it holds that $\perp_Q=S\circ\perp_P$, and since $S\in\OO(P)$ we conclude that $\perp_Q$ is an isometry of $X_P$. 
	
	\textit{Claim.} There is a $Q$-compatible $P$ and a $\perp_Q$-invariant open set $U$ around $E_0$ such that any two points in $U$ are connected by a unique shortest geodesic.
	
Indeed, let $\Lambda_n(\R^{2n})$ be the set of $Q$-isotropic subspaces, and note that $\Lambda_n(\R^{2n})$ lies in the closure of any $\OO(Q)$-orbit on $\Gr_n(\R^{2n})$. It suffices to find $P$ and some $L_0\in \Lambda_n(\R^{2n})$ such that $d_P(E_0,L_0)<\epsilon_0$, for then $\perp_Q(L_0)=L_0$ and we can take $U:=\{E:d_P(E,L_0)<\epsilon_0\}$. To find such $(L_0,P)$, fix any $Q$-compatible metric $P'$, and choose $g\in \OO(Q)$ and some $L_0'\in\Lambda_n(\R^{2n})$ such that $d_{P'} (gE_0, L_0')<\epsilon_0$. As $\OO(Q)$ acts on the space of $Q$-compatible Euclidean structures by conjugating $S$, we conclude that $P=g^{-1}P'$, $L_0=g^{-1}L_0'$ satisfy  $d_{P} (E_0, L_0)<\epsilon_0$, proving the claim.
	
Fix such $P, U$. For any $E\in U$, let $\Lambda(E)$ be the midpoint of the shortest geodesic $(E, E^Q)$. Since the endpoints are interchanged by the isometry $\perp_Q$, we conclude that $\Lambda(E)^Q=\Lambda(E)$, that is $\Lambda(E)$ is isotropic. Let $\pi_E:\R^{2n}\to \Lambda(E)$ be the $P$-orthogonal projection. Then $B_Ex:= 2\pi_Ex-x\in\mathrm{End}(\R^{2n})$ is the reflection with respect to $\Lambda(E)$, in particular $B_E(E)=E^Q$. It follows that for $x\in E$, $Q(B_Ex)=4Q(\pi_Ex)+Q(x)-4Q(x,\pi_Ex)=Q(x)-4Q(x,\pi_Ex)=-Q(x)$, where the last equality follows from 
\begin{displaymath}
2Q(x,\pi_Ex)=Q(x)\iff Q(x,2\pi_Ex-x)=0\iff Q(x,B_Ex)=0.\qedhere
\end{displaymath}

\begin{Corollary}\label{cor:local_identification_tangent_normal} 
Let $(X^n,g)\subset (M^{n,n},Q)$ be an isometrically immersed manifold, $p\in X$. Then one can find an open neighborhood $U\subset X$ of $p$, and a smooth section $B:U\to \Hom(TX, N^QX)$ such that  $B_x$ is bijective for all $x\in U$, and $B_x^*Q=-g$.
\end{Corollary}
	
\proof
Choose a neighborhood $U_1\subset M$ of $p$ and a smooth section $A:U_1\to \Hom(T_pM, TM)$ such that $A_x\colon T_pM\to T_xM$ is an isometry for all $x\in U_1$. Define $E:U_1\cap X\to \Gr_n(T_pM)$ by $E(x)= A_ x^{-1}(T_xX)$. Use Proposition \ref{prop:linear_smooth_normal_isometry} to find an open neighborhood $U\subset X\cap U_1$ of $p$, and a section $\tilde B:U\to \Hom(E(x), E(x)^Q)$ such that $\tilde B^*Q=-Q$. Clearly $B_x:=A_x\circ \tilde B_x\circ A_x^{-1}$ is the desired section.
\endproof
	
\proof[Proof of Proposition \ref{prop:LC_regular_is_intrinsic}.]
As the statement is local, we may replace immersions by embeddings. The implication $ii)\Rightarrow i)$ follows from Lemma \ref{lemma:baby_nash}. 
	
For the reverse implication, assume the first item, and let $e':X\hookrightarrow M'$ be another embedding. Use Lemma \ref{lemma:two_manifolds_simultaneous_embedding} to construct a pseudo-Riemannian $N$, and isometric embeddings $f:M\to N$, $f':M'\to N$ that coincide on $X$. By Proposition \ref{prop:LC_regularity_transitivity}, $ e(X)$ is LC-transversal in $ M$ if and only if $ f(e(X))=f'(e'(X))$ is LC-transversal in $N$, if and only if $ e'(X)$ is LC-transversal in $M'$.
	
{By Lemma \ref{lemma:baby_nash} and Corollary \ref{cor:local_identification_tangent_normal}, we may embed $(X,g)$ in $(M^{n,n},Q)$ such that the critical zeros of $g \in C^\infty(TX \setminus \underline{0})$ correspond to the critical zeros of $Q \in C^\infty(N^QX \setminus \underline{0})$, thus LC-regularity of $X$ is equivalent to LC-transversality of $X \subset M$. }
\endproof
	
\begin{Corollary}
LC-regularity of $(X,g)$ is a generic property of the metric.
\end{Corollary}
	
\proof 
Fix an isometric embedding $X\hookrightarrow M^{p,q}$, apply Propositions \ref{prop:LC_regular_is_intrinsic} and \ref{prop:transversality_is_generic}. 
\endproof

\subsection{Wave front sets and restrictions of curvature measures}

\begin{Proposition} \label{prop:nondegenerate_restriction}
Let $e:X\looparrowright M^{p,q}$ be an LC-regular immersed submanifold, and assume $\Phi\in\mathcal C_{\emptyset, N^*(\LC_M^*)}^{-\infty}(M)$. Then $e^*\Phi\in\mathcal C^{-\infty}(X)$ is well-defined.
\end{Proposition}

\proof
Let us check that $e$ is transversal to $(\emptyset, N^*(\LC_M^*))$ in the sense of  \eqref{eq_wf_conditions_pull_back}. The first condition  holds trivially. For the second one, we ought to show that $\LC_M^*$ and $X\times_M\mathbb P_M$ intersect transversally in $\mathbb P_M$, which is clear, as $\LC_M^*$ is a fiber bundle over $M$. Finally, the third condition $N^*(\LC_M^*)\cap N^*(N^*X)=\emptyset$ is equivalent to $N^*X\pitchfork \LC^*_M$, which holds by assumption. Proposition \ref{prop:restricting_curvature_measures} concludes the proof.
\endproof

\section{Construction of the curvature measures} 
\label{sec_construction}

Let $(M,Q)$ be a pseudo-Riemannian manifold of dimension $m+1$ and signature $(p,q)$. We will construct certain natural differential forms on $\mathbb P_M=\mathbb P_+(T^*M)$, which throughout this section is identified with $\mathbb P_+(TM)$ using the non-degenerate form $Q$ (i.e. $[Q(v,\cdot)]\equiv [v]$ for $v\in TM \setminus\underline 0$). 

\subsection{Construction of smooth forms outside of the light cone}\label{sec:smooth_forms}
\label{subsection_open_orbits}
Given $\epsilon=\pm 1$, let $K=K_+\cupdot K_-$ be a partition of $\{0,\dots,m\}$ in two sets of cardinalities $\#K_+=p$ and $\#K_-=q$ such that $0\in K_\epsilon$. Let $\epsilon_0,\ldots, \epsilon_m$ be given by $\epsilon_i=1$ if $i\in K_+$ and $\epsilon_i=-1$ if $i\in K_-$. In particular $\epsilon=\epsilon_0$.

Consider the group $G_K^\epsilon=\{g\in \textrm{SL}(m+1,\mathbb R): g^TI_K^\epsilon g=I_K^\epsilon \}$, where $I_K^\epsilon=\textrm {Diag}(\epsilon_0,\dots,\epsilon_m)$, which is isomorphic to $\mathrm{SO}(p,q)$. Note that if $g \in G_K^\epsilon$, then $(g^{-1})_{i,j}=g_{j,i} \epsilon_i \epsilon_j$. Let $\mathfrak g_K^\epsilon$ be its Lie algebra.

We denote by $S^\epsilon M$ the bundle over $M$ consisting of all tangent vectors of norm $\epsilon= \epsilon_0$. Let $F_K^\epsilon$ be the bundle over $M$ consisting of all tuples $(p,B_0,\ldots,B_m)$ such that $B_0,\ldots, B_m$ is a positive basis of $T_pM$ and $Q(B_i,B_j)=\delta_{ij}\epsilon_i$. We will denote $\pi_M\colon S^\epsilon M\to M$ and $\pi_0:F_K^\epsilon \to S^\epsilon M$ the projections $\pi_M(x,v)=p$, and $\pi_0(x,B_0,\ldots,B_m)=(x,B_0)$. The group $G_K^\epsilon$ acts on $F_K^\epsilon$ by
\begin{displaymath}
 g(x,B_0,\ldots,B_m)=(x,\sum_a B_a(g^{-1})_{a,0},\ldots,\sum_a B_a(g^{-1})_{a,m}).
\end{displaymath}

The stabilizer of $(x, B_0)$ will be denoted $H_K^\epsilon$. It is isomorphic to $\mathrm{SO}(p-1,q)$ for $\epsilon=+1$, and to $\mathrm{SO}(p,q-1)$ when $\epsilon=-1$. Forms on $S^\epsilon M$ may be identified with forms on $F_K^\epsilon$ which are invariant under $H_K^\epsilon$ and vanish whenever a tangent vector to the fiber of $\pi_0$ is plugged in. 

Consider the solder forms $\theta_i$, the connection forms $\omega_{i,j}$ associated to the Levi-Civita connection $\nabla$ of $Q$, and the curvature forms $\Omega_{i,j}$ corresponding to the curvature tensor $R(X,Y)Z=\nabla_X \nabla_YZ-\nabla_Y \nabla_XZ-\nabla_{[X,Y]}Z$. They are defined by
\begin{equation}\label{eq:convention}
 d\pi_M=\sum_{a=0}^m \theta_a B_a,\qquad \nabla B_j=\sum_{a=0}^m \omega_{a,j} B_a,\qquad  R(\ ,\ )B_j=\sum_{a=0}^m \Omega_{a,j}B_a.
\end{equation}

Note that $\omega=(\omega_{i,j})$ and $\Omega=(\Omega_{i,j})$ take values in $\mathfrak g_K^\epsilon \simeq\mathfrak{so}(p,q)$, so that 
\begin{equation}
\omega_{j,i} = -\epsilon_i\epsilon_j\omega_{i,j},\qquad\Omega_{j,i} = -\epsilon_i\epsilon_j\Omega_{i,j}.
 \end{equation}
This suggests to introduce the following notation. 
\begin{displaymath}
 \tilde \omega_{i,j} :=\omega_{i,j} \epsilon_j, \quad \tilde \Omega_{i,j}:=\Omega_{i,j} \epsilon_j.
\end{displaymath}
Then $\tilde \omega_{i,j},\tilde \Omega_{i,j}$ are antisymmetric with respect to their indices. The action of $G_K^\epsilon$ on these forms is given by 
\begin{displaymath}
 g^* \theta_i = \sum_{a=0}^m g_{i,a}  \theta_a, \quad g^*\tilde \omega_{i,j}  = \sum_{a,b=0}^m g_{i,a}g_{j,b} \tilde \omega_{a,b}, \quad  g^*\tilde\Omega_{i,j}  = \sum_{a,b=0}^m g_{i,a}g_{j,b} \tilde \Omega_{a,b}.
  \end{displaymath}

The structure equations read
\begin{align}
 d \theta_i & =  \sum_{a=0}^m \epsilon_a    \tilde \omega_{a,i}\wedge \theta_a, \\
 d\tilde \omega_{i,j} & = - \sum_{a=0}^m \epsilon_a \tilde \omega_{i,a} \wedge \tilde \omega_{a,j} + \tilde \Omega_{i,j},  \\
 d\tilde \Omega_{i,j} & = \sum_{a=0}^m \left(\epsilon_a \tilde \Omega_{i,a} \wedge \tilde \omega_{a,j}-\epsilon_a \tilde \omega_{i,a} \wedge \tilde \Omega_{a,j}\right). 
\end{align}

Note that $\theta_0$ descends to a well-defined $\alpha \in \Omega^1(S^\epsilon M)$, which is a contact form.

\begin{Lemma} \label{lemma_def_phi}
The form $\phi^\epsilon_{kr} \in \Omega^m(F_K^\epsilon)$ for $0 \leq 2r \leq k \leq m$ defined by
 \begin{equation}\label{eq_def_phi}
 \phi_{kr}^\epsilon:= \sum_{\pi} \mathrm{sgn}(\pi)  \tilde \Omega_{\pi_1\pi_2}\wedge \ldots  \wedge\tilde \Omega_{\pi_{2r-1}\pi_{2r}} \wedge  \theta_{\pi_{2r+1}}\wedge \ldots \wedge\theta_{\pi_k} \wedge \tilde \omega_{\pi_{k+1}0}\wedge \ldots \wedge\tilde \omega_{\pi_m0},
 \end{equation}
where $\pi$ runs over permutations of $1,\ldots,m$, descends to a form $\phi_{k,r}^\epsilon \in \Omega^m(S^\epsilon M)$. Moreover, this latter form is independent of $K$. 
\end{Lemma}

\proof
 The first statement is an easy computation analogous to the Riemannian case and we omit the details. 

Let us show that $\phi_{k,r}^\epsilon$ is independent of $K$. Let $\rho$ be a permutation of $\{1,\ldots,m\}$ with inverse $\tau:=\rho^{-1}$. Then 
\begin{displaymath}
A: F_K^\epsilon \to F_{\rho K}^\epsilon, (p,B_0,B_1,B_2,\ldots,B_m) \mapsto (p,B_0,\mathrm{sgn}(\rho)B_{\tau_1},B_{\tau_2},\ldots,B_{\tau_m})
\end{displaymath}
is a diffeomorphism which commutes with $\pi_0$ (the sign is needed since we want to map a positive basis to a positive basis). It satisfies $A^*\theta_i=\theta_{\tau_i}, A^*\tilde\omega_{ij}=\tilde\omega_{\tau_i,\tau_j},A^*\tilde\Omega_{ij}=\tilde\Omega_{\tau_i,\tau_j}$ if $i,j>1$, and the same equations with added factor $\mathrm{sgn}(\rho)$ if $i=1$ or $j=1$. It follows that $A^*\phi^\epsilon_{k,r}=\phi^\epsilon_{k,r}$, hence $\phi_{k,r}$ is independent of $K$.
\endproof

\begin{Lemma}\label{lem:interior_smooth_form}
The form
\begin{displaymath}
\psi_{m+1,r}:=\sum_{\pi\in S_{n+1}} \sgn(\pi) \tilde \Omega_{\pi_0 \pi_1} \wedge \ldots \wedge \tilde \Omega_{\pi_{2r-2}\pi_{2r-1}} \wedge \theta_{\pi_{2r}} \wedge \ldots \wedge \theta_{\pi_m}  \in \Omega^{m+1}(F_K^\epsilon)
\end{displaymath}  
descends to a form $\psi_{m+1,r} \in \Omega^{m+1}(M)$ which is independent of $\epsilon$ and $K$.
\end{Lemma}

\proof 
Similar to Lemma \ref{lemma_def_phi} . 
\endproof

\subsection{Compatible Riemannian metrics}
We will be making use of a carefully chosen auxiliary Riemannian structure, as follows.

\begin{Definition}\label{def:compatible}Let $(M,Q)$ be a pseudo-Riemannian manifold.
\begin{enumerate}
	\item A Riemannian metric $P$ on $M$ is \emph{compatible with $Q$} if the tangent space admits a decomposition $TM=V_+\oplus V_-$, with positive definite forms $g_\pm$ on $V_\pm$ such that $P=g_+\oplus g_-$, $Q=g_+\oplus (-g_-)$. 
	\item A compatible Riemannian metric is \emph{quadratically compatible at $y\in M$} if there are coordinates $x_1,\dots, x_{m+1}$ on $M$ around $y$ such that $y=(0,\ldots,0)$,
\begin{displaymath}
P=\sum_i dx_i^2 + \sum\limits_{i\leq j} e_{i,j}dx_idx_j,\quad Q=\sum_i \epsilon_i dx_i^2 + \sum\limits_{i\leq j} \tilde e_{i,j}dx_idx_j,
\end{displaymath}
with $e_{i,j}, \tilde{e}_{i,j}=O( \|x\|_2^2), \epsilon_i=\pm1$.
\end{enumerate}	
\end{Definition}

\begin{Lemma}\label{lem:smooth_decomposition}
	Let $U\subset \R^m$ be open, and $A:U\to \Sym_n(\R)$ a smooth family of non-degenerate symmetric matrices of signature $(p,q)$.  For $x \in U$, let $E_+(x) \subset \R^n$ be the maximal subspace such that $A(x)(E_+(x))=E_+(x)$ and $A(x)|_{E_+(x)}>0$. Then $E_+:U \to \Gr_{p}(\R^n)$ is smooth. Similarly, the maximal negative-definite invariant subspace  $E_-(x)$ depends smoothly on $x$.  Moreover, if $d_xA=0$ then $d_xE_\pm=0$.
\end{Lemma}

\proof
Let $\lambda_1(x)\geq \dots \geq \lambda_p(x)>0>\lambda_{p+1}(x) \geq \dots \geq \lambda_n(x)$ be the eigenvalues of $A(x)$, with corresponding Euclidean-orthonormal eigenbasis $v_1(x),\dots,v_n(x)$. As the statement is local, we may assume that $\lambda_n+c>0$ in $U$ for some fixed $c$. Then the wedge products of $p$-tuples of the $v_j(x)$ form an eigenbasis of $B(x):=\wedge^p(A(x)+ cI) \in \mathrm{End}(\wedge^p \R^n)$. In particular, $B(x)$ has a unique maximal eigenvalue $\lambda(x)=\prod_{j=1}^p(\lambda_j(x)+c)$ with eigenvector $w(x)=v_1(x) \wedge \dots \wedge v_p(x)$, which therefore must depend smoothly on $x$. But $w(x)$ corresponds to $E_+(x)$ under the Pl\"ucker embedding, thus $E_+:U\to \Gr_{p}(\R^n)$ is smooth. The case of $E_-$ is identical. 

Finally, assume $d_x A=0$, hence $d_xB=0$. As $B(x)w(x)=\lambda(x)w(x)$, we find $(B(x)-\lambda(x)I)d_xw=d_x\lambda\cdot w(x)$.
Now $\|w\|_2=1$ everywhere, so $\mathrm{Image}(d_xw)\perp w(x)$. But $w(x)^\perp$ is an invariant subspace of $B(x)-\lambda(x)I$, and so it must hold that $(B(x)-\lambda(x)I)d_xw=0$. As $B(x)-\lambda(x)I$ is invertible on $w^\perp$, we find $d_xw=0$.
\endproof

\begin{Lemma}\label{lemma_quad_comp}\mbox{}
\\\emph{i)} Any pseudo-Riemannian $(M, Q)$ has a compatible Riemannian metric.\\
\emph{ii)} For any $y\in M$ there is a neighborhood $U\subset M$ of $y$ and a compatible Riemannian metric $P$ in $U$ that is quadratically compatible with $Q$ at $y$. \\
\emph{iii)} Consider  $\R^{p,q}=\R^{p+q}$ with the standard bilinear forms $Q,P$. Let $M^{p',q'}\subset \mathbb R^{p,q}$ be  a pseudo-Riemannian submanifold, and $y\in M$ such that $T_{y}M$ is spanned by the coordinate vectors $\frac{\partial}{\partial x_1},\ldots,\frac{\partial}{\partial x_{p'}},\frac{\partial}{\partial x_{p+1}},\ldots, \frac{\partial}{\partial x_{p+{q'}}}$. Then there is a $Q$-compatible Riemannian metric $P'$ on a neighborhood of $y$ in $M$, that is quadratically compatible at $y$, and $P|_M-P'=O(\|x-y\|_2^2)$.
\end{Lemma}

\proof
We follow Chern \cite{chern63}. 

{i)} Start with any Riemannian structure $P_0$, let $V_\pm$ be the resulting decomposition of $TM$ into positive and negative eigenspaces of $Q$ with respect to $P_0$, and then define $g_\pm:= \pm Q|_{V_\pm}$. By Lemma \ref{lem:smooth_decomposition}, $V_\pm$ are smooth subbundles of $TM$, hence $P=g_+\oplus g_-$ is a smooth compatible Riemannian metric. 

{ii)} By Theorem \ref{thm_nash}, ii) follows from iii).

{iii)} By assumption, 
\begin{displaymath}
 P|_M=\sum_{i\in I}  dx_i^2+ \sum\limits_{i, j\in I}e_{i,j}dx_idx_j,\qquad Q|_M=\sum_{i\in I} \epsilon_i dx_i^2+ \sum\limits_{i, j\in I} \tilde e_{i,j}dx_idx_j,
\end{displaymath}
with $\epsilon_i=\pm1$, $ e_{i,j}(x),\tilde e_{i,j}(x)=O(\|x-y\|_2^2)$, and $I=\{1,\ldots,p', p+1,\ldots, p+q'\}$. Take $P_0=P|_M$,  yielding a decomposition $TM=V_+\oplus V_-$ and a compatible Riemannian metric $P'=Q|_{V_+}\oplus (-Q)|_{V_-}$  as in part i). By the last statement of Lemma \ref{lemma_quad_comp}, $P'$ is quadratically compatible at $y$, while $P_0-P'=O(\|x-y\|_2^2)$.\qedhere

\subsection{Traversing the light cone}
In this subsection we suppose that $M$ has an indefinite pseudo-Riemannian metric, and thus $S^+M$ and $S^-M$ are non-empty. We will patch together the forms $\phi_{k,r}^+,\phi_{k,r}^-$ defined on these bundles into globally defined generalized forms on the full bundle $\mathbb P_+(TM)$ of oriented tangent lines. 

Let $P$ be a Riemannian structure compatible with the pseudo-Riemannian $Q$. 
Define the field of involutive operators $S\in \mathrm{End}(TM)$ by $Q(u,v)=P(Su,v)$. Define the smooth function $\sigma:TM\setminus\underline0\to[-1,1]$ by $\sigma(v)=Q(v)/P(v)$, and $\epsilon=\mathrm{sign}(\sigma)$.

We may consider $\sigma$ as a function on $\mathbb P_+(TM)$, with $\LC_M=\sigma^{-1}(0)$. Recall that by Proposition \ref{prop_lc_regularity_pseudo_riemann}, $M$ is LC-regular, that is $0$ is a regular value of $\sigma$.

\subsubsection{The almost biorthonormal frame bundle}
Unlike the Riemannian setting where the orthonormal frame bundle provides adequate information, in the pseudo-Riemannian setting we ought to consider a more general frame bundle due to the existence of null directions.
Denote by $\Phi$ the bundle of full oriented flags over $M$, associated to the tangent bundle $TM$. Namely, $\Phi|_x=\{L_0\subset L_1\subset\dots\subset L_{m}= T_xM: \dim L_j=j+1\}$, where all $L_j$ are oriented subspaces.

Given a partition $ J=J_+\cupdot J_-$ of $\{2,\ldots, m\}$ and $\epsilon=\pm 1$, we consider the partition $ K=K_+\cupdot K_-$ of $\{0,\ldots,m\}$ given by $K_\epsilon=J_\epsilon\cup\{0\},K_{-\epsilon}=J_\epsilon\cup\{1\}$. The corresponding $\epsilon_0,\ldots,\epsilon_m$ are then
\begin{displaymath}
\epsilon_0=\epsilon, \quad\epsilon_1=-\epsilon,\quad \epsilon_i=\begin{cases} 1 & \text{ if } i\in J_+,\\-1 & \text{ if } i\in J_-.\end{cases}
\end{displaymath}
Let $\Phi_J\subset \Phi$ denote the open subbundle where $\mathrm{sign}\, Q|_{L_j}=(1+|J_+\cap[2,j]|,1+|J_-\cap[2,j]|)$ for $j\geq 2$. Let $\Phi_J^\pm\subset \Phi_J$ be the open subbundle where $Q|_{L_0}$ is $\pm-$definite, and $Q|_{L_1}$ is non-degenerate. Let us also put $\Phi_J^{\mathrm{nd}}=\Phi_J^+\cup\Phi_J^-$.

Observe that each flag $\psi=(L_0\subset\cdots\subset L_{m})$ in $\Phi$ corresponds to a unique $P$-orthonormal frame $E=(E_0,\dots, E_m)$ such that $E_0,\ldots, E_j$ is a positive basis of $L_j$ for all $j$. If $\psi\in\Phi_J^\pm$, there is also a unique $Q$-orthonormal frame $B=(B_0,\ldots, B_m)$  such that $B_0,\ldots, B_j$ is a positive basis of $L_j$. By definition, this frame $B$ belongs to the frame bundle $F_K^\pm$.

Let $\Psi_J\subset\Phi_{ J}$ denote the sub-bundle that over $p\in M$ has fiber consisting of oriented flags $L_0\subset\dots\subset L_m=T_pM$, such that: $L_0\not\subset V_+\cup V_-$, $L_1=L_0\oplus S(L_0)$ (with the direct sum orientation), and $S(L_j)=L_j$ for all $j\geq 2$. 

Let $\Psi_J^\pm \subset \Psi_J$ denote the subsets where $L_0$ is positive/negative definite. 
Let $\Psi_J^{\mathrm{nd}}=\Psi_J^+ \cup \Psi_J^-$ denote the open subset of $\Psi_J$ where $L_0$ is not light-like. 
Given  a $P$-orthonormal frame $E=(E_0, \dots, E_{m})$ associated to a flag in $\Psi_J^{\mathrm{nd}}$, both $E_i$ and $SE_i$ are $P$-orthogonal to $L_{i-1}=S(L_{i-1})$ for $i\geq 2$ whence 
\begin{equation}\label{eq_SE_i} 
 SE_0  =\sigma E_0+\tau E_1,\qquad
 SE_1  = \tau E_0-\sigma E_1,\qquad
 SE_i =\epsilon_i E_i,\quad i\geq 2,
 \end{equation}
 where $\sigma=Q(E_0),\tau=Q(E_0,E_1)$ satisfy $\sigma^2+\tau^2=1$ and $\tau\geq 0$. The associated $Q$-orthonormal frame  $B=(B_0,\dots,B_m)$ is
 \begin{equation}
 \label{eq_BitoE}
 B_0  = |\sigma|^{-\frac12}E_0,\qquad
 B_1  =-\epsilon\tau|\sigma|^{-\frac12}E_0+|\sigma|^\frac12 E_1, \qquad
 B_j  =E_j,\quad j\geq 2,
\end{equation}
where $\epsilon=\mathrm{sign}(\sigma)$. We thus call $\Psi^{\mathrm{nd}}_J$ the \emph{almost bi-orthonormal frame bundle}.

\subsubsection{Solder and connection forms}
We define forms on $\Psi_J$ as follows.  We have identified (using $P$) the full oriented flag bundle $\Phi$ in $TM$ with the bundle of oriented $P$-orthonormal frames over $M$. On the latter bundle, there are solder forms $\theta^P_i$, connection forms $\omega^P_{i,j}=-\omega^P_{j,i}$ and curvature forms $\Omega^P_{i,j}=-\Omega^P_{j,i}$. These forms may be pulled back to $\Psi_J$ to define forms  $\theta^E_i,\omega^E_{i,j},\Omega^E_{i,j}\in\Omega(\Psi_J)$. 

Similarly, using $Q$ instead of $P$, we may identify $\Phi_J^\epsilon$ with $F_K^\epsilon$. Let  $\theta^Q_i,\omega^Q_{i,j},\Omega^Q_{i,j}\in \Omega(\Phi_J^\mathrm{nd})$ correspond to the solder, connection and curvature forms. The restrictions of these forms to $\Psi_J^\mathrm{nd}$ will we denoted  $\theta^B_i,\omega^B_{i,j},\Omega^B_{i,j}$. We will also denote $\tilde\omega_{i,j}^B=\epsilon_j \omega_{i,j}^B,\tilde\Omega_{i,j}^B=\epsilon_j \Omega_{i,j}^B$.

\begin{Proposition}\label{prop_solder}
The solder forms with respect to $P$ and $Q$ are related by
\begin{equation*}
\theta_0^B  = |\sigma|^{\frac12}\theta_0^E+\epsilon\tau|\sigma|^{-\frac12}\theta_1^E,\qquad
\theta_1^B  = |\sigma|^{-\frac12}\theta_1^E,\qquad
\theta_j^B  =\theta_j^E,\quad j\geq 2.
\end{equation*}
\end{Proposition}

\proof
We may rewrite equations \eqref{eq_BitoE} as 
$B=EC$ where $B=(B_0,\ldots,B_n)$, $E=(E_0,\ldots, E_n)$ and
\begin{displaymath}
C:=\diag\left( \left(\begin{matrix} |\sigma|^{-\frac12} & -\epsilon\tau|\sigma|^{-\frac12} \\ 0 & |\sigma|^{\frac12}\end{matrix}\right), I_{m-1}\right). 
\end{displaymath}

Then the column vectors $\theta^B=(\theta_0^B,\ldots, \theta_n^B)^T$, $\theta^E=(\theta_0^E,\ldots, \theta_n^E)^T$ of solder forms on $\Psi$ satisfy 
\begin{equation}\label{eq:thetaEtoB}
 \theta^B=C^{-1}\theta^E,
\end{equation}
from which the displayed equations follow.
\endproof

Our next aim is to find the relations between the connection forms for $P$ and $Q$. We write $\sigma=\cos (2\beta)$ and $\tau=\sin(2\beta)$ with $0<\beta<\frac{\pi}{2}$. 

\begin{Proposition}
Modulo $\mathrm{span}(\theta_0^E,\ldots, \theta_n^E)$ we have the following relations. 
\begin{align}
 d\beta & \equiv\omega_{0,1}^E, \label{eq_comparison_connection_forms1}\\
 \omega^B_{1,0} & \equiv|\sigma|^{-1}\omega^E_{1,0}, \label{eq_comparison_connection_forms2}\\
\omega^B_{j,0} & \equiv|\sigma|^{-\frac12}\omega^E_{j,0}, j\geq 2, \label{eq_comparison_connection_forms3}\\
\omega^B_{j,1} & \equiv-|\sigma|^{-\frac12}\epsilon\epsilon_j\omega^E_{j,1}, j\geq 2, \label{eq_comparison_connection_forms4}\\
\omega_{j,1}^E & \equiv\tau^{-1}(\epsilon_j-\sigma)\omega_{j,0}^E,\label{eq_comparison_connection_forms5}\\
\omega^B_{i,j} & \equiv\epsilon_i\epsilon_j\omega_{i,j}^E, i,j\geq 2. \label{eq_comparison_connection_forms6}
\end{align}
It follows that for $i,j\geq 2$,
\begin{align}
\omega^B_{j,1} & \equiv-\epsilon\tau|\sigma|^{-\frac12}\omega^E_{j,0}+|\sigma|^{\frac12}\omega^E_{j,1}, \label{eq_comparison_connection_forms7}\\
\omega_{i,j}^B &\equiv\omega_{i,j}^E=0\quad \mbox{if }\epsilon_i\neq \epsilon_j.\label{eq_different_signs_vanish}
\end{align}
\end{Proposition}

\proof
By the structure equations and \eqref{eq:thetaEtoB}, 
\begin{displaymath}
0=d\theta^B+\omega^B\wedge \theta^B= (dC^{-1}-C^{-1}\omega^E+\omega^BC^{-1})\wedge \theta^E.
\end{displaymath}
It follows by Cartan's lemma \cite[Section 6.3.1]{sternberg12} that
\begin{equation} \label{eq_comparison_connection_forms}
\omega^B\equiv C^{-1}\omega^E C-(dC^{-1}) C. 
\end{equation}

Straightforward computations yield
\begin{displaymath}
C^{-1}\omega^E C=|\sigma|^{-1} \left(\begin{matrix}\epsilon \tau \omega^E_{1,0}& -\omega^E_{1,0} & \epsilon|\sigma|^{\frac12}(\sigma \omega^E_{0, j}+\tau \omega^E_{1,j} )_{j=2}^n \\ 
\omega^E_{1,0} & -\epsilon\tau\omega^E_{1,0} & (|\sigma|^{\frac12}\omega^E_{1, j})_{j=2}^n \\  
-|\sigma|^{\frac12}(\omega_{0,i}^E)_{i\geq 2} & \epsilon|\sigma|^{\frac12}(\tau\omega^E_{0,i}-\sigma\omega^E_{1,i})_{i\geq 2} & |\sigma|(\omega^E_{i,j})_{i,j\geq 2} \end{matrix}  \right), 
\end{displaymath}
and
\begin{displaymath}
 dC^{-1} \cdot C=|\sigma|^{-1} \diag\left( \left(\begin{matrix}-\epsilon\tau & 2 \\0 &\epsilon\tau  \end{matrix}  \right), 0_{n-1} \right)d\beta.
\end{displaymath}

Equations \eqref{eq_comparison_connection_forms1}--\eqref{eq_comparison_connection_forms6} now follow by comparing each entry in the matrix equation \eqref{eq_comparison_connection_forms}; and \eqref{eq_comparison_connection_forms7} follows by combining \eqref{eq_comparison_connection_forms4} and \eqref{eq_comparison_connection_forms5}.

Finally,  assume $i\in J_+, j\in J_-$. By \eqref{eq_comparison_connection_forms6},
\begin{displaymath}
 \omega^B_{i,j}\equiv -\omega^E_{i,j}= \omega^E_{j,i}\equiv -\omega^B_{j,i}.
\end{displaymath}
But $\omega^B_{i,j}=-\epsilon_i\epsilon_j\omega^B_{j,i}=\omega^B_{j,i}$, and thus $\omega^B_{i,j} \equiv 0 $, which is \eqref{eq_different_signs_vanish}
\endproof

\begin{Lemma}
For any $\psi\in \Psi_J$, a basis of  $T_\psi^*\Psi_J$ is given by $(\theta^E_j)_{j=0}^{n}$, $(\omega_{0,j}^E)_{j=1}^{n}$,  $(\omega_{i,j}^E)_{i<j, i,j \in J_+}$, $(\omega_{i,j}^E)_{i<j,i,j \in J_-}$.
\end{Lemma} 

\proof 
The full collection $(\theta_i^E)_{i=0}^{n}$, $(\omega_{i,j}^E)_{i<j}$ is clearly a spanning set. 

For $j\geq 2$, it follows from \eqref{eq_comparison_connection_forms5} that $\omega_{1,j}^E$ belongs to the space spanned by $\{\omega_{0,i}^E,\theta^E_i\}_{i=0}^{n}$. 
Hence, by \eqref{eq_different_signs_vanish}, the displayed elements form a spanning set. To finish the proof, it remains to note that their number coincides with $\dim\Psi_J=(n+1)+n+{p-1 \choose 2}+{q-1 \choose 2}$.
\endproof

\begin{Proposition}\label{prop_connection}
For $i \geq 2$ we have the following relations between forms on $\Psi^{\mathrm{nd}}$, where $\Theta_j$ is the dual of $\theta_j^E$.  
\begin{align*}
 \omega^B_{1,0} & =|\sigma|^{-1}\omega^E_{1,0}\\
  & \quad -\left(|\sigma|^{-1}(1+\sigma^2)\Theta_0(d\beta)+\epsilon\tau\Theta_1(d\beta)\right)\theta_0^E\\
  & \quad -\left(\epsilon\tau\Theta_0(d\beta)+\tau^2|\sigma|^{-1}\Theta_1(d\beta)\right)\theta_1^E \\
  & \quad +\sum_{j\geq 2}\left(\frac{\tau\epsilon}{2}\Theta_1(\omega_{j,1}^E)- |\sigma|^{-1}\Theta_j(d\beta)+\frac12(|\sigma|+\epsilon_j\epsilon)\Theta_0(\omega_{j,1}^E)\right)\theta_j^E,\\
\omega^B_{i,0} & =|\sigma|^{-\frac12}\omega^E_{i,0}\\
 & \quad +\left(\epsilon_i|\sigma|^{-\frac12}\tau( \Theta_i(d\beta)+\Theta_0(\omega_{i,1}^E))\right)\theta_0^E \\
 & \quad +\left(- \epsilon|\sigma|^{\frac12}\epsilon_i\Theta_i(d\beta)-\frac12(\epsilon|\sigma|^{\frac12}\epsilon_i+|\sigma|^{-\frac12} )\Theta_0(\omega^E_{i,1})+\frac12\tau\epsilon_i|\sigma|^{-\frac12}\Theta_1(\omega_{i,1}^E)\right)\theta_1^E\\
 & \quad +\frac{1}{2}|\sigma|^{-\frac12}\sum_{j\geq 2}\left((\epsilon_i\epsilon_j-1)\Theta_0(\omega^E_{i,j}) +\epsilon_i\tau(\Theta_j(\omega^E_{i,1} ) - \Theta_i(\omega^E_{j,1} )) \right) \theta_j^E.
\end{align*}

In particular $\omega^B_{i,0}$ is homogeneous w.r.t. $\sigma$ for all $i\geq 1 $ with degrees 
\begin{displaymath}
 \deg_\sigma\omega_{1,0}^B=1, 
 \quad \deg_\sigma \omega_{i,0}^B=\frac32, \quad i\geq 2. 
\end{displaymath}
\end{Proposition}

\proof
Set $D:=dC^{-1}-C^{-1}\omega^E+\omega^BC^{-1}$. Then for each fixed $i=0,\ldots,n$ we have $\sum_{j=0}^n D_{ij} \wedge \theta_j^E=0$. By Cartan's lemma, there exists some symmetric matrix $R^i$ such that $D_{ij}=\sum_{a=0}^n R^i_{ja}\theta_a^E$. Expanding the coefficients of $D$ in the basis and writing explicitly the symmetry conditions yields the equation 
\begin{equation} \label{eq_symmetry_m_abstract}
 \Theta_a(D_{ib})=\Theta_b(D_{ia}), \quad i,a,b=0,\ldots,n.
\end{equation}

The entries of $D$ are as follows
\begin{align*}
 D_{0,0}&=-\epsilon|\sigma|^{-\frac12}\tau (\omega^E_{1,0}+d\beta), && D_{1,j}=|\sigma|^{-\frac12}\omega_{j,1}^E+\omega_{1,j}^B,
 \\D_{0,1}&=|\sigma|^{-\frac32}((1+\sigma^2)d\beta+\sigma^2\omega^E_{1,0}+|\sigma|\omega^B_{1,0}), &&D_{i,0}=-\omega_{i,0}^E+|\sigma|^{\frac12} \omega_{i,0}^B, 
 \\D_{0,j}&=|\sigma|^{-\frac12}(|\sigma|\omega_{j,0}^E+\epsilon\tau\omega_{j,1}^E+|\sigma|^{\frac12}\omega_{0,j}^B), &&D_{i,1}=|\sigma|^{-\frac12}(-|\sigma|^{\frac12}\omega^E_{i, 1}+\epsilon\tau\omega_{i,0}^B+\omega_{i,1}^B), 
 \\D_{1,0}&=-|\sigma|^{-\frac12}\omega_{1,0}^E+ |\sigma|^{\frac12}\omega^B_{1,0},&&D_{i,j}=-\omega_{i,j}^E+\omega_{i,j}^B.
\\ D_{1,1}&=\epsilon\tau|\sigma|^{-\frac32}( d\beta + |\sigma|\omega_{1,0}^B),
\end{align*}

In the following, let $i,j \geq 2$. Plugging particular triples $(i,a,b)$ into \eqref{eq_symmetry_m_abstract} we obtain the following equations
\begin{align*}
(0,0,1):\quad \Theta_0(\omega^B_{1,0}) & =-|\sigma|^{-1}(1+\sigma^2)\Theta_0(d\beta)-\epsilon\tau\Theta_1(d\beta),\\
(0,0,j):\quad \Theta_0(\omega_{0,j}^B) & =-\epsilon|\sigma|^{-\frac12}\tau(\Theta_j(d\beta)+\Theta_0(\omega_{j,1}^E)),\\
(1,0,1):\quad  \Theta_1(\omega^B_{1,0}) & =-\epsilon\tau\Theta_0(d\beta)-\tau^2|\sigma|^{-1}\Theta_1(d\beta).
\end{align*}

From the equations
\begin{align*}
 (0,i,j):\quad & \epsilon|\sigma|^{\frac12}(\Theta_j(\omega_{0,i}^B)-\Theta_i(\omega_{0,j}^B)) =\tau(\Theta_i(\omega_{j,1}^E)-\Theta_j(\omega_{i,1}^E)) ,\\
 (i,0,j):\quad & \Theta_j(\omega^B_{i,0}) =-|\sigma|^{-\frac12}\Theta_0(\omega^E_{i,j})+|\sigma|^{-\frac12  }\Theta_0(\omega^B_{i,j}) ,
\end{align*}
together with $\omega^B_{i,j}=-\epsilon_i\epsilon_j\omega^B_{j,i}$, we get 
\begin{align*}
\Theta_i(\omega^B_{j,0}) & =\frac{1}{2}|\sigma|^{-\frac12}\left((1-\epsilon_i\epsilon_j)\Theta_0(\omega^E_{i,j}) +\epsilon_j\tau(\Theta_i(\omega^E_{j,1} ) - \Theta_j(\omega^E_{i,1} )) \right).
\end{align*}

Then from the three equations
\begin{align*}
(0,1,j):\quad& \Theta_j(\omega^B_{1,0})-|\sigma|^{\frac12}\Theta_1(\omega^B_{0,j})=\tau\epsilon\Theta_1(\omega^E_{j,1})-(1+\sigma^2)|\sigma|^{-1}\Theta_j(d\beta),\\
(1,0,j):\quad& \Theta_j(\omega^B_{1,0})-|\sigma|^{-\frac12}\Theta_0(\omega^B_{1,j})=-|\sigma|^{-1}\Theta_0(\omega^E_{1,j}),\\
(j,0,1):\quad & \Theta_1(\omega^B_{j,0})-|\sigma|^{-1}\Theta_0(\omega^B_{j,1})=-|\sigma|^{-\frac12}\Theta_0(\omega^E_{j,1})+\tau\sigma^{-1}\Theta_0(\omega^B_{j,0}),
\end{align*}
 we get
\begin{align*} 
\Theta_1(\omega^B_{0,j}) & =|\sigma|^{\frac12}\Theta_j(d\beta)+\frac12(\epsilon|\sigma|^{-\frac12}\epsilon_j+|\sigma|^{\frac12} )\Theta_0(\omega^E_{j,1})-\frac12\tau\epsilon|\sigma|^{-\frac12}\Theta_1(\omega_{j,1}^E),\\
\Theta_j(\omega_{1,0}^B) & =\frac{\tau\epsilon}{2}\Theta_1(\omega_{j,1}^E)-|\sigma|^{-1}\Theta_j(d\beta)+\frac12(|\sigma|+\epsilon_j\epsilon)\Theta_0(\omega_{j,1}^E).
\end{align*}
The statement follows.\endproof

\subsubsection{Curvature forms}

\begin{Lemma}\label{lemma_curvature_extensions}
The curvature forms $\epsilon\Omega_{1,0}^B,|\sigma|^\frac12\Omega^B_{i,0}, \epsilon |\sigma|^\frac12\Omega^B_{i,1}$ and $\Omega^B_{i,j}$ with $i,j\geq 2$ admit  a smooth extension to $\Psi_J$. In particular, all the curvature forms $\Omega_{i, j}^B$ are homogeneous w.r.t. $\sigma$, and $\deg_\sigma \Omega_{1,0}^B=1$,  $\deg_\sigma \Omega_{j,0}^B=\frac32$, $\deg_\sigma \Omega_{j,1}^B=\frac12$, and $\deg_\sigma\Omega_{i,j}^B=0$ for $i,j\geq 2$.
\end{Lemma}
\proof
Let $R=[\nabla,\nabla]-\nabla_{[\,,\,]}$ denote the Riemannian curvature $(3,1)$-tensor of $Q$. Take vector fields $\tilde u, \tilde v$ on $\Psi^{\mathrm{nd}}_J$, and let $u,v$ be their respective projections to $M$. Then \eqref{eq:convention} gives
\begin{equation}\label{eq_expression_Omega}
 \Omega_{i,j}^B(\tilde u,\tilde v)=\epsilon_i Q(R(u,v)B_j, B_i). 
\end{equation} 

By \eqref{eq_BitoE}, it follows for $i,j\geq 2$ that 
\begin{displaymath}
 \Omega^B_{i,j}(\tilde u,\tilde v) = \epsilon_i Q(R(B_j, B_i)u, v)= \epsilon_i Q(R(E_j, E_i)u, v),
\end{displaymath}
which clearly extends smoothly to $\Psi_J$.

Similarly, by \eqref{eq_BitoE},
\begin{align*}	 
   \Omega_{1,0}^B(\tilde u,\tilde v)&=-\epsilon Q(R(u,v)E_0,E_1),\\
   \Omega_{j,0}^B(\tilde u,\tilde v)&=\epsilon_j|\sigma|^{-\frac12} Q(R(u ,v )E_0,E_j),\\
   \Omega^B_{j,1}(\tilde u,\tilde v)&=\epsilon_j(-\epsilon \tau |\sigma|^{-\frac12}Q(R(u,v)E_0,E_j)+|\sigma|^{\frac12}Q(R(u,v)E_1,E_j)).
\end{align*}
The statement follows.	
\endproof

 Recall that we denote by $\omega^P,\omega^Q$ the connection forms on the flag bundle $\Phi_J$, while the restrictions to the subbundle $\Psi_J$ were denoted $\omega^E,\omega^B$.

\begin{Lemma}\label{lemma:flat_space}
Consider $M=\R^{p,q}=\R^{n+1}$ with the standard forms $Q, P$, and fix $\psi\in\Psi_J^{\mathrm{nd}}$. The connection forms $\omega^P_{i,j}$, $\omega^Q_{i,j}$ satisfy the following relations at $\psi$.
\begin{enumerate}
 \item $\omega^Q_{1,0}=|\sigma|^{-1}\omega^P_{1,0}.$
 \item $\omega^Q_{j,0}=|\sigma|^{-\frac12}\omega^P_{j,0}, \quad j\geq 2$.
 \item $\omega^Q_{j,1}=-\epsilon\tau|\sigma|^{-\frac12}\omega^P_{j,0}+|\sigma|^{\frac12}\omega^P_{j,1},\quad j\geq 2$.
 \item $\omega^Q_{i,j}=\omega_{i,j}^P,\quad i>j\geq 2$. 
\end{enumerate}
\end{Lemma}

\proof
Let us prove iii) and iv), the other relations being similar. Recall that $\nabla E_i=\omega_{k,i}^P E_k$ while $\nabla B_i=\omega_{k,i}^Q B_k$. For $i\geq 2$, using \eqref{eq_BitoE} we get
\begin{align*}
\omega_{i,1}^Q &=\epsilon_i Q(E_i,\nabla(-\epsilon\tau|\sigma|^{-\frac12}E_0+|\sigma|^{\frac12}E_1 ))\\
 & =\epsilon_i (-\epsilon\tau|\sigma|^{-\frac12}Q(E_i,\nabla E_0)+|\sigma|^{\frac12}Q(E_i,\nabla E_1) )\\
 & =-\epsilon\tau |\sigma|^{-\frac12} \omega_{i,0}^P + |\sigma|^{\frac12}\omega_{i,1}^P,
\end{align*}
which is iii). To show iv), recall that $B_j\in \mathrm{span}(E_0,\ldots, E_j)$ on $\Phi_J$ and $B_j=E_j$ at $\psi$. Therefore, for $i>j\geq 2$,
\begin{displaymath}
 \omega_{i,j}^Q =\epsilon_i Q(E_i,\nabla B_j)=\epsilon_i Q(E_i,  \nabla E_j)=\omega_{i,j}^P.\qedhere
\end{displaymath}

\begin{Lemma}\label{lemma_gauss_eqs}
Let $M^{p',q'}\subset \R^{p,q}$ be a pseudo-Riemannian submanifold. Let $\Phi_M \subset \Phi_J$ be the subset consisting of flags  $\{L_i\}_{i=0}^n$ of $\R^{p,q}$ such that there exists some $x\in M$ with $L_m=T_xM$. For $\psi\in \Phi_M \cap \Psi_J$ and $2\leq  i,j\leq m$, the following relations hold at $T_\psi\Phi_M$  
\begin{align*}
\tilde\Omega_{0,1}^{Q} &=-\sum_{r=m+1}^{n} \epsilon_r \omega_{0,r}^{P}\wedge \omega_{1,r}^{P},\\
\tilde\Omega_{0,j}^{Q}&=\epsilon\epsilon_j|\sigma|^{-\frac12}\sum_{r=m+1}^{n} \epsilon_r \omega_{0,r}^{P}\wedge \omega_{j,r}^{P},\\
\tilde\Omega_{1,j}^{Q}&=-\epsilon_j|\sigma|^{-\frac12} \sum_{r=m+1}^{n} \epsilon_r (-\tau\omega_{0,r}^{P}+\sigma\omega_{1,r}^{P})\wedge \omega_{j,r}^{P},\\
\widetilde \Omega_{i,j}^{Q}&=\epsilon_i\epsilon_j \sum_{r=m+1}^{n} \epsilon_r \omega_{i,r}^{P}\wedge\omega_{j,r}^{P}.
\end{align*} 
\end{Lemma}

\begin{proof} 
Suppose first  $\psi \in \Phi_M \cap \Psi_J^{\mathrm{nd}}$. By the structure equations on $M$ and $\R^{p,q}$,
 \begin{align*}
  d\tilde\omega_{i,j}^{Q}&= \sum_{r=0}^m \epsilon_r \tilde\omega_{i,r}^{Q}\wedge\tilde\omega_{j,r}^{Q}+\tilde \Omega_{i,j}^{Q}= \sum_{r=0}^n \epsilon_r \tilde\omega_{i,r}^{Q}\wedge\tilde\omega_{j,r}^{Q}
 \end{align*}
 for $0\leq i,j\leq m$, which yields
 \begin{equation}\label{eq:gauss_Q}
  \tilde \Omega_{i,j}^{Q} =\epsilon_i\epsilon_j\sum_{r=m+1}^n  \epsilon_r \omega_{r,i}^{Q}\wedge\omega_{r,j}^{Q}.
 \end{equation}
The stated equations follow using Lemma \ref{lemma:flat_space}. For $\psi\notin \Psi_J^{\mathrm{nd}}$ the equations follow by continuity using Lemma \ref{lemma_curvature_extensions}
\end{proof}

\subsubsection{Globally defined forms}\label{sub:globablly_defined_forms}

Assuming $p,q>0$, we are now able to patch the forms $\phi_{k,r}^\pm\in \Omega(S^\pm M)$ from Lemma \ref{lemma_def_phi} across the light cone. 
Let $\phi_{k,r}\in\Omega(\mathbb P_+(TM)\setminus \LC_M)$ be the smooth form that coincides with $\phi_{k,r}^+, \phi_{k,r}^-$ on $S^+M,S^-M$ respectively, following the identification $S^+M\cup S^-M=\mathbb P_+(TM)\setminus \LC_M$ .

\begin{Proposition}\label{prop:phi_degree}
The form $\phi_{k,r}\in\Omega(\mathbb P_+(TM)\setminus \LC_M)$ is homogeneous w.r.t. $\sigma$, with $\deg_\sigma\phi_{k,r}=\frac {m-k-1}{2}$.
\end{Proposition}

\proof 
Denote $U:=\mathbb P_+(TM)\setminus (\mathbb P_+(V_+)\cup \mathbb P_+(V_-))$.  Restrict $\phi_{k,r}$ to $U$, and identify with its pull-back by $\pi_0:\Psi_J \to U$. From Proposition \ref{prop_solder}
\begin{displaymath}
\deg_\sigma\theta^B_0 =\frac12, \quad \deg_\sigma\theta^B_1 =\frac32, \quad \deg_\sigma\theta^B_i =0, i\geq 2.
\end{displaymath}
By Proposition \ref{prop_connection} and Lemma \ref{lemma_curvature_extensions}, it holds for $i,j\geq 2$ that 
\begin{align*}
 \deg_\sigma\tilde\omega_{1,0}^B  &=\deg_\sigma\omega_{1,0}^B+1=0, \qquad
\hspace{30pt}  \deg_\sigma\tilde\omega_{ i,0}^B  =\deg_\sigma\omega_{ i,0}^B+1= \frac12,\\
 \deg_\sigma \tilde{\Omega}^B_{i,1} &=\frac32,\qquad
\hspace{97pt} \deg_\sigma \tilde{\Omega}^B_{i,j} =0.
\end{align*}
Since $\tilde\omega_{i,j},\tilde\Omega_{i,j}$ are antisymmetric, these $\sigma$-degrees are independent of the order of indices. Using these values, and Proposition \ref{prop:sigma}, it is straightforward to compute the $\sigma$-degree of each term in \eqref{eq_def_phi}. Those terms containing a factor $\theta_1^B$ or a factor $\tilde\Omega_{1,i}^B$ have $\sigma$-degree $\frac32+\frac12(m-k)\equiv \frac{m-k-1}2$. The remaining terms are multiples of $\tilde\omega_{0,1}^B$ and have $\sigma$-degree $\frac{m-k-1}2$. Examining the definition of $\phi_{k,r}$, we conclude it is homogeneous w.r.t. $\pi_0^*\sigma$ of degree $\frac{m-k-1}2$ on $\Psi_J \setminus \pi_0^{-1}\LC_M$, and thus by Lemma \ref{lem:sigma_pull_back} it is homogeneous w.r.t. $\sigma$ over $U\setminus \LC_M$ of the same degree. As $\LC_M\subset U$, it follows by Lemma \ref{lem:ignore_outskirts} that $\phi_{k,r}\in\mathcal H_{\frac{m-k-1}2}(\mathbb P_+(TM),\sigma)$.
\endproof

We now use Definition \ref{def:meromorphic_families} to define certain meromorphic families of generalized forms on $\mathbb P_+(TM)$. By Propositions \ref{prop:independence_meromorphic} and \ref{prop:phi_degree}, the following are well-defined.

\begin{Definition} \label{def_phi_kri}
 For $(m-k)$ even, let $\phi_{k,r}^0,\phi_{k,r}^1 \in \Omega_{-\infty}^{m}(\mathbb P_+(TM))$ be
\begin{align*}
 \phi_{k,r}^0  =\sigma_+^s\phi_{k,r}|_{s=0},\qquad
 \phi_{k,r}^1  =(-1)^{\frac{m-k}2}\sigma_-^s\phi_{k,r}|_{s=0}.
\end{align*}
For $(m-k)$ odd, let $\phi_{k,r}^0,\phi_{k,r}^1\in\Omega_{-\infty}^{m}(\mathbb P_+(TM))$ be
\begin{align*}
 \phi_{k,r}^0  =\epsilon^{\frac{m-1-k}2}|\sigma|^s \phi_{k,r}|_{s=0},\qquad
 \phi_{k,r}^1  =\frac{\pi}2 \mathrm{Res}_{s=0}\epsilon^{\frac{m+1-k}2}|\sigma|^s \phi_{k,r}.
\end{align*}
\end{Definition}

\begin{Lemma}
The generalized forms $\phi_{k,r}^i$ are independent of the choice of $P$.
\end{Lemma}
\proof
We will consider the case of even $(m-k)$ and $i=0$, as all other cases can be treated similarly. 
Let $P_1$, $P_2$ be two Riemannian structures. We then get two meromorphic families that are given by $\phi_j(s)=(\sigma_j)_+^s\psi_j$, $j=1,2$, where $\psi_j=\phi_{k,r}^i$ with the corresponding $P_j$. Outside of the light cone, $\psi_1=\psi_2$ by Lemma \ref{lemma_def_phi}.  

Define $r(L_0):=\frac{\sigma_2}{\sigma_1}=\frac{P_1|_{L_0}}{P_2|_{L_0}}>0$, which is a smooth function on $\mathbb P_+(TM)$. Thus outside of the light cone, $\phi_2(s)=r(L_0)^s  \phi_1(s)$ holds at all $s$. But for large $\Re(s)$, both families are continuous, hence $\phi_2(s)=r(L_0)^s  \phi_1(s)$ holds on $\mathbb P_+(TM)$ for all large $\Re(s)$. The conclusion follows by uniqueness of meromorphic extension.
\endproof

In case $M$ has a definite metric we extend the previous definitions as follows. When $M^{p,0}$ is Riemannian, the smooth form $\phi_{k,r}^+$ is defined on the whole $\mathbb P_+(TM)$ and we put  $\phi_{k,r}^0=\phi_{k,r}^+,\phi_{k,r}^1=0$. When $M^{0,q}$ is negative definite, the form $\phi_{k,r}^-$ is smooth on $\mathbb P_+(TM)$ and we put $\phi_{k,r}^i=(-1)^{m-1-k}\chi_i^{-\frac{m-1-k}2}(-1)\phi_{k,r}^-$.

\subsection{Lipschitz-Killing curvature measures}

\begin{Definition} \label{def_lk_curvatures}\mbox{}
\\\emph{i)} The \emph{Lipschitz-Killing forms} $\kappa_k\in \Omega^{m+1}(M)$, $\lambda_k \in  \Omega^{m}_{-\infty}(\mathbb P_+(TM))$ are defined by
\begin{align*}
\kappa_k & =\begin{cases}
\frac{\i^q}{k!\left(\frac{m+1-k}2\right)!(4\pi)^{\frac{m+1-k}2}}\psi_{m+1,\frac{m+1-k}2}  & \text{if }m+1-k \text{ is even},\\
0 & \text{if }m+1-k \text{ is odd},
\end{cases}\\
\lambda_k & ={\i^q}\sum_{\nu=0}^{\left\lfloor \frac{m-k}{2}\right\rfloor}  \frac{\phi^0_{k+2\nu,\nu}-\i \phi^1_{k+2\nu,\nu}}{k!\nu!(4\pi)^{\nu}(m+1-k-2\nu)!\omega_{m+1-k-2\nu}}.
\end{align*}
	\\\emph{ii)} The \emph{Lipschitz-Killing curvature measures} are the complex valued generalized curvature measures
\begin{displaymath}
\Lambda_k:=[\kappa_k, \lambda_k] \in \mathcal{C}^{-\infty}(M).
\end{displaymath} 
Define $\widetilde{\mathcal {LK}}(M)\subset \mathcal C^{-\infty}(M)$ to be the $\C$-span of all $\mathrm{Re}\Lambda_k,\mathrm{Im}\Lambda_k$ as $k\geq 0$.
	\\\emph{iii)} The \emph{intrinsic volumes} are the complex valued generalized valuations 
\begin{displaymath}
\mu_k:=\mathrm{glob}(\Lambda_k) \in \mathcal{V}^{-\infty}(M).
\end{displaymath}
The $\C$-span of all $\mathrm{Re} \mu_k, \mathrm{Im}\mu_k$ is denoted $\mathcal{LK}(M)\subset\mathcal V^{-\infty}(M)$.
\end{Definition}

\begin{Remark}
	When $(M,g)$ is Riemannian, we have $\mathbb P_M=SM^+$ and $\phi_{k,r}^0=\phi_{k,r}^+, \phi_{k,r}^1=0$. Hence, in this case the Lipschitz-Killing curvature measures and the intrinsic volumes are real valued and coincide with the classical notions.  
\end{Remark}

\begin{Remark}\label{rm:kappas}
The form $\kappa_k$ is a smooth form of top degree on $M$, hence we may define for all relatively compact Borel sets $U\subset M$ that $\Lambda_k(M,U):=\int_U \kappa_k$. It is straightforward to check that, if $m-k+1$ is even,
	\begin{displaymath}
	\kappa_k=\frac{\i^q}{(\frac{m-k+1}2)!(8\pi)^{\frac{m-k+1}2}} \sum_{\alpha,\tau}  \sgn(\tau) R_{\alpha_1,\alpha_2}^{\alpha_{\tau_1},\alpha_{\tau_2}}\cdots R_{\alpha_{m-k},\alpha_{m-k+1}}^{\alpha_{\tau_{m-k}},\alpha_{\tau_{m-k+1}}} \vol_X,
\end{displaymath}
where the sum runs over $1\leq \alpha_1,\ldots,\alpha_{m-k+1}\leq m+1$ and $\tau\in S_{m-k+1}$, and 
$R_{i,j}^{r,s}=\epsilon_s R_{i,j,s}^r$ are the components of the curvature tensor $R$ (in a $Q$-orthonormal basis) after raising one index.
\end{Remark}

It will be convenient to consider the generalized curvature measures $C_{k,r}^i \in \mathcal{C}_k^{-\infty}(M)$, for $0\leq 2r\leq  k\leq m$, $0\leq 2s\leq m+1$ and $i\in\mathbb Z_2$, defined by 
\begin{align} \label{eq:Ckr}
&C_{k,r}^i  :=  (-1)^{i}\frac{\omega_k}{\pi^k(m+1-k)!\omega_{m+1-k}} [0,\phi_{k,r}^i],\\
&C_{m+1,s}^0 :=  \frac{\omega_{m+1}}{\pi^{m+1}} [\psi_{m+1,s},0], \qquad 
 C_{m+1,s}^1 :=0. \label{eq:C_top}
\end{align}
In this notation
\begin{equation}\label{eq:C_to_Lambda}
 \Lambda_k=\i^q \frac{\pi^k}{k!\omega_k} \sum_{\nu=0}^{\frac{m-k+1}{2}}\frac{1}{4^{\nu}}{\frac k 2+\nu\choose \nu} (C^{0}_{k+2\nu,\nu}+\i C^1_{k+2\nu,\nu}).
\end{equation}

\begin{Remark}\label{rmk:definition_of_measure_on_projective_space}
	When $M=\R^{p,q}$, $\mathbb P_+(TM)=\R^{p,q}\times\mathbb P_+(\R^{p,q})$, and since $\lambda_0\in\Omega_{-\infty}^{0,p+q-1}(\mathbb P_+(TM))$ is isometry-invariant, it can be considered as a generalized form of top degree on $\mathbb P_+(\R^{p,q})$, that is $\lambda_0\in\mathcal M^{-\infty}(\mathbb P_+(\R^{p,q}))^{\OO(p,q)}$. It is not hard to show that $\Re(\lambda_0), \Image(\lambda_0)$ span the space of all such generalized measures.
\end{Remark}

\subsection{Wave front sets}
It will be convenient to make the statements in cotangent space, while carrying out the proofs in the tangent space.
\begin{Lemma}\label{lem:fiber_is_transversal}
For pseudo-Riemannian $M$, $\WF(\phi_{k,r}^{i,M})\subset N^*(\LC^*_M)$. 
\end{Lemma}

\proof
Follows at once from eq. \eqref{eq:WF_sigma}, as $\LC_M^*=\{\sigma=0\}$.
\endproof

\begin{Corollary}\label{cor:wave_front_set_of_LK}
  For all $k,r,i$ we have $C_{k,r}^{i, M}\in \mathcal C_{\emptyset, N^*(\LC_M^*)}^{-\infty}(M)$.
\end{Corollary}

\proof
As $C_{k,r}^{i,M}$ is represented by multiples of $\psi_{m+1,r}$ and $\phi_{k,r}^{i,M}$, it remains to note that the former is smooth by Lemma \ref{lem:interior_smooth_form}. 
\endproof

\begin{Corollary}\label{cor:LK_restrictions_well_def}
Let $e:(X,g)\looparrowright M^{p,q}$ be an isometrically immersed pseudo-Riemannian, or more generally LC-regular, manifold. Then $e^*C_{k,r}^{i,M}$ is well-defined.
\end{Corollary}

\proof
This follows from Corollary \ref{cor:wave_front_set_of_LK} and Proposition \ref{prop:nondegenerate_restriction}.
\endproof

We need more information on the wave front set for the restriction computation.
\begin{Lemma}\label{lem:can_restrict_fiberwise}
Fix a pair of pseudo-Riemannian manifolds $M\subset W$, and $x\in M$.
\\\emph{i)} Let $\theta:N^*M\hookrightarrow \mathbb P_W$ be the inclusion. Then $\WF(\theta^*\phi_{k,r}^{i,W})\cap N^*(N^*_xM)=\emptyset$.
\\\emph{ii)}  Let $Z$ be the oriented blow-up of $\mathbb P_W\times_W M$ along $N^*M$, and $\alpha:Z\to \mathbb P_W$ is the composition of the blow-up map $b:Z\to \mathbb P_W\times_W M$ with the inclusion $j:\mathbb P_W\times_W M\hookrightarrow\mathbb P_W$ as in section \ref{sec:restriction_of_curvature_measures}. Let $\pi:Z\to M$ be the obvious projection. 
Then $\WF(\alpha^*\phi_{k,r}^{i,W})\cap N^*(\pi^{-1}x)=\emptyset$.

\end{Lemma}

\proof \noindent 
i) We have $\WF(\theta^*\phi_{k,r}^{i,W})\subset \theta^* \WF(\phi_{k,r}^{i,W})\subset N^*(\LC_W\cap N^Q M)$. 
	Note that $\LC_W\cap N^Q M$ is a hypersurface in $N^QM$, while $N_x^QM\subset T_xW$ is a non-degenerate subspace. It follows that $N_x^QM\pitchfork(\LC_W\cap N^Q M)$ in $N^QM$, implying the statement.
	
ii)  Denote $X|_M:=X\times_W M$ for $X\subset SW:=\mathbb P_+(TW)$. Note that $\WF(j^*\phi_{k,r}^{i,W})\subset N^*(\LC_W\cap SW|_M)$. Hence $\WF(\alpha^*\phi_{k,r}^{i,W})\subset b^*N^*(\LC_W|_M)$. 
	
	Consider $z\in \pi^{-1}x$. If $b(z)\notin N^*M$, we conclude as in i). Assume now $b(z)\in N^*M$. We ought to check 
	$(d_zb)^*(N_{b(z)}^*\LC_W|_M)\cap N^*(\pi^{-1}x)=\emptyset$, equivalently that $d_zb(T_z(\pi^{-1}x))+T_{b(z)}(\LC_W|_M)=T_{b(z)}SW|_M$. Now $d_zb(T_z(\pi^{-1}x))$ contains $T_{b(z)}N^Q_xM$, while $\LC_W|_M\subset SW|_M$ is a hypersurface. As $N^Q_xM$ is non-degenerate, the conclusion follows as before.

\endproof

\section{Restricting Lipschitz-Killing curvature measures}
\label{sec_restriction}

We proceed to establish Weyl's principle for pseudo-Riemannian manifolds.

\begin{Theorem}\label{thm:weyl}
Let $e:M^{p',q'}\looparrowright N^{p,q}$ be an isometric immersion of pseudo-Riemannian manifolds and let $\Lambda_k^N, \Lambda_k^M$ be the Lipschitz-Killing curvature measures of $N$ and $M$ respectively. Then the restriction $e^*\Lambda_k^N$ exists, and $e^*\Lambda_k^{N}=\Lambda_k^{M}$.
\end{Theorem}

\subsection{Fiberwise evaluation of generalized forms.}
The knowledge of the wave front set of $\phi_{k,r}^i$ will allow us to evaluate the restriction fiberwise, as follows.
\begin{Lemma} \label{lem:pointwise}
 Fix a manifold $M$. Let $F\hookrightarrow S\twoheadrightarrow M$ be a fiber bundle, and $E\twoheadrightarrow S$ some vector bundle. For $x\in M$, let $i_x\colon S|_x \hookrightarrow S$ be the fiber inclusion.
\\\emph{i)}  Take $\omega\in\Gamma^{-\infty}(S,E)$ such that $\WF(\omega)\cap N^*(S|_x)=\emptyset\quad\forall x\in M$. Set
\begin{displaymath}
\omega_x:=i_x^*\omega\in \Gamma^{-\infty}(S|_x,i_x^*E).
\end{displaymath}
 If $\omega_x=0$ for all $x\in M$, then $\omega=0$.
 \\\emph{ii)} Let  $W^r$ be a compact manifold (possibly with boundary). Set $\tilde i_x=i_x\times\id\colon S|_x \times W \hookrightarrow  S\times W$, and $p\colon S\times W\to S$, $p_x\colon S|_x\times W\to S|_x$ the projections.  
 Given $\omega\in \Gamma^{-\infty}(S\times W,E \boxtimes \Dens(TW))$ such that $\WF(\omega)\cap N^*(S|_x\times W)=\emptyset$,   we have  $\WF(p_*\omega) \cap N^*(S|_x)=\emptyset$ for all $x\in S$, and
 \begin{displaymath}
 (p_x)_*(\tilde i_x^*\omega)=(p_*\omega)_x\in \Gamma^{-\infty}(S|_x,E).
 \end{displaymath}

\end{Lemma}  

\proof \noindent
i) We may assume $E=S\times \R$, and furthermore by the sheaf property of distributions we may work locally and assume $S=M\times F$, $M=\R^m$, $F=\R^f$, $i_x(\xi)=(x,\xi)$. We use the Lebesgue measures on $M$, $F$ to identify functions and measures.
Let $\eta\in  C^\infty_c(S)$ be a test function supported in $M_\eta\times F_\eta$. We ought to show $\langle \omega,\eta\rangle=0$.

Let $\omega^j\to\omega$, $j\to \infty$ be a smooth approximating sequence in the normal topology of $C^{-\infty}_{\WF(\omega)}(S)$. By Fubini's theorem,
\begin{equation}\label{eq:Fubini}
\langle \omega^j, \eta\rangle=\int_M \int_{F}\omega^j(x,\xi)\eta(x,\xi)d\xi \ dx.
\end{equation} 
As $j\to\infty$, we have $\langle \omega^j, \eta\rangle \to \langle \omega, \eta\rangle$ by definiton, and similarly for every $x$ it holds that $\int_{F}\omega^j(x,\xi)\eta(x,\xi)d\xi\to \langle \omega_x,\eta(x,\bullet)\rangle$ by the continuity of $i_x^*$.

Let us show the latter convergence is uniform in $x\in M_\eta$. Write
\begin{displaymath} \int_{F}\omega^j(x,\xi)\eta(x,\xi)d\xi- \langle \omega_x,\eta(x,\bullet)\rangle_F=\langle i_x^*(\omega^j-\omega), \eta\rangle_F=\langle \omega^j-\omega, \eta\cdot\delta_x\rangle.
 \end{displaymath}
 
Writing $\Lambda=\cup_{x\in M}N^*S|_x\subset T^*S$, it is easy to see that $\{\eta\cdot \delta_x\}_{x\in M_\eta}\subset C^{-\infty}_{\Lambda}(S)$ is a bounded set. It follows by the hypocontinuity of the pairing in the normal topology that $\sup_{x\in M_\eta}|\langle \omega^j-\omega, \eta\cdot\delta_x\rangle|\to 0$ as $j\to\infty$.
Thus we can interchange in \eqref{eq:Fubini} the limit as $j\to\infty$ with integration over $M$, concluding the proof.

ii)
Let us check that $\WF(p_*\omega)\cap N^*(S|_x)=\emptyset$. This follows at once from wave front set calculus: if $\xi\in\WF (p_*\omega)\cap N^*(S|_x)$, we have $\eta=(d_{z, w}p)^*\xi\in \WF(\omega)$ for some $(z,w)\in S\times W$. But then for any $v\in T_{z,w}(S|_x\times W)$ it holds that $v':=d_{z,w}p (v)\in T(S|_x)$ and hence $\eta(v)=\xi(v')=0$. That is $\eta\in N^*(S|_x\times W)$, contradicting the assumption.  

Now for smooth $\omega$ clearly $ i_x^*\circ p(\omega)=p_F\circ \tilde i_x(\omega)$. The general case follows by continuity of the pull-back and push-forward operations. \qedhere

\endproof 

\subsection{Quadratically compatible points}

For further use, we rewrite the forms $\phi_{k,r}^i$ in Riemannian terms at points of quadratic compatibility.

\begin{Lemma}\label{lem:thetas_vanish}
Let the Riemannian metric $P$ be quadratically compatible with $Q$ at $x$, and $\psi\in \Psi|_x$. It then holds at $\psi$ that $\Theta_i(d\sigma)=0$, and $\Theta_i(\omega^E_{j,k})=0$ for all $i,j,k$.
\end{Lemma}

\proof
Denote by $H_\phi^P\subset T_\phi\Phi$ the $P$-horizontal subspace at an arbitrary point $\phi\in \Psi$, and similarly for $Q$. Clearly $H^P_\phi\cap H^Q_\phi\subset T_\phi \Psi$.  

Note that due to the quadratic compatibility, $H_\psi^P=H_\psi^Q$. This is because it is determined by the common zeros of the connection forms, which only depend on the first derivatives of the metrics at $\psi$. Denoting this subspace by $H_\psi$, we have $H_\psi\subset T_\psi\Psi$. If there is a linear relation of the form 
\begin{displaymath}
 \sum a_{ij}\omega_{i,j}^E|_\Psi =\sum B_i\theta_i^E|_\Psi,
\end{displaymath}
we can restrict it to $H_\psi$, where all connection forms vanish, while the solder form $\theta_i^E$ are linearly independent. Hence $B_i=0$, as required.

Now $\sigma=\frac{Q(E_0)}{P(E_0)}$. Working with the distinguished coordinates $x_j$, together with the associated coordinates $\xi_j$ on the tangent space, it is clear that $\frac{\partial \sigma}{\partial x_j}(x)=0$, that is $d_\psi\sigma\in\Span(\omega^E_{0, i})_{i=1}^n$.
\endproof

Recalling that $\tilde \omega_{i,j}=- \epsilon_i\omega_{j,i}$, it follows from Propositions \ref{prop_solder} and \ref{prop_connection} that, at a point $z_0$ of quadratic compatibility,
\begin{equation} \label{eq_phi_in_terms_of_rho}
 (\phi_{k,r})_{z_0}=\epsilon^{m-k}|\sigma|^{-\frac{m+1-k}2}(\rho_{k,r})_{z_0},
\end{equation}
where the smooth form $\rho_{k,r}$ is given by
\begin{align}
 \rho_{k,r} & = \sum_{\substack{\pi \\ 1\notin \{\pi_1,\dots,\pi_{2r}\}}} \sgn \pi \wedge_{j=1}^r \tilde\Omega^B_{\pi_{2j-1},\pi_{2j}}\largewedge_{j=2r+1}^k\theta^E_{\pi_j}\wedge_{j=k+1}^m\omega_{\pi_j,0}^E \nonumber \\
& \quad +|\sigma|^{\frac12}\sum_{\substack{\pi\\1\in \{\pi_1,\dots,\pi_{2r}\}}} \sgn \pi  \wedge_{j=1}^r \tilde\Omega^B_{\pi_{2j-1},\pi_{2j}}\wedge_{j=2r+1}^k\theta^E_{\pi_j}\largewedge_{j=k+1}^m \omega_{\pi_j,0}^E. \label{eq_def_rho}
\end{align}

From \eqref{eq_phi_in_terms_of_rho} and Definition \ref{def_phi_kri} we conclude that, at a point $z_0$ of quadratic compatibility,
\begin{equation} \label{eq_relation_phi_rho}
(\phi_{k,r}^i)_{z_0}=(\chi_i^{-\frac{m+1-k}2}(\sigma)\rho_{k,r})_{z_0}.
\end{equation}

We also can rewrite the form $\psi_{m+1,r}$ in the following way:

\begin{displaymath}
	\psi_{m+1,r}=\sum_{\pi\in S_{m+1}} \sgn(\pi)  \tilde\Omega^B_{\pi_0,\pi_1}\wedge\cdots \wedge\tilde\Omega^B_{\pi_{2r-2},\pi_{2r-1}} \wedge\theta_{\pi_{2r}}\wedge\cdots\wedge\theta_{\pi_m}.
\end{displaymath}

\subsection{Weyl's lemma}

The following is well-known.
\begin{Lemma}
Let $\langle\cdot,\cdot\rangle$ be a symmetric form of signature $(p,q)$  in $\mathbb R^n=\mathbb R^{p+q}$. The algebra of $\mathrm{SO}(p,q)$-invariant elements
of $\mathrm{Sym}(\mathbb R^{n})^*$ is generated by $\langle\cdot,\cdot\rangle$ and $\det$. 
\end{Lemma}

%

We need to generalize Weyl's lemma from \cite{weyl_tubes} to indefinite signatures. 

\begin{Lemma}[Weyl lemma] \label{lemma_weyl}
Let $\mathbb R^{p,q}$ be  endowed with the standard bilinear forms $Q,P$.
Let $V$ be a vector space, and let $\zeta_1,\ldots, \zeta_h\in V^*\otimes \mathbb R^{p,q}$. Let $S^{p+q-1}=\{y\colon P(y,y)=1\}$ be the unit sphere endowed with the volume form $dS^{p+q-1}$ induced by $P$. Then
\begin{multline*}
 \int_{S^{p+q-1}}\chi_i^{-\frac{p+q+h}2}(Q(y,y)) Q( y, \zeta_1)\wedge \cdots\wedge Q( y,\zeta_h) d_yS^{p+q-1}=\\
 =c(p+q,h)\chi_i^{-\frac{q}2}(-1)\frac{1}{h!}\sum_{\pi\in S_h} \sgn \pi Q( \zeta_{\pi_1},\zeta_{\pi_2})\wedge\cdots\wedge Q(\zeta_{\pi_{h-1}},\zeta_{\pi_h}),
\end{multline*}
where $Q(\zeta_k,\zeta_j) \in \largewedge^2V^*$ is defined by $Q(\zeta_k,\zeta_j)(u,v):=Q(\zeta_k(u),\zeta_j(v)) -Q(\zeta_k(v),\zeta_j(u))$, and
where $c(p+q,h)=0$ if $h$ is odd, and for even $h$
\begin{displaymath}
c(p+q,h)= \frac{2\Gamma(\frac{h+1}2) \pi^\frac{p+q-1}2}{\Gamma(\frac{p+q+h}2)}.
\end{displaymath}
\end{Lemma}

\proof
For $p=0$ or $q=0$, the statement follows from the Lemma in Weyl's paper \cite{weyl_tubes}. We proceed with the case $pq\neq0$.

Note that the integral is unchanged by any $g\in\mathrm{SO}(p,q)$. Indeed, let  $\bar g\colon S^{p+q-1}\to S^{p+q-1}$ be given by $\bar g(y)=\frac{g(y)}{P(g(y))^{\frac12}}$. The Jacobian of $z=\bar g y$ is $\frac{d_zS^{p+q-1}}{d_yS^{p+q-1}}=P(gy)^{-\frac{p+q}{2}}$. This follows by identifying $S^{p+q-1}=\mathbb P_+(\R^{p,q})$, and noting that the bundle of dual densities is $\mathrm{SL}(\R^{p+q})$-equivariantly isomorphic to the bundle $\mathbb  \Dens(L)^{p+q}$ over $L\in\mathbb P_+(\R^{p,q})$, and the action of $\bar g$ on $\Dens(L)$ is by $P(gy)^{-\frac12}$.
Then
\begin{align*}
&\int_{S^{p+q-1}}\chi_i^{-\frac{p+q+h}2}(Q(y,y)) Q( y, g^{-1}\zeta_1)\wedge \cdots\wedge Q( y,g^{-1}\zeta_h) d_yS^{p+q-1}=\\
&=\int_{S^{p+q-1}}\chi_i^{-\frac{p+q+h}2}(Q(gy,gy)) Q( gy, \zeta_1)\wedge \cdots\wedge Q( gy,\zeta_h) d_yS^{p+q-1}\\
&=\int_{S^{p+q-1}}\chi_i^{-\frac{p+q+h}2}(Q(\bar gy,\bar gy)) Q( \bar gy, \zeta_1)\wedge \cdots\wedge Q( \bar gy,\zeta_h) P(g(y))^{-\frac{p+q}2} d_yS^{p+q-1}\\
&=\int_{S^{p+q-1}}\chi_i^{-\frac{p+q+h}2}(Q(z,z)) Q(z, \zeta_1)\wedge \cdots\wedge Q( z,\zeta_h) d_zS^{p+q-1},
\end{align*}
as claimed.

By linearity and the previous lemma, the integral must be a linear combination of terms of the form
$Q(\zeta_{\pi_1},\zeta_{\pi_2})\cdots Q(\zeta_{\pi_{h-1}},\zeta_{\pi_h})$.
The result follows by skew-symmetry, except for the constants. To find these, let us take $\zeta_{ j}=dx_j \otimes (1,0,\cdots,0)^T$ and compute
\begin{displaymath}
 I=\int_{S^{p+q-1}}\chi_i^{-\frac{p+q+h}2}(Q( y,y)) y_1^h dS^{p+q-1}.
\end{displaymath}

Let us consider $u(y)=Q( y,y)$ as a function on $S^{p+q-1}$. Its gradient has $P$-norm $|\nabla u(y)|=2\sqrt{1-u(y)^2}$. Its level sets are $u^{-1}(r)=\left(\frac{1+r}2\right)^\frac12S^{p-1}\times \left(\frac{1-r}2\right)^\frac12S^{q-1}$. The coarea formula then gives
\begin{align*}
 I & =\int_{-1}^1 \left(\int_{u^{-1}(r)} \chi(r)|\nabla u|^{-1}\right)dr\\
 &=\int_{-1}^1 \chi(r)\frac12 \left(1-r^2\right)^{-\frac12} \left(\frac{1+r}2\right)^\frac{p+h-1}2 \int_{S^{p-1}} y_1^h dS^{p-1} \left(\frac{1-r}2\right)^\frac{q-1}2\int_{S^{q-1}} dS^{q-1} dr\\
 &= \frac{2\Gamma\left(\frac{h+1}2\right)\pi^{\frac{p-1}2}}{\Gamma\left(\frac{p+h}2\right)} \frac{2\pi^{\frac{q}2}}{\Gamma\left(\frac{q}2\right)} \frac14 \int_{-1}^1  \left(\frac{1+r}2\right)^\frac{p+h-2}{2} \left(\frac{1-r}2\right)^\frac{q-2}2\chi^{-\frac{p+q+h}2}(r) dr
\end{align*} 
if $h$ is even and $I=0$ otherwise. 

The change of variables $r=\cos(2t)=\cos^2t-\sin^2t$ turns the previous integral into $4J_{p+q+h-2,q-1}(1,-1,\chi_i)=4J_{p+q+h-2,p+h-1}(-1,1,\chi_i)$ (see \eqref{eq:integral_definition}). Using Proposition \ref{prop:J_integral_2d} we find that 
$$
 I = \frac{2\Gamma\left(\frac{h+1}{2}\right)\pi^{\frac{p-1}2}}{\Gamma\left(\frac{p+h}2\right)}\frac{2\pi^{\frac{q}2}}{\Gamma\left(\frac{q}2\right)}\frac12 B\left(\frac{p+h}2,\frac{q}2\right) \chi_i^{-\frac{q}2}(-1)=\frac{2\Gamma\left(\frac{h+1}2\right) \pi^\frac{p+q-1}2}{\Gamma\left(\frac{p+q+h}2\right)}\chi_i^{-\frac{q}2}(-1).\qedhere
$$

\subsection{Computation of the restriction}

\begin{Proposition}\label{prop_res_boundary}
Let $e\colon M^{p',q'}\looparrowright \mathbb R^{p,q}$ be an isometric immersion, and put $m=p'+q'-1, n=p+q-1$.  Then modulo filtration $m+1$, the restriction  $e^*C_{k,0}^i$ is
\begin{displaymath}
e^* C^i_{k,0} \equiv (-1)^{(q-q')i +\lfloor \frac{q-q'}{2}\rfloor}\sum_{\nu=0}^{\lfloor\frac{m-k}{2}\rfloor}\frac{1}{2^{2\nu}}{\frac k 2+\nu\choose \nu} C^{i+q-q'}_{k+2\nu,\nu}.
\end{displaymath}
\end{Proposition}

\begin{proof} 
We will compute $\tilde\phi_{k,0}^i=\beta_*\alpha^*\phi_{k,0}^i$ where $\alpha,\beta$ are as in \eqref{eq:restriction_curvature_measure}.
	By Lemma \ref{lem:can_restrict_fiberwise} and Lemma \ref{lem:pointwise}, $(\tilde\phi_{k,0}^i)_{z_0}$ is well-defined for all $z_0\in M$, and those values determine $\tilde\phi_{k,0}^i$. By invariance of the constructions, we may only consider $z_0$
	such that $T_{z_0}M$ is spanned by the coordinate vectors $\frac\partial{\partial x_ 0},\cdots ,\frac\partial{\partial x_{p'-1}},\frac\partial{\partial x_p},\cdots,\frac\partial{\partial x_{p+q'-1}}$.

Let $P,Q$ be the standard bilinear forms on $\mathbb R^{n+1}\equiv \mathbb R^{p,q}:=N$. Let $E_0,\ldots, E_n$ be a $P$-orthonormal frame of $\R^{p,q}$ defined locally on $\mathbb P_M$ so that $E_0, \ldots , E_m$ define a local section of the full flag bundle $\Phi(M)\to \mathbb P_M$,  (where $\mathbb P_M$ is identified with $\mathbb P_+(TM)$ using $Q$). Suppose further that $E_0,\ldots, E_n$ defines an element of $\Psi_J(\R^{p,q})$ (e.g. with respect to the partition $J$ that has $J_+=\{2,\ldots,p'-1, m+1,\ldots,m+p-p'+1\}$) whenever $E_0\in S_{z_0}M$. In particular,  $SE_r=\epsilon_rE_r$ for $m+1\leq r\leq n$ at $S_{z_0}M$, with $\epsilon_r=1$ if $r\in J_+$ and $\epsilon_r=-1$ otherwise.

Given $\xi\in \mathbb P_M, t\in [0,\frac\pi2]$ and $y\in S^{n-m-1}$ put 
\begin{equation*}
 \widetilde H(\xi,t,y)=\cos(t) \xi+\sin(t)\sum_{r=1}^{n-m} y_r SE_{m+r}(\xi),
\end{equation*}
and consider $H\colon \mathbb P_M\times[0,\frac\pi 2]\times S^{n-m-1}\rightarrow \mathbb P_N$ given by 
\begin{displaymath}
 H(\xi,t,y)=P^{-\frac12}(\widetilde H(\xi,t,y)) \widetilde H(\xi,t,y).
\end{displaymath}  Note that $H(\xi,\frac\pi2,y)$ is $Q$-orthogonal to $T_xM$ for $\xi\in S_xM$. We will assume that $E_0,\ldots, E_m$ and $SE_{m+1},\ldots, SE_n$ are positively oriented bases of $T_xM$ and its $Q$-orthogonal complement, respectively.

Let $\sigma_N(\zeta)=Q(\zeta),\sigma_M(\xi)=Q(\xi)$ for $\zeta\in \mathbb P_N, \xi\in \mathbb P_M$. For $\xi\in S_{z_0}M$ we have
\begin{equation}\label{eq:sigma_N}
 \sigma_N(H(\xi,t,y))=\cos^2(t)\sigma_M(\xi)+\sin^2(t)\sum_{r=1}^{n-m}\epsilon_{m+r} y_r^2.
\end{equation}

Let $\hat p$ be the projection of $\mathbb P_M\times [0,\frac\pi2]\times S^{n-m-1}$ to the first factor. Then $\beta_* \alpha^*(\phi_{k,0}^i)=\hat p_* H^*(\phi_{k,0}^i)$. We proceed to compute this generalized form evaluated at $z_0$.  
By Lemma \ref{lem:pointwise} ii), we have  $\hat p_*(H^*(\phi_{k,0}^i))_{z_0}=(\hat p_{z_0})_*(H^*(\phi_{k,0}^i)_{z_0})$, where $\hat p_{z_0}$ is the restriction of $\hat p$ to $\mathbb P_+(T_{z_0}M)\times [0,\frac\pi2]\times S^{n-m-1}$.

Given $\xi\in \mathbb P_M,t\in[0,\frac\pi2],y\in S^{n-m-1}$, let $\hat E_0(\xi,t,y),\ldots,\hat E_{m+1}(\xi,t,y)$ be the $P$-orthonormal basis obtained by the Gram-Schmidt process applied to the sequence $H(\xi,t,y),$  $E_1(\xi),\ldots, E_m(\xi), \frac{d}{dt}H(\xi,t,y)$. Note that $\hat E_0\equiv H$. Given $y\in S^{n-m-1}$ and $\xi\in S_xM$, let $\hat E_{m+2},\ldots, \hat E_n$ be a positively oriented $P$-orthonormal basis of the $P$-orthogonal space to $T_xM\oplus\mathbb{R} \sum_{r=1}^{n-m} y_r SE_{m+r}$. Taken together, $\hat E_0,\ldots, \hat E_n$ form a $P$-orthonormal basis.
Let $\theta^E_i,\omega^E_{i,j}$ (resp. $\hat\theta^E_i,\hat\omega^E_{i,j}$) be the solder and connection forms associated to $E_0,\ldots, E_n$ (resp. $\hat E_0,\ldots,\hat E_n$).  Thus $\theta_i^E,\omega_{i,j}^E$ are differential forms on $\mathbb P_M$, while $\hat\theta^E_i,\hat\omega^E_{i,j}$ are forms on $\mathbb P_M\times[0,\frac{\pi}2]\times S^{n-m-1}$.

By \eqref{eq_def_rho} and \eqref{eq_relation_phi_rho}, and noting that $\rho_{k,0}$ is $\OO(n)$-invariant, we have 
\begin{displaymath}
 H^* \phi_{k,0}^i=\chi_i^{-\frac{n-k+1}2}(\sigma)\sum_\pi\mathrm{sgn}(\pi) \hat \theta^E_{\pi_1}\wedge\cdots\wedge\hat\theta^E_{\pi_k}\wedge \hat \omega^E_{\pi_{k+1},0}\wedge \cdots \wedge\hat \omega^E_{\pi_n,0}.
\end{displaymath}

Given $\xi\in S_{z_0}M$, note that $P(\tilde H(\xi,t,y))=1, d(P(\tilde H))_{(\xi,t,y)}=0$ and $\hat E_{i}(\xi,t,y)= E_i(\xi)$ for $i=1,\ldots,m$. Hence, at $(\xi,t,y)$ 
\begin{align*}
\hat E_{m+1}&=-\sin(t)E_0+\cos(t) \sum_{r=1}^{n-m} y_r SE_{m+r},\\
d\hat E_0&=d\tilde H=\hat E_{m+1} dt+\cos(t)dE_0+\sin(t)\sum_{r=1}^{n-m} d(y_r SE_{m+r}).
\end{align*}

Thus, the following relations hold at $(\xi,t,y)$ with $\xi\in S_{z_0}M$
\begin{align*}
\hat\theta_i^E&=P(\hat E_i,\cdot)=\theta_i^E, \qquad i=1,\ldots, m,\\
\hat\theta_{m+1}^E&=-\sin(t) P(E_0,\cdot )+\cos(t)\sum_r y_r P(SE_{m+r},\cdot)=-\sin(t) \theta_0^E,
\end{align*}
\begin{align*}
\hat\omega_{1,0}^E=P(\hat E_1,d\hat E_0) &=\cos(t) P(E_1, dE_0)+\sin(t)\sum_r P(SE_1, y_rdE_{m+r})\\
& =\cos(t) \omega_{1,0}^E+\sin(t)\sum_r y_r(\tau \omega_{0,m+r}^E-\sigma\omega_{1,m+r}^E), 
\end{align*}
\begin{align*}
\hat \omega_{i,0}^E=P(\hat E_i, d\hat E_0)
&=\cos(t) P(E_i, dE_0)+\sin(t)\sum_r P(SE_i,y_r dE_{m+r})\\
&=\cos(t) \omega_{i,0}^E +\sin(t)\epsilon_i\sum_r y_r \omega_{i,m+r}^E,\qquad i=2,\ldots,m,
\end{align*}
since $SE_1=\tau E_0-\sigma E_1, SE_i=\epsilon_i E_i$ at $z_0$. Also, up to terms lacking $dt$ we have 
\begin{align*}
\hat\omega_{m+1,0}^E & =P(\hat E_{m+1}, d\hat E_0) \equiv dt.
\end{align*}

For $j=2,\ldots,n- m$, up to terms lacking $dy$
\begin{align*}
\hat\omega^E_{m+j,0}&=\sin(t) \sum_r dy_r P(\hat E_{m+j}, SE_{m+r})=\sin(t)\langle dy,e_{j}\rangle,
\end{align*}
where $e_j=(P( \hat E_{m+j},SE_{m+1}),\ldots, P(\hat E_{m+j},SE_{n}))\in \mathbb R^{n-m}$. The vectors $e_2,\ldots, e_{n-m}$ are orthonormal, as their entries are the  coordinates of $\hat E_{m+2},\ldots, \hat E_{n}$ with respect to the $P$-orthonormal basis $SE_{m+1},\ldots,SE_n$. Hence, $e_2,\ldots, e_{n-m}$ form a positively oriented orthonormal basis of $T_y S^{n-m-1}$,  and thus
\begin{align*}
\hat\omega^E_{m+1,0}\wedge\cdots\wedge\hat\omega^E_{n,0}&=\sin^{n-m-1}(t)dt\wedge dS^{n-m-1}
\end{align*}
modulo terms vanishing on $\mathbb R \cdot \frac{\partial}{\partial t} \oplus T_yS^{n-m-1}$. Hence 
\begin{align*}
 H^*(\phi_{k,0}^i)_{z_0}&\equiv\frac{(n-k)!}{(m-k)!}\sum_{h=0}^{m-k}f_h(t)T_h\wedge dt\wedge dS^{n-m-1},
\end{align*}
where
\begin{align*}
 f_h(t)&={m-k\choose h}\cos^{m-k-h}(t)\sin^{n-m+h-1}(t)\chi^{-\frac{n-k+1}2}_i(H^*\sigma_{N}),\\
 T_h&=\sum_{\pi\in S_m} \sgn \pi \langle y,{\hat\zeta}_{\pi_1}\rangle\wedge\cdots \wedge\langle y,{\hat\zeta}_{\pi_h}\rangle\wedge \theta^E_{\pi_{h+1}}\wedge\cdots\wedge\theta^E_{\pi_{h+k}}\wedge\omega^E_{\pi_{h+k+1},0}\wedge\cdots\wedge\omega^E_{\pi_{m},0},
\end{align*}
with 
\begin{align}
\hat \zeta_1&=\tau(\omega_{0,m+1}^E,\ldots, \omega_{0,n}^E)-\sigma(\omega_{1,m+1}^E,\ldots,\omega_{1,n}^E),\label{eq_zeta_1}\\
\hat \zeta_i&=\epsilon_{i}(\omega^E_{i,m+1},\ldots, \omega^E_{i,n})\quad  \mbox{ if }1<i\leq m.\label{eq_zeta_r}
\end{align}

 We now compute $(\hat p_{z_0})_* H^*(\phi_{k,0}^i)_{z_0}=(p_2)_* (p_1)_* H^*(\phi_{k,0}^i)_{z_0}$, where  
$$\mathbb P_+(T_{z_0}M)\times [0,\frac\pi2]\times S^{n-m-1}\stackrel{p_1}\longrightarrow\mathbb  P_+(T_{z_0}M)\times S^{n-m-1}\stackrel{p_2}\longrightarrow \mathbb P_+(T_{z_0}M)$$are the projections. Denoting $\sigma_t(\xi,u)=\sigma_N(H(\xi,t,u))$, by \eqref{eq:sigma_N} and Proposition \ref{prop:J_integral_2d}, we have
\begin{align}
 (p_1)_*&(f_h(t) dt) ={m-k\choose h} \int_0^{\frac\pi2} \cos^{m-k-h}(t)\sin^{n-m+h-1}(t)\sigma_t^*\chi^{-\frac{n-k+1}2}_i dt\notag\\
 =&{m-k\choose h}S\left({n-m+h-1},{m-k-h}\right)\Big(\chi_i^{-\frac{m-k-h+1}2}(\sigma_M)\chi_0^{-\frac{n-m+h}2}(Q(y))+\notag \\
 &+(-1)^{i+1}\chi_{i+1}^{-\frac{m-k-h+1}2}(\sigma_M)\chi_1^{-\frac{n-m+h}2}(Q(y))\Big).\label{int_fdt}
\end{align}

Considering
\begin{equation} \label{eq_zetas}
\zeta_j:=\hat \zeta_j \cdot \mathrm{diag}(\epsilon_{m+1},\ldots, \epsilon_{n}),
 \end{equation}
 we have $\langle y,\hat \zeta_j\rangle =Q(y,\zeta_j)$, and by the Gauss equations of Lemma \ref{lemma_gauss_eqs}
\begin{align*}
 Q(\zeta_i,\zeta_j)=\tilde\Omega^B_{i,j},\qquad Q(\zeta_1,\zeta_j)=|\sigma_M|^\frac12\tilde\Omega^B_{1,j},\qquad  2\leq i,j\leq m.
\end{align*}
Given $\pi\in S_m$, let $\delta(\pi)=|\sigma_M|^{\frac12}$ if $1\in \{\pi_1,\ldots, \pi_h\}$, and $\delta(\pi)=1$ otherwise.
 
The Weyl Lemma \ref{lemma_weyl} yields
\begin{align*}
&\int_{S^{n-m-1}} \chi^{-\frac{n-m+h}2}_j(Q(y)) Q(y,\zeta_{\pi_1}) \wedge\cdots \wedge Q(y,\zeta_{\pi_h}) dS^{n-m-1}=\\
 &=c(n-m,h)\chi_j^{-\frac{q-q'}2}(-1)\frac1{h!}\sum_{\tau\in S_h} \sgn \tau Q(\zeta_{\pi\circ\tau_1},\zeta_{\pi\circ\tau_2})\wedge \cdots \wedge Q(\zeta_{\pi\circ\tau_{h-1}},\zeta_{\pi\circ\tau_h})\\
 &=c(n-m,h)\chi_j^{-\frac{q-q'}2}(-1)\delta(\pi)  \tilde\Omega^B_{\pi_1,\pi_2}\wedge \cdots \wedge \tilde\Omega^B_{\pi_{h-1},\pi_h}.
\end{align*}

 Now we take the Riemannian metric $P'$ given by Lemma \ref{lemma_quad_comp} iii). Since $P'$ coincides with $P$ up to second order, their solder and connection forms coincide at $z_0$. Therefore, for even $h$, 
\begin{align*}
  (p_2)_*& (p_1)_*( f(t)T_hdtdS^{n-m-1})={m-k\choose h} S(n-m+h-1,m-k-h)c(n-m,h)\cdot\\
 &\cdot\left(\chi_0^{-\frac{q-q'}2}(-1)\chi_i^{-\frac{m-k-h+1}2}(\sigma_M)+(-1)^{i+1}\chi_1^{-\frac{q-q'}2}(-1)\chi_{i+1}^{-\frac{m-k-h+1}2}(\sigma_M)\right)\rho_{k+h,h}\\
 &=(-1)^{(q-q')(i+1)+\left\lfloor\frac{q-q'}2\right\rfloor}{m-k\choose h}\frac{\Gamma\left(\frac{m-k-h+1}2\right)\Gamma\left(\frac{h+1}2\right) \pi^\frac{n-m-1}2}{\Gamma\left(\frac{n-k+1}2\right)}\phi_{k+h,h}^{i+q-q'}.
\end{align*}
Thus,
\begin{align}\notag
\beta_*\alpha^*\phi^i_{k,0} = & (-1)^{(q-q')(i+1)+\lfloor\frac{q-q'}{2}\rfloor} \frac{(n-k)!}{(m-k)!}\cdot\\
&\cdot\sum_{\nu=0}^{\lfloor\frac{m-k}{2}\rfloor}{m-k\choose 2\nu}\frac{\Gamma(\frac{m-k-2\nu+1}2)\Gamma(\frac{2\nu+1}2) \pi^\frac{n-m-1}2}{\Gamma(\frac{n-k+1}2)} \phi_{k+2\nu,\nu}^{i+q-q'}.\label{eq:restriction_form_boundary}
\end{align}
The statement follows.
\end{proof}

\begin{Proposition} \label{prop_res_interior}
Let $e\colon M^{p',q'}\rightarrow \mathbb R^{p,q}$ be an isometric embedding, and put $m=p'+q'-1, n=p+q-1$. The interior term of the restriction $e^*C_{k,0}^i$ is 
 \begin{displaymath}
\frac{(-1)^{(q-q')i+\lfloor \frac{q-q'}2\rfloor}}{2^{m+1-k}}{\frac{m+1}{2}\choose \frac{m+1-k}{2}} C_{m+1,\frac{m+1-k}{2}}^{i+q-q'}
\end{displaymath}
if $m-k$ is odd, and it vanishes otherwise.
\end{Proposition}

\begin{proof} 
Let $q,\theta$ be as in \eqref{eq:restriction_curvature_measure}, and $z_0$ as in the previous proof. Using Lemma \ref{lem:can_restrict_fiberwise}, and applying Lemma \ref{lem:pointwise}ii) with $S=M$ and $W=N_{z_0}^*M$, we conclude that $q_*\theta^*\phi_{k,0}^i$ is a smooth measure, and we may compute the restriction $q_*\theta^*\phi_{k,0}^i|_{z_0}$ fiberwise. Take a $P$-orthonormal frame $E_0,\ldots, E_n$ defined on $M$, such that $E_0,\ldots, E_m$ are tangent to $M$, and $SE_i=\epsilon_i E_i$ at $z_0$ for $i=0,\ldots,n$. Define  $G\colon {M\times S^{n-m-1}}\rightarrow \mathbb P_N$ by
\begin{displaymath}
G(x,y)=\sum_{r=1}^{n-m}y_rSE_{m+r}(x).
\end{displaymath}
Note that $\sigma_N(G(z_0,y))=\sum_{r=1}^{n-m}\epsilon_{m+r} y_r^2=Q(y)$.

Let $\hat E_0(x,y), \ldots, \hat E_{m+1}(x,y)$ be the $P$-orthonormal vectors obtained by applying the Gram-Schmidt process to $G(x,y), E_1,\ldots, E_m,-E_0$. Let $ \hat \theta_i,\hat\omega_{i,j}$ be the corresponding solder and connection forms. Note that these $\hat E_i$ correspond to the previous $\hat E_i$  when  $t=\frac\pi2, \sigma=\epsilon, \tau=0$. Hence, at $z_0$ we can use the previously obtained relations for $\hat \theta_i^E,\hat\omega_{i,j}^E$ taking these values for $t,\sigma,\tau$.  

Moreover, at $z_0$ we have
\begin{displaymath}
\hat\omega_{m+1,0}^E=\sum_{r=1}^{n-m}y_r \omega_{m+r,0}^E
\end{displaymath}
and
\begin{displaymath}
\hat\omega^E_{m+2,0}\wedge\cdots\wedge\hat\omega^E_{n,0}\equiv dS^{n-m-1}
\end{displaymath}
modulo terms vanishing at $T_yS^{n-m-1}$. Therefore
\begin{displaymath}
 G^*(\phi_{k,0}^i)_{z_0}\equiv\frac{(n-k)!}{(m+1-k)!}\chi_i(G^*\sigma_N)Y_{m+1-k}\wedge dS^{n-m-1},
\end{displaymath}
where, using the notation  \eqref{eq_zeta_1}, \eqref{eq_zeta_r}, \eqref{eq_zetas}, and putting $\hat \zeta_0=\epsilon(\omega_{0,m+1}^E,\ldots,\omega_{0,n}^E)$,
\begin{displaymath}
Y_h=\sum_{\pi\in S_{m+1}} \sgn \pi \langle y,\hat \zeta_{\pi_0} \rangle\wedge\cdots \wedge\langle y,\hat \zeta_{\pi_{h-1}}\rangle\wedge \theta^E_{\pi_{h}}\wedge\cdots\wedge\theta^E_{\pi_{m}}.
\end{displaymath} 

By the Weyl Lemma \ref{lemma_weyl} and the Gauss equations of Lemma \ref{lemma_gauss_eqs},
\begin{multline*}
\int_{S^{n-m-1}} \chi_i^{-\frac{n-k+1}2}(Q(y)) \langle y,\hat \zeta_{\pi_0}\rangle\wedge\cdots\wedge\langle y,\hat \zeta_{\pi_{m-k}}\rangle dy\\
=\frac{c(n-m,m+1-k)\chi_i^{-\frac{q-q'}2}(-1)}{(m+1-k)!}\sum_{\tau\in S_{m+1-k}} \sgn \tau Q(\zeta_{\pi\circ\tau_0},\zeta_{\pi\circ\tau_1}) \wedge \cdots \wedge Q(\zeta_{\pi\circ\tau_{m-k-1}},\zeta_{\pi\circ\tau_{m-k}})\\
=c(n-m,m+1-k)\chi_i^{-\frac{q-q'}2}(-1) \sum_{\tau\in S_{m+1-k}} \sgn \tau  \tilde\Omega^B_{\pi_0,\pi_1}\wedge\cdots\wedge \tilde\Omega^B_{\pi_{m-k-1},\pi_{m-k}},
\end{multline*}
where we put $\zeta_0=\epsilon \hat \zeta_0$. Therefore, 
\begin{align}
q_*\theta^*\phi_{k,r}^i&=\int_{S^{n-m-1}} G^*(\phi_{k,0}^i) dS^{n-m-1}\notag \\
&= \frac{c(n-m,m+1-k)\chi_i^{-\frac{q-q'}2}(-1)(n-k)!}{(m+1-k)!} \psi_{m+1,\frac{m+1-k}2}\label{eq:restriction_form_interior}
\end{align}
which vanishes if $m-k$ is even.
The statement follows.
\end{proof}

\proof[Proof of Theorem \ref{thm:weyl}]
Let $e:M^{p',q'}\looparrowright N^{p,q}$ be an isometric immersion between pseudo-Riemannian manifolds. We want to show that $e^*\Lambda_k^{N}=\Lambda^{M}_{k}$. 

The case $N=\R^{p,q}$ follows from Propositions \ref{prop_res_boundary} and \ref{prop_res_interior} as we show next. Suppose that $e:M^{p',q'} \looparrowright \R^{p,q}$ is an isometric immersion. Since the curvature of $\R^{p,q}$ vanishes, we have  by \eqref{eq:C_to_Lambda}
\begin{align*}
e^* \Lambda_k^{\R^{p,q}} & = \i^q \frac{\pi^k}{k! \omega_k} e^*(C_{k,0}^0+\i C_{k,0}^1)\\
& = \i^q \frac{\pi^k}{k! \omega_k} (-1)^{\lfloor \frac{q-q'}{2}\rfloor} \sum_{\nu=0}^{\frac{m+1-k}{2}} \frac{1}{2^{2\nu}} C_{k+2\nu,\nu}^{q-q'}\\
& \quad + \i^{q+1} \frac{\pi^k}{k! \omega_k}  (-1)^{(q-q')+\lfloor \frac{q-q'}{2}\rfloor} \sum_{\nu=0}^{\frac{m+1-k}{2}} \frac{1}{2^{2\nu}} C_{k+2\nu,\nu}^{1+q-q'} = \Lambda_k^{M^{p',q'}},
\end{align*} 
using \eqref{eq:C_to_Lambda} and considering different cases according to the parity of $q-q'$.

For the general case, let $\tilde e\colon N^{p,q}\rightarrow \R^{a,b}$ be an isometric embedding, which exists by the pseudo-Riemannian Nash Theorem \ref{thm_nash}. Using the previous case and the functoriality of pull-backs we deduce
 \begin{displaymath}
   e^* \Lambda^N_k=e^* \tilde e^*\Lambda_k^{\R^{a,b}}= (\tilde e \circ e)^*\Lambda_k^{\R^{a,b}}= \Lambda_k^M.\qedhere
\end{displaymath}

\begin{Remark}\label{rem:forms_are_defined}
	In fact, since \eqref{eq:restriction_form_boundary} and \eqref{eq:restriction_form_interior} hold exactly and not just at the level of curvature measures, we have $e^*(\kappa^N_k, \lambda^N_k)=(\kappa^M_k,\lambda^M_k)$ also at the level of forms.
\end{Remark}

\begin{Definition}\label{def:LC_regular_LK}
	Let $(X,g)$ be an LC-regular manifold of changing signature. Define its Lipschitz-Killing forms (resp. curvature measures, valuations) by $(\kappa^X_k, \lambda^X_k):=e^*(\kappa^M_k,\lambda^M_k)$, resp. $\Lambda_k^X:=e^*\Lambda_k^M$, $\mu_k^X:=e^*\mu_k^M$, where $e:X\hookrightarrow M$ is an isometric embedding into any pseudo-Riemannian manifold $M$.
\end{Definition}
\begin{Proposition}\label{prop:LC_regular_LK}
	The Lipschitz-Killing curvature measures and forms of an LC-regular manifold are well-defined.
\end{Proposition}
\proof
The existence of some isometric embedding follows from Lemma \ref{lemma:baby_nash}. The restriction is then well-defined by Proposition \ref{prop:nondegenerate_restriction}. The restriction is independent of the choice of $M$ by Lemma \ref{lemma:two_manifolds_simultaneous_embedding} and Theorem \ref{thm:weyl}, resp. Remark \ref{rem:forms_are_defined}.
\endproof

\section{Basic properties of the Lipschitz-Killing curvature measures}\label{sec:properties}
For a manifold $X$, let $\mathrm{Met}_{p,q}(X)$ denote the space of pseudo-Riemannian metrics of signature $(p,q)$ on $X$, and similarly $\mathrm{Met}_{\LC}(X)$ is the space of LC-regular metrics.

\begin{Proposition} \label{prop_wave_front_lkc}
For any $M^{p,q}$ it holds that $\Lambda_k^M\in \mathcal C^{-\infty}_{\emptyset, N^*(\LC^*_M)}(M)$. In particular, $\Lambda_k^M( A,\bullet)\in\mathcal M^{-\infty}(M)$ is well-defined for any LC-transversal $A \in \mathcal P(M)$. 
\end{Proposition}

\proof
The first assertion is immediate from Definition \ref{def_lk_curvatures} and Corollary \ref{cor:wave_front_set_of_LK}. The second follows from Proposition \ref{prop:curvature_measures_at_polyhedra}.
\endproof

\begin{Proposition}\label{prop:top_degree}
	Let $(X^n,g)$ be LC-regular, and put $\mathrm{sign}g|_{T_xX}=(p_x,q_x)$. Then 
	\begin{displaymath}
	\Lambda_n^X= \i^{q_x}\vol_{X,g}.
	\end{displaymath}
\end{Proposition}
\proof
Fix an embedding $e:X^n\hookrightarrow \R^{p,q}$, and denote $V=\R^{p,q}$. Note that $e^*\mu_n^{V}=\glob(e^*\Lambda_n^{V})$ can be identified with $\Lambda_n^X=e^*\Lambda_n^{V}$, as both are elements of $\mathcal M^{-\infty}(X)$.
Now consider test valuations $\psi\in\mathcal V_c^\infty(X)$ of the form $\psi=\int_S \chi_{Y} d\nu(Y)$, where $S$ is a smooth family of $n$-submanifolds with boundary, and $\nu\in\mathcal M^\infty_c(S)$, see \cite{fu_alesker_product}. Note that all possible images $[\psi]\in \mathcal V_c^\infty(X)/\mathcal W^\infty_1(X)=C_c^\infty(X)$ form a dense subset. We compute
\[\langle e^*\mu_n^{V}, \psi\rangle = \langle \mu_n^{V}, e_*\psi\rangle=\int_S\langle \mu_n^{V}, \chi_{e(Y)}\rangle d\nu(Y)=\int_S\int_{Y}\i^{q_x}\vol_{X,g}d\nu(Y),\]
where the last two equalities follow from the continuity of the Klain section of $\mu_n^{V}$ and its explicit value, see \cite{bernig_faifman_opq}. Thus $ \langle \Lambda_n^X, \psi\rangle=\langle \i^{q_x}\vol_{X,g}, \psi\rangle$, concluding the proof.

\endproof

\begin{Proposition}\label{prop:filtration_degree}
	For $(X,g)$ LC-regular,	$\Lambda_k^X\in\mathcal C_k^{-\infty}(X)\setminus \mathcal C_{k+1}^{-\infty}(X)$, and $\mu_k^X\in\mathcal W_k^{-\infty}(X)\setminus \mathcal W_{k+1}^{-\infty}(X)$. Moreover, $\mu_k^X$ has Euler-Verdier eigenvalue $(-1)^k$. 
\end{Proposition}

\proof
All statements are easily verified for $X=\R^{p,q}$. For any LC-regular $X$ and any $k\leq \dim X$, we can choose by Proposition \ref{prop_lc_regularity_pseudo_riemann} a $k$-dimensional pseudo-Riemannian submanifold $M\subset X_{\textrm{ND}}$, whence $\Lambda_k^X|_M=\Lambda_k^M\neq 0$ by Proposition \ref{prop:top_degree}. Thus $\Lambda_k^X\notin  \mathcal C_{k+1}^{-\infty}(X)$, and similarly $\mu_k^X \notin  \mathcal W_{k+1}^{-\infty}(X)$.  The last statement reduces to $\R^{p,q}$ by the Nash embedding theorem and the Weyl principle.
\endproof

\begin{Corollary}\label{cor:LK_valuations}
If $(X^n,g)$ is LC-regular and contains a pseudo-Riemannian subset of indefinite signature, then $(\mu^X_k)_{k=0}^n$, $(\bar\mu^X_k)_{k=1}^{n-1}$ are linearly independent, as well as $(\Lambda_k^X)_{k=0}^n$, $(\bar\Lambda_k^X)_{k=0}^{n-1}$.
In particular for $M^{p,q}$ with $p,q\geq 1$, we have $\dim \widetilde{\mathcal {LK}}(M)=2\dim M+1$ and $\dim \mathcal {LK}(M)=2\dim M$.	
\end{Corollary}

\proof
 
By Proposition \ref{prop:filtration_degree}, we may consider each $k$ separately. For $k\geq1$, the statement follows from
Propositions \ref{prop:top_degree} and \ref{prop_lc_regularity_pseudo_riemann}. For $k=0$, $\mu^X_0$ is the Euler characteristic by Theorem \ref{thm:Lambda_0_glob}. By assumption, we may find a submanifold $M^{1,1}\subset X_{\textrm{ND}}$. It now can be seen from definition that $\Lambda_0^M=\Lambda_0^X|_M$ has linearly independent real and imaginary parts, completing the proof. 
\endproof

\begin{Remark}
For $(X^n,g)$ LC-regular, the last statement of the corollary fails in general. By Proposition \ref{prop:top_degree}, $\dim\Span\{\Re\Lambda_n, \Image \Lambda_n\}\in\{1,2\}$ depending on the signatures occurring in $X$.
\end{Remark}

\proof[Proof of Theorem \ref{mainthm_weyl}]
The existence of the functors has been proven in Section \ref{sec_restriction}. 
For uniqueness of the valuation functor, assume that $\psi:\mathbf{\Psi Met}\to \mathbf {GVal}$ is a covariant functor commuting with $\mathcal{M}$. Fix $p,q>0$. By \cite{bernig_faifman_opq} we know that $\dim \Val^{-\infty}(\R^{p,q})^{\OO(p,q)}=2(p+q)$. Using Corollary \ref{cor:LK_valuations} we conclude that 
\begin{equation} \label{eq_decomposition_flat}
\psi^{\R^{p,q}}=\sum_{k=0}^{p+q} a_k \mu_k^{\R^{p,q}}+\sum_{k=1}^{p+q-1} b_k \bar \mu_k^{\R^{p,q}},
\end{equation}
with constants $a_k,b_k$.  These constants do not depend on the choice of $\R^{p,q}$ as long as $p+q>k$, as follows by applying the functoriality of $\mu_k$ to some simultaneous isometric embeddings of $\R^{p,q}$ and $\R^{p',q'}$ into some $\R^{p'',q''}$. 

We now define sequences $a_k,b_k, k=0,1,\ldots$ by taking arbitrary $p+q>k$ in \eqref{eq_decomposition_flat}. Then $\psi=\sum_{k=0}^\infty a_k \mu_k^{\R^{p,q}}+\sum_{k=1}^\infty b_k \bar \mu_k^{\R^{p,q}}$ on each pseudo-Euclidean space. By functoriality and the pseudo-Riemannian Nash embedding theorem \ref{thm_nash}, this equation then holds on each pseudo-Riemann manifold. 
\endproof

The Lipschitz-Killing curvature measures are homogeneous.

\begin{Proposition} \label{prop_scaling} Let $(X,g)$ be LC-regular, and $\lambda\neq 0$. Then
	\begin{align}\label{eq:homothety}
	\Lambda_k^{X,\lambda g}=\begin{cases} 
	\sqrt{\lambda}^k\Lambda_k^{X,g},& \lambda>0,\\ 
	 \sqrt{\lambda}^k \overline{\Lambda_k^{X,g}},& \lambda<0.
	\end{cases}
	\end{align}
Analogous formulas hold for the intrinsic volumes.
\end{Proposition}

\begin{proof} 
By Definition \ref{def:LC_regular_LK}, we may assume $(X,g)=(M,Q)$ is pseudo-Riemannian. For $\lambda>0$, the stated formula is easy to check. It remains to prove the case $\lambda=-1$.

Put $\widehat Q=-Q$. The corresponding solder, connection and curvature forms are related by $\widehat \theta_i=\theta_i, \widehat \omega_{i,j}=\omega_{i,j}, \widehat\Omega_{i,j}=\Omega_{i,j}$, while $\widehat \sigma=-\sigma, \widehat\epsilon_i=-\epsilon_i$. It follows that $\widehat\phi_{k,r}=(-1)^{m-k+r}\phi_{k,r}$ on the open orbits. For $(m-k)$ even, one checks that 
$
\widehat\phi_{k,r}^i=(-1)^{\frac{m-k+2r}2}\phi_{k,r}^{i+1},
$
while for $(m-k)$ odd, it holds
$
\widehat\phi_{k,r}^i=(-1)^{\frac{m-k+2r+1-2i}2}\phi_{k,r}^i.
$
It follows that $
\widehat\Lambda_k=\i^{k}\overline{\Lambda_k}$ as stated. 
\end{proof}

Next we address the dependence of $\Lambda_k$ on the metric. We work with the local H\"older space of functions $C^{n+\alpha}=C^{n,\alpha}$ with $n\geq0$ and $0<\alpha<1$, consisting of $C^n$ functions whose derivatives of order $n$ are locally $\alpha$-H\"older. For $f\in C_c^{n+\alpha}(M)$, write $\|f\|_{n+\alpha}$ for the corresponding norm.
The norm of a differential form is taken to be the maximum of the norms of its coefficients in the  standard basis.  

\begin{Proposition}\label{prop:continuity_of_LK}
	For a manifold $M$ of dimension $m+1$ and $0\leq k\leq m-1$, the assignment $\Lambda_k: \mathrm{Met}_{p,q}(M)\to \mathcal C^{-\infty}(M)$ is continuous in the $C^{\frac{m+3-k}{2}+\epsilon}$ local H\"older topology on $\mathrm{Met}_{p,q}(M)$, for all $\epsilon>0$.
\end{Proposition}


\proof
We may assume $M=\R^{m+1}$. As the inner term of $\Lambda_k$ is a smooth measure given by a polynomial in the curvature tensor, it remains to consider the boundary term. It suffices to show that the forms $\phi_{k,r}^i$ defined in Section \ref{sub:globablly_defined_forms} are continuous. 

Denote $\nu=\frac{m+3-k}{2}+\epsilon$. We may choose a $C^\nu$-smooth assignment of compatible Riemannian metrics $P=P(Q)$ over $M$ for $Q\in \mathrm{Met}_{p,q}(M)$ in a $C^\nu$-neighborhood $\Omega$ of some $Q_0$. The forms $\phi_{k,r}$ are given by linear combinations of products $\sigma_Q^*\chi^d\cdot \rho_Q$, where $\sigma_Q=\frac{Q}{P}\in C^\infty(\mathbb P_+(TM))$, $\chi^d\in C^{-\infty}(\R)$ is $d$-homogeneous with $-\frac12\geq d\geq -\frac{m+1-k}{2}$, and $\rho_Q$ is a smooth form that is $C^{\nu-2}$-continuous in $Q\in \Omega$. 

Note that $\chi^d\in C_c^{-d-1+\epsilon}(\R)^*$ for all $\epsilon>0$, see \eqref{eq:homogeneous_distribution_def}. It follows that $\sigma_Q^*\chi^d\in C^{\nu-2}(\mathbb P_+(TM))^*$, as the singular points of $\chi^d$ are regular values of $\sigma_Q$. 

Fix cut-off functions $\psi_1,\psi_2\in C^\infty_c(\mathbb P_+(TM))$ such that $\Supp\psi_{2}\subset\psi_1^{-1}(1)$. Denote $B=\{\beta\in C^{\nu-2}_{\Supp \psi_1}(\mathbb P_+(TM)):\|\beta\|_{\nu-2}= 1\}$. For $\rho\in C^{\nu-2}(\mathbb P_+(TM))$ and $\eta\in C^{\nu-2}(\mathbb P_+(TM))^*$ we have
\begin{align*}
\|\psi_2 \eta  \rho \|^*_{\nu-2}\leq \sup_{\beta\in B}\langle \beta,\psi_2 \eta \rho\rangle= \sup_{\beta\in B}\langle \beta\psi_1\rho,\psi_2 \eta\rangle\leq C\| \psi_2\eta\|^*_{\nu-2}\|\psi_1\rho\|_{\nu-2}
\end{align*}
for some constant $C>0$. We then may write 
\begin{multline*}
|\psi_2(\sigma_Q^*\chi^d\rho_Q-\sigma_{Q_0}^*\chi^d\rho_{Q_0})\|^*_{\nu-2} \\
\leq C(\| \psi_2\sigma_Q^*\chi^d\|^*_{\nu-2}\|\psi_1(\rho_Q-\rho_{Q_0})\|_{\nu-2} +\| \psi_2(\sigma_Q^*\chi^d-\sigma_{Q_0}^*\chi^d)\|^*_{\nu-2}\|\psi_1\rho_{Q_0}\|_{\nu-2}  ). 
\end{multline*}

It remains to show that $\| \psi_2(\sigma_Q^*\chi^d-\sigma_{Q_0}^*\chi^d)\|^*_{\nu-2}=o(\|\psi_1(\sigma_Q-\sigma_{Q_0})\|_{\nu})$. Write 
\begin{displaymath}
\|\psi_2(\sigma_Q^*\chi^d-\sigma_{Q_0}^*\chi^d)\|^*_{\nu-2}\leq \sup_{\beta\in B} \langle \chi^d, (\sigma_Q)_*(\psi_2\beta)-(\sigma_{Q_0})_*(\psi_2\beta)\rangle.
\end{displaymath}

Splitting $\chi^d=\chi^d\cdot \mathbbm1_{[-1/2,1/2]}+ \chi^d\cdot (1-\mathbbm1_{[-1/2,1/2]})$, the corresponding second summand is $o(\|\psi_1(\sigma_Q-\sigma_{Q_0})\|_{C^1})$, uniformly in $\beta\in B$. For the first summand,
\begin{multline*} 
\sup_{\beta\in B} \langle \chi^d \mathbbm1_{[-1/2,1/2]}, (\sigma_Q)_*(\psi_2\beta)-(\sigma_{Q_0})_*(\psi_2\beta)\rangle \leq\\
\|  \chi^d \mathbbm1_{[-1/2,1/2]} \|^*_{\nu-2}\sup_{\beta\in B}\|(\sigma_Q)_*(\psi_2\beta)-(\sigma_{Q_0})_*(\psi_2\beta) \|_{\nu-2}.  
\end{multline*}
We use an auxiliary Riemannian structure and the co-area formula to write 
\begin{displaymath}
(\sigma_Q)_*(\psi_2\beta)(t)= \int_{\{\sigma_Q=t\}}\frac{\psi_2(z)\beta(z)}{\|\nabla \sigma_Q\|}dz,
\end{displaymath}
which is readily seen to be $C^{\nu-2}$-continuous in $Q\in\Omega$, uniformly in $\beta\in B$. 

The inclusion $C^\nu(\mathbb P_+(TM))^* \hookrightarrow C^{-\infty}(\mathbb P_+(TM))$ is continuous with the strong dual topology on the right hand side, hence $Q \mapsto \sigma_Q^*\chi^d\cdot \rho_Q$ is continuous on $\Omega$.
\endproof

\begin{Corollary}\label{cor:LC_continuity}
	For a manifold $X^{n}$, $0\leq k\leq n-2$ and $\epsilon>0$, the assignment $\Lambda_k: \mathrm{Met}_{\LC}(X)\to \mathcal C^{-\infty}(M)$ is continuous in the $C^{n-\frac{k}2+1+\epsilon}$ topology on $\mathrm{Met}_{\LC}(X)$.
\end{Corollary}

\proof
Recall the locally defined embedding $e:X \hookrightarrow (M, G)$ constructed in Lemma \ref{lemma:baby_nash}.  Clearly $G$ can be chosen to depend arbitrarily smoothly on $g$. By Proposition \ref{prop:continuity_of_LK}, $\Lambda_k^{M, G}$ is continuous in $G$, hence $\Lambda_k^{X,g}=e^*\Lambda_k^{M,G}$ is continuous in $g$.
\endproof
%

We next generalize the tube formula by Willison.
Fix $M^{p',q'}\subset \R^{p,q}$, and let $N^1(M)$ consist of those $(x,v)\in N^QM$ with $|Q(v)|=1$. Consider the map $$\exp\colon N^1(M)\times\R\to \R^{p,q},\qquad \exp(x,v,t)=x+tv.$$ Given $r>0$ and a non-empty, relatively compact Borel set $U\subset M$, the $r$-tube of $M$ over $U$ is the following region of $\R^{p,q}$
\begin{displaymath}
	T(M,U,r)=\exp(N^1(M)\cap \pi^{-1}U\times[0,r]).
\end{displaymath}

\begin{Proposition} Under the above assumptions, and for sufficiently small $r>0$, the tube $T(M,U,r)$ is bounded if and only if $p=p'$ or $q=q'$, in which case
	\begin{align}
		\vol(T(M,U,r))&=(-\mathbf i)^{q'}\sum_{\nu=0}^{\lfloor {m+1\over 2}\rfloor}(-1)^{\nu(q-q')} \omega_{n-m+2\nu} \Lambda^M_{m-2\nu+1}(M,U) r^{n-m+2\nu}.\label{eq:tube1}
	\end{align}
\end{Proposition}

\begin{proof}
	The tube is bounded precisely when $(T_xM)^Q$ has definite signature; i.e. when $p=p'$ or $q=q'$.  
	Consider first the case $q=q'$ and put $p+q=n+1$, $p'+q=m+1$. Then
	\begin{align*}
		\exp^*(d\vol)= dt\wedge (\theta_1+ t\omega_{1,0})\wedge\cdots\wedge(\theta_n+t\omega_{n,0})=dt\wedge\sum_{k=0}^{n} \frac{ t^k}{k!(n-k)!} \phi_{n-k,0}^0.
	\end{align*}
	By Theorem \ref{thm:weyl},
	\begin{align*}
		\int_{N^1(M)\cap \pi^{-1}U\times [0,r]} \exp^*(d\vol)  &=\sum_{k=0}^{n}\frac{r^{k+1}}{(k+1)!(n-k)!}  \int_{U} \pi_*\phi_{n-k,0}^{0}\\
		&=\mathbf i^{-q} \sum_{\nu=0}^{\lfloor\frac{m+1}2\rfloor}r^{n-m+2\nu} \omega_{n-m+2\nu} \int_U \kappa_{m-2\nu+1}.
	\end{align*}
	Eq. \eqref{eq:tube1} follows for $q=q'$. The case $p=p'$ follows by flipping the sign of the metric, and using $\Lambda_k(M^{q',p'},U)=\i^k(-1)^{q'} \Lambda_k(M^{p',q'},U)$ (see Prop. \ref{prop_scaling}).
\end{proof}
 
  
\section{Euler characteristic and Chern-Gauss-Bonnet}\label{sec:gauss_bonnet}
We obtain a pseudo-Riemannian representation of the Euler characteristic, extending the  integral formula discovered by Chern in the Riemannian case \cite{chern44} to the LC-regular setting. We also supplement the 
pseudo-Riemannian Weil-Allendoerfer-Chern-Gauss-Bonnet theorem \cite{avez63,chern63} with the boundary term,  giving an expression for $\chi( A)$ in pseudo-Riemannian terms for $A \subset M$, rather than for $ A=M$ alone.

\label{sec_gauss_bonnet}

\begin{Theorem}\label{thm:Lambda_0_glob}
	It holds on any LC-regular manifold $(X,g)$ that $\mu_0^X=\chi$. In particular, if $X$ is compact, then $$\int_X \kappa_0^X=\chi(X).$$
\end{Theorem}
Recall those are in fact two equalities on real-valued curvature measures, globalizing to $\chi$ and $0$, respectively.
\proof

Consider first $X=\R^{p,q}$. Recall \cite{bernig_faifman16} that the space of translation invariant generalized valuations on $X$ coincides with $\Val^{-\infty}(X)=\Val^\infty(X)^*\otimes \Dens(X)$. In particular, $\Val_0^{-\infty}(X)=\Val_0^\infty(X)=\Span\{\chi\}$.

It follows that $\mu_0^X=c(p,q)\chi$ for some $c(p,q)\in\C$. It holds by Theorem \ref{thm:weyl} that for all $p,q,p',q'$, $c(p,q)=c(p',q')$. One easily verifies that $c(1,0)=1$.

Now for any LC-regular manifold $X$ there is an isometric embedding $e:X\hookrightarrow\R^{p, q}$. We find by Theorem \ref{thm:weyl} that $\mu_0^X=e^*\mu_0^{\R^{p,q}}=e^*\chi=\chi$.
\endproof

An immediate corollary is a Chern-Gauss-Bonnet theorem  for pseudo-Riemannian manifolds with generic boundary, as follows.

\begin{Theorem}\label{thm:gauss_bonnet_with_boundary}Let $M^{p,q}$ be a compact pseudo-Riemannian manifold with LC-regular boundary.  Let $\nu$ be the outer normal to $\partial M$.
Then 
\begin{displaymath}
\chi(M)=\int_M \kappa^M_0+\int_{\partial M}\nu^*\lambda_0^M.
\end{displaymath}
\end{Theorem}

Inspecting the definition of $\Lambda_0$, we obtain a-priori information about the Euler characteristic of subsets of a pseudo-Riemannian manifold.

\begin{Corollary}\label{cor:CGB_boundary_cases}
For $X\in\mathcal P(M^{p,q},Q)$ LC-transversal, $\chi(X)$ is determined by 
\begin{enumerate}
\item the $Q$-positive (negative) part of $N^QX$, if $p+q$ is odd and $q$ is even (odd);
\item any open subset of $N^QX$ containing the $Q$-degenerate subset, if $p q$ is odd.
\end{enumerate}
\end{Corollary}

\noindent{\bf Examples.} We conclude with two simple examples:\\
\noindent$\bullet$ For a smooth compact domain $X\subset M^{1,1}$ one has $\chi(X)=\frac14 NX\cdot \LC_M$,
where $NX\cdot \LC_M$ is the intersection number. 

\noindent$\bullet$ For an LC-regular closed surface $X\subset\R^{2,1}$,
\begin{equation}\label{eqn:gauss_Bonnet} 
\chi(X) = -\frac{1}{2\pi}\int_{X} Q(\nu_E)_-^{-\frac 3 2}K_E dA_E,
\end{equation} 
where $K_E$ is the Gaussian curvature with respect to the standard Euclidean structure, $dA_E$ the Euclidean area measure, and $\nu_E$ the Euclidean outer unit normal. 
\bibliographystyle{abbrv}
\bibliography{references}

\begin{thebibliography}{10}

\bibitem{aguirre_fernandez_lafuente}
E.~Aguirre, V.~Fern\'{a}ndez, and J.~Lafuente.
\newblock On the conformal geometry of transverse {R}iemann-{L}orentz
  manifolds.
\newblock {\em J. Geom. Phys.}, 57(7):1541--1547, 2007.

\bibitem{alesker_val_man1}
S.~Alesker.
\newblock Theory of valuations on manifolds. {I}. {L}inear spaces.
\newblock {\em Israel J. Math.}, 156:311--339, 2006.

\bibitem{alesker_val_man2}
S.~Alesker.
\newblock Theory of valuations on manifolds. {II}.
\newblock {\em Adv. Math.}, 207(1):420--454, 2006.

\bibitem{alesker_survey07}
S.~Alesker.
\newblock Theory of valuations on manifolds: a survey.
\newblock {\em Geom. Funct. Anal.}, 17(4):1321--1341, 2007.

\bibitem{alesker_val_man4}
S.~Alesker.
\newblock Theory of valuations on manifolds. {IV}. {N}ew properties of the
  multiplicative structure.
\newblock In {\em Geometric aspects of functional analysis}, volume 1910 of
  {\em Lecture Notes in Math.}, pages 1--44. Springer, Berlin, 2007.

\bibitem{alesker_intgeo}
S.~Alesker.
\newblock Valuations on manifolds and integral geometry.
\newblock {\em Geom. Funct. Anal.}, 20(5):1073--1143, 2010.

\bibitem{alesker_alexandrov_spaces}
S.~Alesker.
\newblock Some conjectures on intrinsic volumes of {R}iemannian manifolds and
  {A}lexandrov spaces.
\newblock {\em Arnold Math. J.}, 4(1):1--17, 2018.

\bibitem{alesker_bernig}
S.~Alesker and A.~Bernig.
\newblock The product on smooth and generalized valuations.
\newblock {\em Amer. J. Math.}, 134(2):507--560, 2012.

\bibitem{alesker_bernig_convolution}
S.~Alesker and A.~Bernig.
\newblock Convolution of valuations on manifolds.
\newblock {\em J. Differential Geom.}, 107(2):203--240, 2017.

\bibitem{alesker_faifman}
S.~Alesker and D.~Faifman.
\newblock Convex valuations invariant under the {L}orentz group.
\newblock {\em J. Differential Geom.}, 98(2):183--236, 2014.

\bibitem{alesker_val_man3}
S.~Alesker and J.~H.~G. Fu.
\newblock Theory of valuations on manifolds. {III}. {M}ultiplicative structure
  in the general case.
\newblock {\em Trans. Amer. Math. Soc.}, 360(4):1951--1981, 2008.

\bibitem{allendoefer_weil}
C.~B. Allendoerfer and A.~Weil.
\newblock The {G}auss-{B}onnet theorem for {R}iemannian polyhedra.
\newblock {\em Trans. Amer. Math. Soc.}, 53:101--129, 1943.

\bibitem{atiyah_bott_patodi}
M.~Atiyah, R.~Bott, and V.~K. Patodi.
\newblock On the heat equation and the index theorem.
\newblock {\em Invent. Math.}, 19:279--330, 1973.

\bibitem{avez63}
A.~Avez.
\newblock Formule de {G}auss-{B}onnet-{C}hern en m\'{e}trique de signature
  quelconque.
\newblock {\em Rev. Un. Mat. Argentina}, 21:191--197 (1963), 1963.

\bibitem{bernig_broecker07}
A.~Bernig and L.~Br\"{o}cker.
\newblock Valuations on manifolds and {R}umin cohomology.
\newblock {\em J. Differential Geom.}, 75(3):433--457, 2007.

\bibitem{bernig_faifman16}
A.~Bernig and D.~Faifman.
\newblock Generalized translation invariant valuations and the polytope
  algebra.
\newblock {\em Adv. Math.}, 290:36--72, 2016.

\bibitem{bernig_faifman_opq}
A.~Bernig and D.~Faifman.
\newblock Valuation theory of indefinite orthogonal groups.
\newblock {\em J. Funct. Anal.}, 273(6):2167--2247, 2017.

\bibitem{part2}
A.~Bernig, D.~Faifman, and G.~Solanes.
\newblock Uniqueness of curvature measures in pseudo-{R}iemannian geometry.
\newblock {\em J. Geom. Anal.}, 31(12):11819--11848, 2021.

\bibitem{bernig_fu_hig}
A.~Bernig and J.~H.~G. Fu.
\newblock Hermitian integral geometry.
\newblock {\em Ann. of Math. (2)}, 173(2):907--945, 2011.

\bibitem{bernig_fu_solanes}
A.~Bernig, J.~H.~G. Fu, and G.~Solanes.
\newblock Integral geometry of complex space forms.
\newblock {\em Geom. Funct. Anal.}, 24(2):403--492, 2014.

\bibitem{bernig_fu_solanes_proceedings}
A.~Bernig, J.~H.~G. Fu, and G.~Solanes.
\newblock Dual curvature measures in {H}ermitian integral geometry.
\newblock In {\em Analytic aspects of convexity}, volume~25 of {\em Springer
  INdAM Ser.}, pages 1--17. Springer, Cham, 2018.

\bibitem{birman_nomizu}
G.~S. Birman and K.~Nomizu.
\newblock The {G}auss-{B}onnet theorem for {$2$}-dimensional spacetimes.
\newblock {\em Michigan Math. J.}, 31(1):77--81, 1984.

\bibitem{bott_tu}
R.~Bott and L.~W. Tu.
\newblock {\em Differential forms in algebraic topology}, volume~82 of {\em
  Graduate Texts in Mathematics}.
\newblock Springer-Verlag, New York-Berlin, 1982.

\bibitem{brouder_dang_helein}
C.~Brouder, N.~V. Dang, and F.~H\'{e}lein.
\newblock Continuity of the fundamental operations on distributions having a
  specified wave front set (with a counterexample by {S}emyon {A}lesker).
\newblock {\em Studia Math.}, 232(3):201--226, 2016.

\bibitem{chern44}
S.-s. Chern.
\newblock A simple intrinsic proof of the {G}auss-{B}onnet formula for closed
  {R}iemannian manifolds.
\newblock {\em Ann. of Math. (2)}, 45:747--752, 1944.

\bibitem{chern63}
S.-s. Chern.
\newblock Pseudo-{R}iemannian geometry and the {G}auss-{B}onnet formula.
\newblock {\em An. Acad. Brasil. Ci.}, 35:17--26, 1963.

\bibitem{clarke70}
C.~J.~S. Clarke.
\newblock On the global isometric embedding of pseudo-{R}iemannian manifolds.
\newblock {\em Proc. Roy. Soc. London Ser. A}, 314:417--428, 1970.

\bibitem{dabrowski_brouder}
Y.~Dabrowski and C.~Brouder.
\newblock Functional properties of {H}\"{o}rmander's space of distributions
  having a specified wavefront set.
\newblock {\em Comm. Math. Phys.}, 332(3):1345--1380, 2014.

\bibitem{donnelly}
H.~Donnelly.
\newblock Heat equation and the volume of tubes.
\newblock {\em Invent. Math.}, 29(3):239--243, 1975.

\bibitem{faifman_crofton}
D.~Faifman.
\newblock Crofton formulas and indefinite signature.
\newblock {\em Geom. Funct. Anal.}, 27(3):489--540, 2017.

\bibitem{faifman_contact}
D.~Faifman.
\newblock Contact integral geometry and the {H}eisenberg algebra.
\newblock {\em Geom. Topol.}, 23(6):3041--3110, 2019.

\bibitem{federer59}
H.~Federer.
\newblock Curvature measures.
\newblock {\em Trans. Amer. Math. Soc.}, 93:418--491, 1959.

\bibitem{fu_alesker_product}
J.~H.~G. Fu.
\newblock Intersection theory and the {A}lesker product.
\newblock {\em Indiana Univ. Math. J.}, 65(4):1347--1371, 2016.

\bibitem{fu_wannerer}
J.~H.~G. Fu and T.~Wannerer.
\newblock Riemannian curvature measures.
\newblock {\em Geom. Funct. Anal.}, 29(2):343--381, 2019.

\bibitem{gelfand_shilov}
I.~M. Gel'fand and G.~E. Shilov.
\newblock {\em Generalized functions. {V}ol. 1}.
\newblock AMS Chelsea Publishing, Providence, RI, 2016.
\newblock Properties and operations, Translated from the 1958 Russian original
  [ MR0097715] by Eugene Saletan, Reprint of the 1964 English translation [
  MR0166596].

\bibitem{gilkey_park15}
P.~Gilkey and J.~H. Park.
\newblock Analytic continuation, the {C}hern-{G}auss-{B}onnet theorem, and the
  {E}uler-{L}agrange equations in {L}ovelock theory for indefinite signature
  metrics.
\newblock {\em J. Geom. Phys.}, 88:88--93, 2015.

\bibitem{gray_book}
A.~Gray.
\newblock {\em Tubes}, volume 221 of {\em Progress in Mathematics}.
\newblock Birkh\"{a}user Verlag, Basel, second edition, 2004.
\newblock With a preface by Vicente Miquel.

\bibitem{hartle_hawking}
J.~B. Hartle and S.~W. Hawking.
\newblock Wave function of the universe.
\newblock {\em Phys. Rev. D (3)}, 28(12):2960--2975, 1983.

\bibitem{hitchin}
N.~J. Hitchin.
\newblock The moduli space of special {L}agrangian submanifolds.
\newblock {\em Ann. Scuola Norm. Sup. Pisa Cl. Sci. (4)}, 25(3-4):503--515
  (1998), 1997.
\newblock Dedicated to Ennio De Giorgi.

\bibitem{honda_saji_teramoto}
A.~Honda, K.~Saji, and K.~Teramoto.
\newblock Mixed type surfaces with bounded {G}aussian curvature in
  three-dimensional {L}orentzian manifolds.
\newblock {\em Adv. Math.}, 365:107036, 46, 2020.

\bibitem{hoermander_pde1}
L.~H\"{o}rmander.
\newblock {\em The analysis of linear partial differential operators. {I}},
  volume 256 of {\em Grundlehren der mathematischen Wissenschaften [Fundamental
  Principles of Mathematical Sciences]}.
\newblock Springer-Verlag, Berlin, second edition, 1990.
\newblock Distribution theory and Fourier analysis.

\bibitem{kobayashi}
O.~Kobayashi.
\newblock Ricci curvature of affine connections.
\newblock {\em Tohoku Math. J. (2)}, 60(3):357--364, 2008.

\bibitem{kriele_kossowski_a}
M.~Kossowski and M.~Kriele.
\newblock Transverse, type changing, pseudo-{R}iemannian metrics and the
  extendability of geodesics.
\newblock {\em Proc. Roy. Soc. London Ser. A}, 444(1921):297--306, 1994.

\bibitem{kossowski97}
M.~Kossowski and M.~Kriele.
\newblock The volume blow-up and characteristic classes for transverse,
  type-changing, pseudo-{R}iemannian metrics.
\newblock {\em Geom. Dedicata}, 64(1):1--16, 1997.

\bibitem{kriele_kossowski_b}
M.~Kriele and M.~Kossowski.
\newblock Pseudo-{R}iemannian metrics with signature type change.
\newblock In {\em Geometry and topology of submanifolds, {VII} ({L}euven,
  1994/{B}russels, 1994)}, pages 155--158. World Sci. Publ., River Edge, NJ,
  1995.

\bibitem{melrose93}
R.~B. Melrose.
\newblock {\em The {A}tiyah-{P}atodi-{S}inger index theorem}, volume~4 of {\em
  Research Notes in Mathematics}.
\newblock A K Peters, Ltd., Wellesley, MA, 1993.

\bibitem{Pelletier}
F.~Pelletier.
\newblock Pseudo m\'{e}triques g\'{e}n\'{e}riques et th\'{e}or\`eme de
  {G}auss-{B}onnet en dimension {$2$}.
\newblock In {\em Singularities and dynamical systems ({I}r\'{a}klion, 1983)},
  volume 103 of {\em North-Holland Math. Stud.}, pages 219--238. North-Holland,
  Amsterdam, 1985.

\bibitem{rumin94}
M.~Rumin.
\newblock Formes diff\'{e}rentielles sur les vari\'{e}t\'{e}s de contact.
\newblock {\em J. Differential Geom.}, 39(2):281--330, 1994.

\bibitem{sakharov}
A.~Sakharov.
\newblock Cosmological transitions with changes in the signature of the metric.
\newblock {\em Soviet Physics Uspekhi}, 34:409, 10 2007.

\bibitem{smolyaninov_narimanov}
I.~Smolyaninov and E.~Narimanov.
\newblock Metric signature transitions in optical metamaterials.
\newblock {\em Physical review letters}, 105:067402, 08 2010.

\bibitem{solanes}
G.~Solanes.
\newblock Contact measures in isotropic spaces.
\newblock {\em Adv. Math.}, 317:645--664, 2017.

\bibitem{solanes_wannerer}
G.~Solanes and T.~Wannerer.
\newblock Integral geometry of exceptional spheres.
\newblock {\em J. Differential Geom.}, 117(1):137--191, 2021.

\bibitem{steller}
M.~Steller.
\newblock A {G}auss-{B}onnet formula for metrics with varying signature.
\newblock {\em Z. Anal. Anwend.}, 25(2):143--162, 2006.

\bibitem{sternberg12}
S.~Sternberg.
\newblock {\em Curvature in mathematics and physics}.
\newblock Dover Publications, Inc., Mineola, NY, 2012.

\bibitem{totaro}
B.~Totaro.
\newblock The curvature of a {H}essian metric.
\newblock {\em Internat. J. Math.}, 15(4):369--391, 2004.

\bibitem{wannerer_angular}
T.~Wannerer.
\newblock Classification of angular curvature measures and a proof of the
  angularity conjecture.
\newblock {\em Amer. J. Math.}
\newblock to appear.

\bibitem{weyl_tubes}
H.~Weyl.
\newblock On the {V}olume of {T}ubes.
\newblock {\em Amer. J. Math.}, 61(2):461--472, 1939.

\bibitem{white_weinfurtner_visser}
A.~White, S.~Weinfurtner, and M.~Visser.
\newblock Signature change events: a challenge for quantum gravity?
\newblock {\em Classical Quantum Gravity}, 27(4):045007, 20, 2010.

\bibitem{willison09}
S.~Willison.
\newblock Lovelock gravity and {W}eyl's tube formula.
\newblock {\em Phys. Rev. D}, 80(6):064018, 7, 2009.

\bibitem{zaehle86}
M.~Z\"{a}hle.
\newblock Integral and current representation of {F}ederer's curvature
  measures.
\newblock {\em Arch. Math. (Basel)}, 46(6):557--567, 1986.

\end{thebibliography}

 \end{document}